\documentclass[a4paper,12pt,reqno]{amsart}  %%%%% When compiling with pdfLaTeX

\usepackage{amsmath}
\usepackage{tikz-cd}

\newif\iffastcompile
\fastcompilefalse

\iffastcompile
  \RenewDocumentEnvironment{tikzcd}{O{} +b}
    {\hbox{\fbox{\strut tikz-cd omitted in fast mode}}}{}
\fi

\usepackage{amssymb}
\usepackage{amsthm}
\usepackage{amsfonts}
\usepackage{mathtools}
\usepackage{comment}
\usepackage{mathrsfs}
\usepackage{pifont}
\usepackage{cite}
\usepackage{graphicx}
\usepackage[colorlinks]{hyperref}  %%%%%%%%%% pdfLaTeX ver
\usepackage{xcolor}
\hypersetup{
	bookmarksnumbered=true,
    colorlinks=true,
    citecolor=blue,
    linkcolor=purple,
    urlcolor=orange,
}
\usepackage{mleftright} % \mleft \mrightの使用
\usepackage{extpfeil} % use for \xtwoheadrightarrow

\usepackage{pgfplots}
\pgfplotsset{compat=1.18}

\usepackage[capitalize, nameinlink]{cleveref}
\everymath{\displaystyle}
\numberwithin{equation}{section}

\newcommand{\redtext}[1]{\textcolor{red}{#1}}
\newcommand{\memo}[1]{\redtext{[{#1}]}} %% \memo{#1}

\theoremstyle{plain}

\newtheorem{thm}{Theorem}[section]
\crefname{thm}{Theorem}{Theorems}

\newtheorem{lem}[thm]{Lemma}
\crefname{lem}{Lemma}{Lemmas}
\newtheorem{prop}[thm]{Proposition}
\crefname{prop}{Proposition}{Propositions}
\newtheorem{cor}[thm]{Corollary}
\crefname{cor}{Corollary}{Corollaries}

\newtheorem*{claim*}{Claim}
\newtheorem*{thm*}{Theorem}

\newtheorem{introthm}{Theorem}[section]

\newtheorem{introcor}[introthm]{Corollary}
\crefname{introthm}{Theorem}{Theorems}
\crefname{introprop}{Proposition}{Propositions}
\crefname{introcor}{Corollary}{Corollaries}

\theoremstyle{definition}

\newtheorem{dfn}[thm]{Definition}
\crefname{dfn}{Definition}{Definitions}

\newtheorem{eg}[thm]{Example}
\crefname{eg}{Example}{Examples}
\newtheorem*{Ack}{Acknowledgement}
\newtheorem*{NoCon}{Notation and Conventions}
\newtheorem*{Out}{Outline of this paper}

\newtheorem*{HisNote}{Historical Note}
\newtheorem*{AI}{AI Usage}

\crefname{fact}{Fact}{Facts}

\theoremstyle{remark}
\newtheorem{rem}[thm]{Remark}
\crefname{rem}{Remark}{Remarks}

\usepackage{enumitem} %% [label={(M\arabic*)}] の使用

\crefname{appendix}{Appendix}{Appendices}
\Crefname{appendix}{Appendix}{Appendices}

\DeclareMathOperator{\Aut}{Aut}

\DeclareMathOperator{\Ker}{Ker}

\DeclareMathOperator{\id}{id}

\DeclareMathOperator{\GL}{GL}

\DeclareMathOperator{\Image}{Im}

\DeclareMathOperator{\End}{End}

\DeclareMathOperator{\Eq}{Eq}

\DeclareMathOperator{\Fit}{Fit}

\newcommand\dual{\raise0.9ex\hbox{$\scriptscriptstyle\vee$}}

\newcommand{\catname}[1]{\mathsf{#1}}

\newcommand{\Grp}{\catname{Grp}}
\newcommand{\Ab}{\catname{Ab}}

\DeclareMathOperator{\Map}{Map} % set of maps
\DeclareMathOperator{\Inn}{Inn} % Inn group of quandles
\DeclareMathOperator{\Dis}{Dis} % Displacement group of quandles
\DeclareMathOperator{\Trans}{Trans} % transvection group of quandles
\DeclareMathOperator{\relInn}{Inn_{rel}} % relative inner automorphism group
\DeclareMathOperator{\relTrans}{Trans_{rel}} % relative transvection group
\DeclareMathOperator{\relAs}{Adj_{rel}} % relative automorphism group

\DeclareMathOperator{\ncl}{ncl} % nilpotent class
\DeclareMathOperator{\covlen}{covlen} % covering length
\DeclareMathOperator{\tncl}{tncl} % transvection nilpotent class

\newcommand{\Quandle}{\catname{Qnd}}% the cat of quandles
\DeclareMathOperator{\Conj}{Conj} % conjugacy quandle of a group
\DeclareMathOperator{\Adj}{Adj} % adjoint group of a quandle
\DeclareMathOperator{\As}{Adj} % adjoint (associated) group of a quandle
\DeclareMathOperator{\Red}{Red} % 
\DeclareMathOperator{\OpRed}{OpRed} % 

\DeclareMathOperator{\Alex}{Alex} % Alexander quandle

\newcommand{\qCong}{\sim}% A generic quandle congruence.
\newcommand{\nilpTrunc}[2]{#2^{\mathrm{nil}}_{#1}}% Nilpotent truncation homomorphism.
\newcommand{\covTrunc}[2]{#2^{\mathrm{cov}}_{#1}}% Covering-length truncation homomorphism.

\newcommand{\redCong}[1]{\sim_{#1, \mathrm{red}}}% Congruence defining the reduced quotient.
\newcommand{\redQuot}[1]{\Red(#1)}% Domain of the reduced factor.
\newcommand{\redMap}[1]{\pi_{#1}^{\mathrm{red}}}% Quotient map to the reduced quotient.
\newcommand{\redFac}[1]{#1^{\mathrm{red}}}% Reduced factor over the original target.
\newcommand{\opredDef}[1]{E_{#1}}% Normal subgroup defining the operator-reduced quotient.
\newcommand{\opredQuot}[1]{\OpRed(#1)}% Domain of the operator-reduced factor.
\newcommand{\opredMap}[1]{\pi_{#1}^{\mathrm{opred}}}% Quotient map to the operator-reduced quotient.
\newcommand{\opredFac}[1]{#1^{\mathrm{opred}}}% Operator-reduced factor over the original target.

\newcommand{\redOpGrp}{O}% Abelian operator group attached to a point.
\newcommand{\relLCS}{\Gamma}% Relative lower central series.
\newcommand{\grpLCS}{\Gamma}% Ordinary lower central series.
\newcommand{\relInnLCS}[2]{\Gamma_{#1}(#2)}% Relative lower central series of \relInn(f).
\newcommand{\relTransLCS}[2]{\gamma_{#1}(#2)}% Relative lower central series of \relTrans(f).
\providecommand{\Hyp}{\operatorname{Hyp}}

\providecommand{\relHypCenter}[2]{\Hyp_{#1}(#2)}
\providecommand{\hypCenter}[1]{\Hyp(#1)}
\providecommand{\hypMap}[1]{\pi_{#1}^{\mathrm{hyp}}}
\providecommand{\hypFac}[1]{#1^{\mathrm{hyp}}}

\newcommand{\ZZ}{\mathbb{Z}}

\author{Yuki Imamura}
\address[Y.Imamura]{Osaka Central Advanced Mathematical Institute, Osaka Metropolitan University, Osaka 558-8585, Japan}
\email{\href{mailto:u287972b@alumni.osaka-u.ac.jp}{u287972b@alumni.osaka-u.ac.jp}}

\author{Tomoki Yoshida}
\address[T.Yoshida]{Department~of~Mathematics, School~of~Science~and~Engineering, Waseda~University, Ohkubo~3-4-1, Shinjuku, Tokyo~169-8555, Japan}
\email{\href{mailto:tomoki_y@asagi.waseda.jp}{tomoki\_y@asagi.waseda.jp}}

\title{On Nilpotent and Hypocentral Quandle Homomorphisms}
\date{September 12, 2026}

\keywords{Quandles, Nilpotent homomorphisms, Covering homomorphisms, Relative inner automorphisms, Orthogonal factorization systems}
\subjclass[2020]{20N02 (primary), 57K12, 20F14, 20F18, 18A32 (secondary).}

\begin{document}

\begin{abstract}
We introduce a notion of nilpotency for surjective quandle homomorphisms, which relativizes the nilpotency of quandles recently defined by Darn\'e. 
A surjective homomorphism is called nilpotent if its relative inner automorphism group is nilpotent relative to the inner automorphism group of the source. 
We show that this condition can equally be described through the relative transvection group and the relative adjoint group, and that it is equivalent to being a finite composite of covering homomorphisms. 

Extending the relative lower central series to transfinite ordinals, we then introduce hypocentral homomorphisms and characterize them as the surjections that are right orthogonal to the strongly connected homomorphisms. 
As a consequence, the strongly connected and the hypocentral homomorphisms form an orthogonal factorization system for the surjections in the category of quandles. 

Finally, we also introduce reduced and operator-reduced homomorphisms, which relativize reduced quandles, and give sufficient conditions for such a homomorphism to be nilpotent.
\end{abstract}
%Specializing to the terminal quandle, we obtain new characterizations of nilpotent quandles.

\maketitle

\setcounter{tocdepth}{2} %table of contents subsection
\tableofcontents
\setcounter{section}{-1}

\section{Introduction}
\label{section: introduction}

Quandles, introduced independently by Joyce~\cite{joyce_1982_a_classifying_invariant_of_knots_the_knot_quandle} and Matveev~\cite{matveev_1982_distributive_groupoids_in_knot_theory}, are algebraic structures whose axioms are an abstraction of the Reidemeister moves for knot diagrams.
Besides providing invariants of knots, quandles are closely related to groups.
With a quandle $Q$, one associates several groups, such as the automorphism group $\Aut(Q)$, the inner automorphism group $\Inn(Q)$, the transvection group $\Trans(Q)$, and the adjoint group $\As(Q)$.
Conversely, every group $G$ gives rise to a quandle $\Conj(G)$, called the \emph{conjugacy quandle} of $G$, with underlying set $G$ and operation given by conjugation in $G$.
The functor $\Conj$ is faithful and admits a left adjoint, given by $Q\mapsto \Adj(Q)$.
Thus, quandles provide a broader setting in which groups can be studied through their conjugation structures.

Several of the groups associated with a quandle $Q$ reflect the structure of $Q$ through their natural action on $Q$. Accordingly, one way to study a quandle $Q$ is to study the properties of these groups, or those of their actions on $Q$. For example, a quandle $Q$ is called \emph{homogeneous} if the action $\Aut(Q)\curvearrowright Q$ is transitive, and \emph{connected} if the action $\Inn(Q)\curvearrowright Q$ is transitive. 
Likewise, the \emph{mediality} of a quandle $Q$ is equivalent to the commutativity of $\Trans(Q)$.

In his paper~\cite{darne_2026_nilpotent_quandles}, Darn\'e introduced a notion of nilpotency for quandles, calling a quandle $Q$ \emph{nilpotent} if its inner automorphism group $\Inn(Q)$ is nilpotent.
The resulting class of quandles had appeared earlier in the literature under the terminology of reductive quandles; the use of the word nilpotency is justified by the fact that these quandles behave in close parallel to nilpotent groups.
One manifestation of this parallel is their characterization as the quandles admitting a finite tower of quandle coverings.

Recall that a group $G$ is \emph{nilpotent} if its lower central series becomes trivial after finitely many steps, or equivalently, if $G$ can be obtained from the trivial group by a finite sequence of central group extensions.
The quandle-theoretic counterpart of a central extension is a covering homomorphism, introduced by Eisermann~\cite{eisermann_2014_quandle_coverings_and_their_galois_correspondence} in the context of the covering theory of quandles; coverings of conjugation quandles arising from group extensions correspond to central extensions of groups (see for instance \cite[Proposition 5.2]{preprint_yuki_tomoki_2026_relativization_of_symmetries_on_quandles}).
As an analogue of the above characterization of group nilpotency, Darn\'e showed in \cite[Theorem 2.13]{darne_2026_nilpotent_quandles} that a non-empty quandle $Q$ is nilpotent if and only if the canonical map from $Q$ to the one-element quandle is expressible as a finite composite of covering homomorphisms.

In the present paper, we extend this notion of nilpotency and develop its relative version for surjective homomorphisms of quandles.

In the previous work \cite{preprint_yuki_tomoki_2026_relativization_of_symmetries_on_quandles}, the authors defined the \emph{relative inner automorphism group}
\[ \relInn(f) = \Ker(f_*\colon \Inn(P) \twoheadrightarrow \Inn(Q)) \]
of a surjective quandle homomorphism $f\colon P \twoheadrightarrow Q$, where $f_*$ denotes the induced surjective group homomorphism.
There, a surjective homomorphism $f$ is called \emph{connected} if $\relInn(f)$ acts transitively on every fiber of $f$. 
This is a relativization of the connectedness of quandles, in the sense that a quandle $Q$ is connected if and only if the unique homomorphism $Q\to\{\ast\}$ into the terminal quandle is connected.

% Relativization is not merely a formal generalization.
% なんかいい感じのrelativizeする理由を書く？

Along this line, we define nilpotency for quandle homomorphisms.
It should be noted that the nilpotency of the relative inner automorphism group alone is not sufficient: we must also take into account the conjugation action of the whole inner automorphism group. We therefore make use of the notion of relative nilpotency for a normal subgroup; see \cref{appendix: relative nilpotency for groups} for its general account.

\subsection{Results}
\label{subsection: results in introduction}

We first fix notation. For a normal subgroup $N \trianglelefteq G$, the \emph{$G$-relative lower central series} of $N$ is defined inductively by $\relLCS_0(G,N)=N$ and $\relLCS_{i+1}(G,N)=[G,\relLCS_i(G,N)]$. If $\relLCS_c(G,N)=1$ for some $c\geq 0$, then $N$ is called \emph{$G$-nilpotent}.

Let $f \colon P\twoheadrightarrow Q$ be a surjective quandle homomorphism.
By the surjectivity, there is a unique surjective group homomorphism $f_*\colon \Inn(P)\twoheadrightarrow \Inn(Q)$ satisfying $f_*(s_x)=s_{f(x)}$ for all $x\in P$.
The \emph{relative inner automorphism group} $\relInn(f)$ is defined as $\relInn(f)=\Ker(f_*)$, which is a normal subgroup of $\Inn(P)$.
Then, we call $f$ \emph{nilpotent} if $\relInn(f)$ is nilpotent relative to its ambient group $\Inn(P)$, namely if $\relLCS_{c}(\Inn(P),\relInn(f))=1$ for some $c\geq0$. Thus, relative nilpotency depends not only on the abstract group $\relInn(f)$, but also on the conjugation action of $\Inn(P)$ on it.
For a non-empty quandle $P$ and the unique surjective homomorphism $P\twoheadrightarrow\{\ast\}$ into the terminal quandle, this recovers Darn\'e's nilpotency of $P$.

Our first result generalizes to this relative setting the characterization of nilpotent quandles as those obtained as iterated coverings of the terminal quandle.

\begin{introthm}[{\cref{theorem: nilpotent homomorphism and covering rigid tower}}]
\label[introthm]{introtheorem: nilpotent homomorphism and covering rigid tower}
    Let $f\colon P\twoheadrightarrow Q$ be a surjective quandle homomorphism and let $c\geq0$. Then the following conditions are equivalent.
    \begin{enumerate}
        \item\label{introitem:c-nilpotent} $f$ is $c$-nilpotent.
        \item\label{introitem:factorization_into_c_covering_and_1_rigid} $f$ admits a factorization
        \[
        \begin{tikzcd}
            P=P_{c}
            \arrow[two heads]{r}{p_c}
            &
            P_{c-1}
            \arrow[two heads]{r}{p_{c-1}}
            &
            \cdots
            \arrow[two heads]{r}
            &
            P_{1}
            \arrow[two heads]{r}{p_1}
            &
            P_{0}
            \arrow[two heads]{r}{h}
            &
            Q
        \end{tikzcd}
        \]
        in which the first $c$ homomorphisms $p_c,\dots,p_1$ are coverings and the last homomorphism $h$ is rigid (i.e.~$\relInn(h)=1$).
    \end{enumerate}
\end{introthm}

Since rigid homomorphisms are coverings, \cref{introtheorem: nilpotent homomorphism and covering rigid tower} says that the nilpotent homomorphisms are precisely the finite composites of coverings.

In \cite{preprint_yuki_tomoki_2026_relativization_of_symmetries_on_quandles}, the authors also introduced the \emph{relative transvection group} of $f$,
\[ \relTrans(f) = \left\langle s_xs_y^{-1} \mid x,y \in P \text{ such that }f(x)=f(y) \right\rangle, \]
which is a normal subgroup of $\Inn(P)$ as well. One may use this subgroup to define nilpotency, but the two notions turn out to coincide, because the inclusion $[\Inn(P),\relInn(f)] \subseteq \relTrans(f)$ holds (see \cref{lem:commutator_inn_rel_trans_rel}).
Furthermore, a third group-theoretic description of the nilpotency condition is deduced from the \emph{relative adjoint group} of $f$, which is defined as
\[ \relAs(f)=\Ker\bigl(\As(f)\colon\As(P)\twoheadrightarrow\As(Q)\bigr). \]
Using the fact that the surjective group homomorphism $\rho_P \colon \Adj(P) \twoheadrightarrow\Inn(P)$ is a central extension and sends $\relAs(f)$ onto $\relTrans(f)$, we prove that the relative nilpotency of $\relAs(f)$ is equivalent to that of $\relTrans(f)$.
 
Summarizing, we establish the following theorem.

\begin{introthm}[{\cref{theorem: main structure theorem for nilpotent homomorphisms}}]
\label[introthm]{introtheorem: main structure theorem for nilpotent homomorphisms}
    Let $f\colon P\twoheadrightarrow Q$ be a surjective quandle
    homomorphism. Then the following conditions are equivalent.
    \begin{enumerate}[label={(\roman*)}]
        \item\label{introitem: f is nilp} $f$ is nilpotent, that is, $\relInn(f)$ is $\Inn(P)$-nilpotent.
        \item\label{introitem: reltrans is nilp} $\relTrans(f)$ is $\Inn(P)$-nilpotent.
        \item\label{introitem: relas is nilp} $\relAs(f)$ is $\As(P)$-nilpotent.
        \item\label{introitem: composition of coverings} $f$ is a finite composition of covering homomorphisms.
    \end{enumerate}
\end{introthm}

Considering the case with $Q=\{\ast\}$, we also obtain equivalent descriptions of nilpotency for a non-empty quandle (\cref{corollary: characterization_of_nilpitent_quandles}). In particular, a non-empty quandle $P$ is nilpotent (i.e., $\Inn(P)$ is nilpotent) if and only if $\Trans(P)$ is $\Inn(P)$-nilpotent.
This is a novel equivalent condition for nilpotency, which is not included in \cite[Corollary~2.5 and Theorem~2.13]{darne_2026_nilpotent_quandles}.

It is worth mentioning that Bonatto and Stanovsk\'{y}~\cite{bonatto_stanovsky_2021_commutator_theory_for_racks_and_quandles} also study nilpotency for quandles from the universal-algebraic viewpoint. According to \cite[Theorem 1.2]{bonatto_stanovsky_2021_commutator_theory_for_racks_and_quandles}, the nilpotency of a quandle $P$ in their sense is equivalent to that of $\Trans(P)$. Hence, our notion of nilpotency is stronger than theirs, as noted by Darn\'{e} in the introduction of \cite{darne_2026_nilpotent_quandles}.

%% section 3

Next, we extend the relative lower central series to transfinite ordinals.
Even when it does not reach $1$ after finitely many steps, the series may still do so at an infinite stage; for groups, this weaker condition is classical and known as hypocentrality.
So, we call $f$ a \emph{hypocentral} homomorphism if the transfinite $\Inn(P)$-relative lower central series of $\relInn(f)$ reaches $1$ at some, possibly infinite, ordinal stage. Every nilpotent homomorphism is hypocentral, but the converse fails in general.
We obtain the following group-theoretic characterization of hypocentrality.

\begin{introthm}[\cref{theorem: characterizations of hypocentral homomorphisms}]
\label[introthm]{introtheorem: characterizations of hypocentral homomorphisms}
Let $f\colon P\twoheadrightarrow Q$ be a surjective quandle homomorphism. Then, the following conditions are equivalent.
\begin{enumerate}[label={(\roman*)}]
    \item\label{introitem: f is hypocentral}
    The homomorphism $f$ is hypocentral, that is, $\relInn(f)$ is $\Inn(P)$-hypocentral.
    \item\label{introitem: relative transvection is hypocentral}
    The group $\relTrans(f)$ is $\Inn(P)$-hypocentral.
    \item\label{introitem: relative adjoint group is hypocentral}
    The group $\relAs(f)$ is $\As(P)$-hypocentral.
    \item\label{introitem: no perfect subgroup in relative inner group}
    The group $\relInn(f)$ contains no non-trivial $\Inn(P)$-perfect
    normal subgroup of $\Inn(P)$.
\end{enumerate}
\end{introthm}

Passing from nilpotency to hypocentrality clarifies the relationship with the notion of strongly connected homomorphisms.
In \cite[Definition 3.6]{preprint_yuki_tomoki_2026_relativization_of_symmetries_on_quandles}, a surjective homomorphism $f\colon P\twoheadrightarrow Q$ is called \emph{strongly connected} if the relative transvection group $\relTrans(f)$ acts transitively on every fiber of $f$.
The authors observed in \cite[Proposition 3.21]{preprint_yuki_tomoki_2026_relativization_of_symmetries_on_quandles} that the class of strongly connected homomorphisms is left orthogonal to the class of pre-covering homomorphisms. Here, a \emph{pre-covering} homomorphism is a not necessarily surjective homomorphism $f$ such that $f(x)=f(y)$ implies $s_x=s_y$.
We extend this observation by showing that hypocentral homomorphisms admit a category-theoretic characterization as the class of morphisms right orthogonal to strongly connected homomorphisms.

\begin{introthm}[\cref{theorem: transfinite covering characterization of hypocentral homomorphisms}]
\label[introthm]{introtheorem: transfinite covering characterization of hypocentral homomorphisms}
    Let $f\colon P\twoheadrightarrow Q$ be a surjective quandle homomorphism. Then, the following conditions are equivalent.
\begin{enumerate}[label={(\roman*)}]
    \item\label{introitem: homomorphism is hypocentral}
    $f$ is hypocentral.
    \item\label{introitem: transfinite cocomposite of covering condition}
    $f$ is a transfinite cocomposite of pre-covering homomorphisms.
    \item\label{introitem: hypocentral right orthogonal}
    $f$ is right orthogonal to every strongly connected homomorphism.
\end{enumerate}
\end{introthm}
The use of pre-covering homomorphisms is essential: a hypocentral homomorphism need not be a transfinite cocomposite of coverings (see \cref{example: hypocentral but not nilpotent Alexander quandle}).

Even and Gran~\cite[Proposition 3.2]{even_gran_2014_on_factorization_systems_for_surjective_quandle_homomorphisms} previously showed that the classes of connected and rigid homomorphisms form an orthogonal factorization system relative to surjective homomorphisms (see also \cite[Proposition 2.24]{preprint_yuki_tomoki_2026_relativization_of_symmetries_on_quandles}). 
We establish the analogous statement for strongly connected homomorphisms.

\begin{introthm}[Strongly connected--hypocentral factorization system; \cref{theorem: universal strongly connected-hypocentral factorization,theorem: strongly connected-hypocentral factorization system}]
\label[introthm]{introtheorem: strongly connected-hypocentral factorization system}
    Every surjective quandle homomorphism $f$ admits a factorization $f=\hypFac{f}\circ \hypMap{f}$ into a strongly connected homomorphism $\hypMap{f}$ followed by a hypocentral homomorphism $\hypFac{f}$.
    Thus, the pair $(\{\mathrm{strongly\>connected}\}, \{\mathrm{hypocentral}\})$ of classes of surjections is an orthogonal factorization system for the surjections in $\Quandle$.
\end{introthm}

From \cite[Proposition 3.32]{preprint_yuki_tomoki_2026_relativization_of_symmetries_on_quandles}, strongly connected homomorphisms are closed under pullbacks along surjective homomorphisms. Therefore, this factorization system for the surjections is in fact \textit{stable} in the sense of \cite[\href{https://ncatlab.org/nlab/show/stable+factorization+system}{stable factorization system}]{preprint_nlab_2008_nlab}.

%We further show that the hypocentral factor $\hypFac{f}$ is nilpotent precisely when the relative transvection series of $f$ stabilizes at a finite stage. In particular, hypocentrality and nilpotency agree for every surjection between quandles with finite inner automorphism groups, so that the classes of strongly connected and nilpotent homomorphisms already form a stable orthogonal factorization system on this subcategory, for instance on that of finite quandles. This is no longer true in general: for the Alexander quandle associated with $-\id$ on $\mathbb{Z}$, the terminal homomorphism is hypocentral but not nilpotent.

%%% section 4

In the third part of the paper, we introduce relative analogues of reduced quandles. Reduced quandles, also called quasi-trivial quandles in part of the literature, arise in the study of link-homotopy \cite{hughes_2011_link_homotopy_invariant_quandles,inoue_2013_quasitriviality_of_quandles_for_linkhomotopy,elhamdadi_liu_nelson_2018_quasitrivial_quandles_and_biquandles_cocycle_enhancements_and_linkhomotopy_of_pretzel_links} and are studied systematically in \cite[Section~7]{darne_2026_nilpotent_quandles}.
We study their relativization, called reduced homomorphisms, together with the slightly more general operator-reduced homomorphisms; these two notions lie between the first and the second nilpotency.

A surjective homomorphism $f$ is called \emph{reduced} if $s_{\varphi(x)}(x)=x$, and \emph{operator-reduced} if $s_xs_{\varphi(x)}=s_{\varphi(x)}s_x$, for every $x\in P$ and every $\varphi\in\relInn(f)$.
As we will see in \cref{proposition: commutator interpretation of reducedness}, $1$-nilpotent implies reduced, reduced implies operator-reduced, and $2$-nilpotent implies operator-reduced.
We then investigate when an operator-reduced homomorphism is nilpotent, and establish the following ascent theorem.

\begin{introthm}[Nilpotency ascent theorem;~\cref{theorem: finite operator reduced ascent}]
\label[introthm]{introtheorem: finite operator reduced ascent}
    Let $f\colon P\twoheadrightarrow Q$ be an operator-reduced quandle homomorphism such that $\relInn(f)$ is a finite group. If $Q$ is nilpotent, then both $f$ and $P$ are nilpotent.
\end{introthm}
In Darn\'e's absolute setting \cite[Proposition 7.13]{darne_2026_nilpotent_quandles}, the target is the one-point quandle and hence automatically nilpotent. 
The relative viewpoint reveals the importance of this condition: even a connected reduced homomorphism with trivial fibers and a finite relative inner automorphism group need not be nilpotent (see \cref{example: finite connected reduced homomorphism not nilpotent}).

Under some technical assumptions on the commutativity of the conjugation action on abelian sections of $\relInn(f)$, we also prove another criterion for the nilpotency of an operator-reduced homomorphism (\cref{theorem: commutative action criterion}).
In contrast with \cref{introtheorem: finite operator reduced ascent}, it requires no finiteness assumption on $\relInn(f)$.
Applying this criterion, we obtain the following corollary, which covers a number of natural situations.

\begin{introcor}[\cref{lemma: abelian images of inner automorphism groups,corollary: easy consequences of comm action criterion}]
\label[introcor]{introcorollary: easy consequences of comm action criterion}
    Let $f\colon P\twoheadrightarrow Q$ be an operator-reduced quandle homomorphism. 
    Suppose that $P$ satisfies either the condition that $\Inn(P)$ is generated by finitely many symmetry maps or the condition that $P$ has finitely many connected components.
    If $\relInn(f)$ is nilpotent and $\Inn(Q)$ is abelian, then $f$ is nilpotent.
\end{introcor}

%The nilpotency of the target is essential in \cref{introtheorem: finite operator reduced ascent}. We exhibit a finite connected reduced homomorphism with trivial fibers and abelian relative inner automorphism group which is nevertheless not nilpotent. Hence, neither the triviality of the fibers nor the abstract nilpotency of $\relInn(f)$ controls relative nilpotency, as long as no condition is imposed on the action coming from the base.

\begin{Out}
\cref{section: preliminaries on quandles} collects the preliminaries on quandles used throughout the paper. We also recall the relative viewpoint on surjective quandle homomorphisms introduced in \cite{preprint_yuki_tomoki_2026_relativization_of_symmetries_on_quandles}.
 
In \cref{section: nilpotent homomorphisms}, we introduce the notion of nilpotency for surjective quandle homomorphisms and develop its basic theory. After recalling Darn\'e's nilpotent quandles, we define nilpotent homomorphisms through the relative lower central series of $\relInn(f)$ and prove \cref{introtheorem: nilpotent homomorphism and covering rigid tower}.
%, together with the behaviour of nilpotency under composition, products, pullbacks and passage to fibers, as well as the construction of the universal $c$-nilpotent truncation.
We then give the descriptions of nilpotency in terms of the relative transvection group and of the relative adjoint group, which yield \cref{introtheorem: main structure theorem for nilpotent homomorphisms}.
%; along the way, the relative transvection group is shown to compute the covering length.
%The section closes with a comparison of the three resulting numerical invariants and with explicit computations for homomorphisms of Alexander quandles, which show that the bounds obtained are sharp.
 
\cref{section: hypocentral homomorphisms} extends the relative lower central series to transfinite ordinals and introduces hypocentral homomorphisms. We first establish their basic properties and prove \cref{introtheorem: characterizations of hypocentral homomorphisms}. We then construct the strongly connected--hypocentral factorization and prove \cref{introtheorem: transfinite covering characterization of hypocentral homomorphisms,introtheorem: strongly connected-hypocentral factorization system}.
%In the last part, we determine when the hypocentral factor is nilpotent; in particular, the two classes coincide under a descending chain condition, which applies to all finite quandles, whereas an Alexander quandle over $\mathbb{Z}$ provides a hypocentral homomorphism that is not nilpotent.
 
\cref{section: relative reduced homomorphisms} is devoted to reduced and operator-reduced homomorphisms. We first define these two classes and locate them between $1$-nilpotency and $2$-nilpotency.
%, and establish their closure properties under composition, products and pullbacks.
We next construct the universal reduced and operator-reduced factors of a surjective homomorphism. 
Finally, we prove \cref{introtheorem: finite operator reduced ascent} and the commutative action criterion (\cref{theorem: commutative action criterion}), from which \cref{introcorollary: easy consequences of comm action criterion} is deduced, and we conclude with the case of a cyclic relative inner automorphism group.
 
\cref{appendix: relative nilpotency for groups} collects the purely group-theoretic material on the relative nilpotency of a normal subgroup, or equivalently on nilpotent surjections of groups.
\end{Out}

\begin{NoCon}
Groups act on the left. For elements $x,y$ of a group, we use the commutator convention $[x,y]=xyx^{-1}y^{-1}$. For subgroups $A$ and $B$, the notation $[A,B]$ denotes the subgroup generated by the commutators $[a,b]$ with $a\in A$ and $b\in B$.
\end{NoCon}

\begin{AI}
The authors used GPT-5.6 Sol and GPT-6 Astra (OpenAI), Gemini 3.1 Pro (Google), and Claude Opus 5 (Anthropic) throughout the preparation of this work, particularly for the proof of \cref{theorem: finite operator reduced ascent}, the construction of \cref{example: hypocentral but not nilpotent Alexander quandle}, and for English proofreading. 
All arguments and proofs were independently verified by the authors, who take full responsibility for the contents of this paper.
\end{AI}

\begin{Ack}
This work was supported by JSPS KAKENHI Grant Number JP25K23333 and JP26KJ0276.
The first author acknowledges access to Claude through the Claude Team plan for scientists and the second author acknowledges support from OpenAI through the ChatGPT for Academic Researchers program.
\end{Ack}

\section{Preliminaries on quandles}
\label{section: preliminaries on quandles}

In this section, we recall the basic definitions and fix the examples and notation used later. The reader may consult \cite{joyce_1982_a_classifying_invariant_of_knots_the_knot_quandle,
eisermann_2014_quandle_coverings_and_their_galois_correspondence,
preprint_valeriy_mohamed_mahender_2025_yangbaxter_equation_and_related_algebraic_structures} for more background.

\subsection{Quandles and fundamental notions}
\label{subsection: quandles and standard examples}

\begin{dfn}
\label[dfn]{definition: quandle}
    A \emph{quandle} $(Q,s)$ is a pair consisting of a set $Q$ and a map $s\colon Q\to\Map(Q,Q)$, $x\mapsto s_x$, satisfying the following conditions for all $x,y\in Q$:
\begin{enumerate}
    \item\label{condition:Q1} $s_x(x)=x$;
    \item\label{condition:Q2} $s_x$ is bijective;
    \item\label{condition:Q3} $s_x\circ s_y=s_{s_x(y)}\circ s_x$.
\end{enumerate}
\end{dfn}

\begin{dfn}
\label[dfn]{definition: quandle homomorphism}
    Let $(Q,s)$ and $(Q',s')$ be quandles. A map $f\colon Q\to Q'$ is a \emph{quandle homomorphism} if $f\circ s_x=s'_{f(x)}\circ f$ for every $x\in Q$. The category of quandles and quandle homomorphisms is denoted by $\Quandle$.
\end{dfn}

The condition~\eqref{condition:Q3} says each symmetry map $s_x$ is a quandle homomorphism $Q\to Q$.

\begin{eg}[Trivial quandles]
\label[eg]{example: trivial quandle}
    Every set $X$ has a quandle structure given by $s_x=\id_X$ for all $x\in X$. We call it the \emph{trivial quandle} structure on $X$.
\end{eg}

\begin{dfn}
\label[dfn]{definition: inner and transvection groups}
    Let $Q$ be a quandle.
\begin{enumerate}
    \item The group of quandle automorphisms of $Q$ is denoted by $\Aut(Q)$.
    \item The \emph{inner automorphism group} of $Q$ is $\Inn(Q)=\langle s_x\mid x\in Q\rangle\subseteq\Aut(Q)$.
    \item The \emph{transvection group} of $Q$ is $\Trans(Q) = \langle s_xs_y^{-1}\mid x,y\in Q\rangle$.
    \footnote{It is also called the \emph{displacement group} in some literature.}
    \item The quandle $Q$ is called \emph{connected} if the natural action of $\Inn(Q)$ on $Q$ is transitive. The $\Inn(Q)$-orbits are called the connected components of $Q$, and their set is denoted by $\pi_0(Q)$.
\end{enumerate}
\end{dfn}

For later quantitative statements, we recall finite generation. A subset $R\subseteq Q$ of a quandle $Q$ is called a \emph{subquandle} when $R$ is closed under both $s_z$ and $s_z^{-1}$ for every $z\in R$. A subset $S\subseteq Q$ \emph{generates} $Q$ if $Q$ is the smallest subquandle containing $S$. The quandle $Q$ is \emph{finitely generated} if it admits a finite generating set.

\begin{lem}
\label[lem]{lemma: quandle generators generate inner group}
    Let $Q$ be a quandle. If $Q$ is generated by $x_1,\ldots,x_r$ as a quandle, then $\Inn(Q)$ is generated by $s_{x_1},\ldots,s_{x_r}$.
\end{lem}

\begin{proof}
Let $H=\langle s_{x_1},\ldots,s_{x_r}\rangle \subseteq \Inn(Q)$. Since $s_{\varphi(x_i)}=\varphi s_{x_i}\varphi^{-1}\in H$ for $\varphi\in H$, the subset $H\cdot\{x_1,\ldots,x_r\}$ is closed under the quandle operation. So it is a subquandle of $Q$, and hence we have $Q=H\cdot\{x_1,\ldots,x_r\}$ by the assumption. Then, for every $x\in Q$, if we take an index $i$ and $\varphi\in H$ with $\varphi(x_i)=x$, it follows that $s_x=s_{\varphi(x_i)} \varphi s_{x_i} \varphi^{-1}\in H$, which shows that $\Inn(Q)=H$.
\end{proof}

We use the notation $\As(Q)$ for the adjoint group of a quandle. The following definition and properties are standard; see \cite[Sections~2 and 5]
{eisermann_2014_quandle_coverings_and_their_galois_correspondence}.

\begin{dfn}
\label[dfn]{definition: adjoint group}
    Let $Q$ be a quandle. The \emph{adjoint group} $\As(Q)$ is the group generated by $\{g_{x}\}_{x\in Q}$ subject to the relations $g_{x}g_{y}=g_{s_x(y)}g_{x}$, for $x,y\in Q$.
\end{dfn}

\begin{prop}
\label[prop]{proposition: elementary properties of adjoint groups}
Let $Q$ be a quandle.
\begin{enumerate}
    \item The assignment $g_{x}\mapsto s_x$ induces a surjective group homomorphism $\rho_{Q}\colon\As(Q)\twoheadrightarrow\Inn(Q)$, and $\As(Q)$ acts on $Q$ via $\rho_Q$.
    \item Every quandle homomorphism $f\colon P\to Q$ induces a group homomorphism $\As(f)\colon\As(P)\to\As(Q)$ satisfying $\As(f)(g_{x})=g_{f(x)}$. If $f$ is surjective, then so is $\As(f)$.
    \item The homomorphism $\rho_{Q}\colon\As(Q)\twoheadrightarrow\Inn(Q)$ is a central extension of groups.
\end{enumerate}
\end{prop}

There is a natural degree homomorphism $\varepsilon_{Q}\colon\As(Q)\to\mathbb Z$ defined by $\varepsilon_{Q}(g_{x})=1$; see \cite[Definition~2.33]{eisermann_2014_quandle_coverings_and_their_galois_correspondence}. We put $\As_0(Q)\coloneqq \Ker(\varepsilon_{Q})$. If $Q$ is connected, the group $\As_0(Q)$ agrees with the commutator subgroup of $\As(Q)$ \cite[Remark~2.34]{eisermann_2014_quandle_coverings_and_their_galois_correspondence}; this equality need not hold for a disconnected quandle.

\subsection{Surjective homomorphisms and relative symmetries}
\label{subsection: quandle homomorphisms}

In this subsection, we recall the results on surjective quandle homomorphisms used throughout the paper. Our main sources are Eisermann's covering theory \cite{eisermann_2014_quandle_coverings_and_their_galois_correspondence}, the quotient and factorization results of \cite{bunch_lofgren_rapp_yetter_2010_on_quotients_of_quandles,even_gran_2014_on_factorization_systems_for_surjective_quandle_homomorphisms}, and the relative symmetry groups introduced in \cite{preprint_yuki_tomoki_2026_relativization_of_symmetries_on_quandles}.

We shall use quotients by quandle congruences and by group actions.

\begin{dfn}
\label[dfn]{definition: quandle congruence}
An equivalence relation $\qCong$ on a quandle $Q$ is a \emph{quandle
congruence} if
\[
    x\mathrel\qCong x',\quad y\mathrel\qCong y'
    \quad\Longrightarrow\quad
    s_x^{\varepsilon}(y)\mathrel\qCong
    s_{x'}^{\varepsilon}(y')
\]
for $\varepsilon\in\{1,-1\}$.
\end{dfn}

\begin{prop}
%[{\cite[Propositions~1.22 and 1.23]{preprint_yuki_tomoki_2026_relativization_of_symmetries_on_quandles}}]
\label[prop]{proposition: quotients by quandle congruences}
    Let $\qCong$ be a quandle congruence on $Q$. Then, the quotient set $Q/{\qCong}$ has a unique quandle structure satisfying
    \[ s_{[x]}([y])=[s_x(y)], \]
    and the quotient map $Q\twoheadrightarrow Q/{\qCong}$ is a surjective quandle homomorphism. If $f\colon Q\twoheadrightarrow R$ is surjective, then its kernel relation
    \[ \Eq(f) = \{(x,y)\in Q\times Q\mid f(x)=f(y)\} \]
    is a quandle congruence. These correspondences are inverse to each other.
\end{prop}

\begin{prop}[{see \cite[Proposition~1.25 and Remark~1.26]{preprint_yuki_tomoki_2026_relativization_of_symmetries_on_quandles}}]
%\cite[Lemma~1.20]{darne_2026_nilpotent_quandles}; 
\label[prop]{proposition: orbit quotients of quandles}
    Let $A\subseteq\Aut(Q)$ be a subgroup normalized by $\Inn(Q)$. Then the relation given by the $A$-orbits is a quandle congruence. We denote the quotient quandle by $Q/A$. In particular, this construction applies to every normal subgroup $A\trianglelefteq\Inn(Q)$.
\end{prop}

The assignment $Q\mapsto\Inn(Q)$ is not functorial for arbitrary quandle homomorphisms, but it is functorial for surjections.

\begin{prop}[{\cite[Proposition~2.26]{eisermann_2014_quandle_coverings_and_their_galois_correspondence}}]
\label[prop]{proposition: functoriality of Inn for surjections}
    Let $f\colon P\twoheadrightarrow Q$ be a surjective quandle homomorphism. There is a unique surjective group homomorphism
    \[ f_*\colon\Inn(P)\twoheadrightarrow\Inn(Q) \]
    satisfying $f_*(s_x)=s_{f(x)}$ for every $x\in P$.
    Moreover, this gives rise to a functor $\Inn(-) \colon \Quandle^\mathrm{surj} \to \Grp^\mathrm{surj}$ from the category of surjective quandle homomorphisms to the category of surjective group homomorphisms.
\end{prop}

We next recall coverings in the sense of Eisermann.

\begin{dfn}[{\cite[Definition~2.42]{eisermann_2014_quandle_coverings_and_their_galois_correspondence}}]
\label[dfn]{definition: quandle covering}
    A quandle homomorphism $f\colon P\to Q$ is a \emph{covering homomorphism} if it is surjective and $f(x)=f(y)$ implies $s_x=s_y$ for all $x,y\in P$.
\end{dfn}

\begin{prop}[{\cite[Proposition~4.19]{eisermann_2014_quandle_coverings_and_their_galois_correspondence}}]
\label[prop]{proposition: pull-back of covering is again covering}
    Covering homomorphisms are preserved by pullback along arbitrary quandle homomorphisms.
\end{prop}

\begin{prop}[{\cite[Proposition~2.49]{eisermann_2014_quandle_coverings_and_their_galois_correspondence}}]
\label[prop]{proposition: covering hom is central extension}
    If $f\colon P\twoheadrightarrow Q$ is a covering homomorphism, then $f_*\colon\Inn(P)\twoheadrightarrow\Inn(Q)$ is a central extension of groups.
\end{prop}

We now recall the relative symmetry groups.

\begin{dfn}[{\cite[Definitions~2.1 and 3.1]{preprint_yuki_tomoki_2026_relativization_of_symmetries_on_quandles}}]
\label[dfn]{definition: relative inner and transvection groups}
    Let $f\colon P\twoheadrightarrow Q$ be a surjective quandle homomorphism.
\begin{enumerate}
    \item The \emph{relative inner automorphism group} of $f$ is
    \[ \relInn(f) = \Ker\bigl( f_*\colon \Inn(P)\to\Inn(Q) \bigr). \]
    \item The \emph{relative transvection group} of $f$ is
    \[ \relTrans(f) = \left\langle s_xs_y^{-1} \mid x,y\in P \text{ such that } f(x)=f(y) \right\rangle. \]
\end{enumerate}
\end{dfn}

\begin{prop}[{\cite[Propositions~2.4 and~3.4, and Lemma~3.15]{preprint_yuki_tomoki_2026_relativization_of_symmetries_on_quandles}}]
\label[prop]{lemma: elementary properties of relative symmetry groups}
    For every surjective homomorphism $f\colon P\twoheadrightarrow Q$, the following hold.
    \begin{enumerate}
        \item The group $\relInn(f)$ is normal in $\Inn(P)$ and $\relInn(f)=\{\varphi\in \Inn(P) \mid f\circ \varphi=f \}$. Every inner automorphism $\varphi\in \relInn(f)$ preserves the fibers of $f$, and hence $\relInn(f)$ acts on each fiber.
        \item The group $\relTrans(f)$ is normal in $\Inn(P)$ and $\relTrans(f)\subseteq\relInn(f)$. Moreover, $f$ is covering if and only if $\relTrans(f)=1$.
    \end{enumerate}
\end{prop}

\begin{rem}
\label[rem]{remark: relative transvection as relative displacement group}
    The group $\relTrans(f)$ is described in \cite{bonatto_stanovsky_2021_commutator_theory_for_racks_and_quandles} as the relative displacement group $\Dis_\alpha$ where $\alpha$ is the congruence corresponding to the surjection $f$.
\end{rem}

\begin{prop}[{\cite[Proposition~2.15]{preprint_yuki_tomoki_2026_relativization_of_symmetries_on_quandles}}]
\label{prop:basic_property_of_quotient_by_normal_subgroup}
Let $Q$ be a quandle and let $N\subseteq \Inn(Q)$ be a normal subgroup.
\begin{enumerate}
    \item\label{item:N_is_contained_in_relInn_of_projection}
    For the quotient map $\pi_N\colon Q\twoheadrightarrow Q/N$, we have $N \subseteq \relInn(\pi_N)=\Ker(\Inn(\pi_N))$.
    \item\label{item:factoring_condition_through_quotient_by_normal_subgroup}
    A quandle homomorphism $f\colon Q\to R$ factors through $\pi_N\colon Q\twoheadrightarrow Q/N$ if and only if $f\circ \varphi = f$ holds for all $\varphi\in N$, or equivalently if and only if $N\subseteq \relInn(f)$.
\end{enumerate}
\end{prop}

The following lemma will be used repeatedly.

\begin{lem}[{\cite[Lemma 3.8]{preprint_yuki_tomoki_2026_relativization_of_symmetries_on_quandles}}]
\label[lem]{lemma: relative symmetry groups under a surjective factorization}
    Suppose that $f=h\circ u$, where $u\colon P\twoheadrightarrow R$ and $h\colon R\twoheadrightarrow Q$ are surjective. Then
    \[ u_*(\relInn(f))=\relInn(h), \qquad u_*(\relTrans(f))=\relTrans(h).\]
\end{lem}

\begin{dfn}
\label[dfn]{definition: connected and rigid quandle homomorphisms}
    Let $f\colon P\twoheadrightarrow Q$ be a surjective quandle homomorphism.
\begin{enumerate}
    \item $f$ is called \emph{connected} if
    $\relInn(f)$ acts transitively on every fiber of $f$ (\cite[Definition~2.5]{preprint_yuki_tomoki_2026_relativization_of_symmetries_on_quandles}).
    \item $f$ is called \emph{strongly connected} if $\relTrans(f)$ acts transitively on every fiber of $f$ (\cite[Definition~3.6]{preprint_yuki_tomoki_2026_relativization_of_symmetries_on_quandles}).
    \item The homomorphism $f$ is called \emph{rigid} if $f_*\colon\Inn(P)\to\Inn(Q)$ is an isomorphism, or equivalently, if $\relInn(f)=1$ (\cite[Definition~4.1]{bunch_lofgren_rapp_yetter_2010_on_quotients_of_quandles}).
\end{enumerate}
\end{dfn}

We can see that a rigid homomorphism is covering, and that a connected rigid homomorphism is an isomorphism \cite[Lemmas~2.19 and~2.20]{preprint_yuki_tomoki_2026_relativization_of_symmetries_on_quandles}.
%Connectedness descends along surjective factorizations \cite[Proposition~2.12(3) and Remark~2.13]{preprint_yuki_tomoki_2026_relativization_of_symmetries_on_quandles}.

The following factorization is essentially \cite[Theorem~8.1]{bunch_lofgren_rapp_yetter_2010_on_quotients_of_quandles}, and its factorization-system formulation is due to \cite[Propositions~3.1 and~3.2]{even_gran_2014_on_factorization_systems_for_surjective_quandle_homomorphisms}.

\begin{thm}[Quandle Stein factorization; {\cite[Theorem~2.22 and Proposition~2.24]{preprint_yuki_tomoki_2026_relativization_of_symmetries_on_quandles}}]
\label[thm]{theorem: connected-rigid factorization}
    Let $f\colon P\twoheadrightarrow Q$ be a surjective quandle homomorphism. Then, the canonical factorization $P\twoheadrightarrow P/\relInn(f)\twoheadrightarrow Q$ has connected first homomorphism and rigid second homomorphism. 
    Moreover, the classes of connected and rigid homomorphisms form an orthogonal factorization system for the surjections.
    % Remark~2.25: satisfying the Frobenius condition.
\end{thm}

In particular, rigid homomorphisms are preserved by pullback along surjective homomorphisms \cite[Proposition~2.21(2)]{preprint_yuki_tomoki_2026_relativization_of_symmetries_on_quandles}.

The relative transvection group gives the maximal covering factorization.

\begin{thm}[{
\cite[Propositions~3.17 and~3.18]{preprint_yuki_tomoki_2026_relativization_of_symmetries_on_quandles}}]
\label[thm]{theorem: maximal covering factorization}
    Let $f\colon P\twoheadrightarrow Q$ be a surjective quandle homomorphism. Then, the canonical factorization $P\twoheadrightarrow P/\relTrans(f)\twoheadrightarrow Q$ has connected first homomorphism and covering second homomorphism. It is universal among factorizations of $f$ whose second homomorphism is covering.
\end{thm}

Recall that a normal subgroup $K\trianglelefteq G$ is called \emph{$G$-perfect} if $[G,K]=K$. 

\begin{thm}[{\cite[Theorem 3.13 and Corollary 3.14]{preprint_yuki_tomoki_2026_relativization_of_symmetries_on_quandles}}]
\label{theorem: characterization of strongly connected via perfect subgroup}
    A surjective quandle homomorphism $f\colon P\twoheadrightarrow Q$ is strongly connected if and only if it is isomorphic to the quotient $\pi_N\colon P\twoheadrightarrow P/N$ by an $\Inn(P)$-perfect subgroup $N\subseteq \Inn(P)$.

    Moreover, there is a one-to-one correspondence between the isomorphism classes of strongly connected quotients of $P$ and $\Inn(P)$-perfect normal subgroups of $\Inn(P)$.
\end{thm}

\begin{cor}
\label{proposition:relTrans_of_strongly_conncted_is_Inn-perfect}
    For a strongly connected homomorphism $f\colon P\twoheadrightarrow Q$, the group $\relTrans(f)$ is $\Inn(P)$-perfect.
\end{cor}

We next recall the standard examples used in this paper.

\begin{eg}[Conjugacy quandles]
\label[eg]{example: conjugacy quandle}
    Let $G$ be a group. Its \emph{conjugacy quandle} $\Conj(G)$ is the set $G$ equipped with $s_x(y)=xyx^{-1}$ for $x,y\in G$. A surjective group homomorphism $\phi$ induces a surjective quandle homomorphism $f\coloneqq \Conj(\phi)\colon\Conj(G)\to\Conj(H)$. Write $K\coloneqq \Ker(\phi)$ and $A\coloneqq \phi^{-1}(Z(H))$. Then, we have
    \begin{align*}
        \relInn(f) &= \{ i_g \mid g \in G \text{ such that } \phi(g) \in Z(H)\} \cong A/Z(G), \text{ and} \\
        \relTrans(f) &= \{i_k \mid k\in \Ker(\phi)\} \cong K /(K\cap Z(G)),
    \end{align*}
    where $i_g$ is given by $i_g(x)=gxg^{-1}$ (see {\cite[Proposition~5.1]{preprint_yuki_tomoki_2026_relativization_of_symmetries_on_quandles}}).
\end{eg}

\begin{eg}[Alexander quandles]
\label[eg]{example: Alexander quandle}
    Let $A$ be an abelian group and let $\sigma\in\Aut(A)$ be a group automorphism of $A$. The \emph{Alexander quandle} $\Alex(A,\sigma)$ is the set $A$ equipped with
    \[ s_x(y)=(1-\sigma)x+\sigma(y) \qquad (x,y\in A). \]
    For $a\in A$, let $L_a\colon A\to A$ be the translation $L_a(y)=a+y$. It is immediate to see that $s_0=\sigma$ and for any $x\in \Alex(A,\sigma)$, $s_x = L_{(1-\sigma)x} \circ \sigma$ as maps.
    
    Let $(B,\tau)$ be another abelian group with an automorphism, and let $\Phi\colon A\twoheadrightarrow B$ be a surjective group homomorphism satisfying $\Phi\sigma=\tau\Phi$. It induces a surjective quandle homomorphism $f\coloneqq \Alex(\Phi)\colon \Alex(A,\sigma)\twoheadrightarrow\Alex(B,\tau)$.
    Write $K\coloneqq\Ker(\Phi)$ and $D\coloneqq(1-\sigma)A\cap K$. Then, we have
    \begin{align*}
        \relInn(f) &= \{L_a\sigma^n \mid a\in (1-\sigma)A\cap K,\ n\in\ZZ \text{ such that }\tau^n=\id_B\}, \text{ and} \\
        \relTrans(f) &= \{L_a \mid a\in(1-\sigma)K\} \cong (1-\sigma)K
    \end{align*}
    (see {\cite[Propositions~5.12 and 5.14]{preprint_yuki_tomoki_2026_relativization_of_symmetries_on_quandles}}).
    
    In particular, if the induced homomorphism $\langle\sigma\rangle \twoheadrightarrow \langle\tau\rangle$ is injective, then $\relInn(f)=\{L_a\mid a\in D\}$.
\end{eg}

\begin{eg}[Linear Alexander quandles]
\label[eg]{example: linear Alexander quandle}
    Let $n\geq1$, and let $a$ be an integer coprime to $n$. The Alexander quandle associated with multiplication by $a$ on $\ZZ/n\ZZ$ is denoted by $\Lambda_{n,a}$. Explicitly, the symmetry maps on the Alexander quandle $\Lambda_{n,a}$ are given by
    \[ s_x(y)=(1-a)x+ay \qquad (x,y\in\mathbb Z/n\mathbb Z). \]
    In particular, $\Lambda_{n,1}$ is the trivial quandle of order $n$ and $R_n\coloneqq \Lambda_{n, -1}$ is the dihedral quandle of order $n$.
\end{eg}

\section{Nilpotent quandle homomorphisms}
\label{section: nilpotent homomorphisms}

In this section, we introduce nilpotent surjective homomorphisms of
quandles and study their structure in terms of coverings.

\subsection{Nilpotent quandles}
\label{subsection: nilpotent quandles}

For a group $G$, the \emph{lower central series} of $G$ is defined inductively by $\grpLCS_0(G)=G$ and $\grpLCS_{i+1}(G)=[G,\grpLCS_i(G)]$. 
The group $G$ is called \emph{nilpotent} if $\grpLCS_c(G)$ is trivial for some $c\geq 0$. 
Remark here that we adopt the degree-zero convention on indexing: the lower central series starts at $0$, while Darn\'e's notation uses indices starting at $1$.

Using the inner automorphism group, Darn\'e introduced nilpotent quandles:

\begin{dfn}[{\cite[Definition~2.1]{darne_2026_nilpotent_quandles}}]
\label[dfn]{definition: nilpotent quandle}
    A quandle $Q$ is called \emph{nilpotent} if $\Inn(Q)$ is a nilpotent group. More precisely, we say that $Q$ is \emph{nilpotent of class at most $c$}, or \emph{$c$-nilpotent}, if $\grpLCS_{c}(\Inn(Q))=1$. The smallest such integer $c$ is called the \emph{nilpotency class} of $Q$ and is denoted by $\ncl(Q)$.
\end{dfn}

\begin{rem}
    As already remarked, our indexing is slightly different from Darn\'e's \cite{darne_2026_nilpotent_quandles}.
    In terms of our notation, his definition of $(c+1)$-nilpotency is $\relLCS_{c}(\Inn(Q))=1$.
    Thus, our numerical class is one less than Darn\'e's quandle-class convention.
\end{rem}

\begin{eg}
    A quandle $Q$ is $0$-nilpotent if and only if $\Inn(Q)$ is a trivial group, if and only if $Q$ is a trivial quandle.
    A quandle $Q$ is $1$-nilpotent if and only if $\Inn(Q)$ is an abelian group.
\end{eg}

\subsection{Nilpotent homomorphisms and their basic properties}
\label{subsection: nilpotent homomorphisms and their basic properties}

In the spirit of relativization, we now introduce a notion of nilpotency for quandle homomorphisms. The nilpotency of the relative inner automorphism group alone, however, is not sufficient; we need to consider the conjugation action of the whole inner automorphism group. In other words, we utilize the notion of relative nilpotency of a normal subgroup.

Recall that for a normal subgroup $N$ of a group $G$, the \emph{$G$-relative lower central series} of $N$ is defined inductively by $\relLCS_0(G,N)=N$ and $\relLCS_{i+1}(G,N)=[G,\relLCS_i(G,N)]$. 
Then, $N$ is called \emph{$G$-nilpotent} if $\relLCS_c(G,N)$ is trivial for some $c\geq 0$. See \cref{subsection in appendix: relative nilpotency for groups} for the details on relative nilpotency of groups.

Applying this to $\relInn(f)\trianglelefteq\Inn(P)$, we introduce the following definition.

\begin{dfn}
\label[dfn]{definition: nilpotent quandle homomorphism}
    A surjective quandle homomorphism $f\colon P\twoheadrightarrow Q$ is called \emph{nilpotent} if $\relInn(f)$ is $\Inn(P)$-nilpotent, or equivalently if the group homomorphism $f_*\colon \Inn(P)\twoheadrightarrow \Inn(Q)$ is nilpotent. More precisely, for $c\geq0$, we say that $f$ is \emph{nilpotent of class at most $c$}, or \emph{$c$-nilpotent}, if $\relLCS_{c}(\Inn(P),\relInn(f))=1$. We write $\ncl(f)\coloneqq \ncl(f_*)=\ncl_{\Inn(P)}(\relInn(f))$ for its nilpotency class.
\end{dfn}

In what follows, we will often use the abbreviation $\relLCS_i(f)\coloneqq \relLCS_i(\Inn(P),\relInn(f))$.

%Note that, in our convention in \cref{appendix: relative nilpotency for groups}, when $f$ is not nilpotent, $\ncl(f)=\infty$.

\begin{eg}
    For a non-empty quandle $P$, consider the unique homomorphism $f\colon P\twoheadrightarrow\{\ast\}$ into the terminal quandle. Then, since $\relInn(f) = \Inn(P)$, $f$ is $c$-nilpotent if and only if $P$ is a $c$-nilpotent quandle.

    On the other hand, we do not consider the empty quandle nilpotent, since the unique homomorphism $\emptyset\to\{\ast\}$ is not surjective. This differs from Darn\'e's convention; see \cref{remark: minor caveat concerning the empty quandle}.
\end{eg}

The following characterizations of the first two nilpotency classes are immediate.

\begin{prop}
\label[prop]{proposition: low nilpotency classes}
    Let $f\colon P\twoheadrightarrow Q$ be a surjective quandle homomorphism.
    \begin{enumerate}
        \item $f$ is $0$-nilpotent if and only if it is rigid.
        \item $f$ is $1$-nilpotent if and only if $f_*\colon \Inn(P)\to \Inn(Q)$ is a central extension.
    \end{enumerate}
\end{prop}

In particular, we have the implications: $0$-nilpotent $\Leftrightarrow$ rigid $\Rightarrow$ covering $\Rightarrow$ $1$-nilpotent.

\begin{prop}
\label[prop]{proposition: nilpotency under composition}
Let $f\colon P\twoheadrightarrow Q$ and $g\colon Q\twoheadrightarrow R$ be surjective quandle homomorphisms. 
\begin{enumerate}
    \item\label{item:composition_of_nilpotent_quandle_hom} If $f$ is $c$-nilpotent and $g$ is $d$-nilpotent, then $g \circ f$ is $(c+d)$-nilpotent.
    \item\label{item:cancelation_of_nilpotent_quandle_hom} If $g\circ f$ is $c$-nilpotent, then both $f$ and $g$ are so.
\end{enumerate}
Consequently, whenever all three homomorphisms are nilpotent,
\[
\max\{\ncl(f),\ncl(g)\}\leq\ncl(g\circ f) \leq\ncl(f)+\ncl(g).
\]
\end{prop}

\begin{proof}
By the functoriality of $\Inn(-)$ for surjections (\cref{proposition: functoriality of Inn for surjections}), we have surjective group homomorphisms $f_*\colon \Inn(P)\to \Inn(Q)$ and $g_*\colon \Inn(Q)\to\Inn(R)$.
Applying \cref{proposition: composition of nilpotent group homomorphisms} to these homomorphisms yields the assertions.
\end{proof}

\begin{prop}[cf. {\cite[Proposition 2.9]{ishihara_tamaru_2016_flat_connected_finite_quandles}}]
\label[prop]{proposition: inner automorphism groups of products}
Let $P$ and $P'$ be quandles. Then, there is a canonical injective group homomorphism
\[ \iota_{P,P'}\colon \Inn(P\times P') \hookrightarrow \Inn(P)\times\Inn(P') \]
satisfying $\iota_{P,P'}(s_{(x,x')})=(s_x,s_{x'})$. Moreover, these homomorphisms are natural with respect to surjective quandle homomorphisms.
\end{prop}

\begin{proof}
If either factor is empty, then $\Inn(P\times P')=1$, and the assertions are immediate.
Suppose that $P$ and $P'$ are non-empty, and let $\pi_1\colon P\times P'\twoheadrightarrow P$ and $\pi_2\colon P\times P'\twoheadrightarrow P'$ be the projections. Then, we obtain a natural group homomorphism $\iota_{P,P'}\coloneqq((\pi_1)_*,(\pi_2)_*)\colon \Inn(P\times P')\to \Inn(P)\times\Inn(P')$ from the universal property of the product of groups.

Let $\varphi\in\Ker(\iota_{P,P'})$. Then, $(\pi_1)_*(\varphi)=\id=(\pi_2)_*(\varphi)$, so $\pi_1\circ\varphi=\pi_1$ and $\pi_2\circ\varphi=\pi_2$. Thus, the universal property of $P\times P'$ gives $\varphi=\id_{P\times P'}$. This shows $\iota_{P,P'}$ is injective.

For surjective homomorphisms $f\colon P\twoheadrightarrow Q$ and $f'\colon P'\twoheadrightarrow Q'$, both $(f_*\times f'_*)\circ\iota_{P,P'}$ and $\iota_{Q,Q'}\circ(f\times f')_*$ send $s_{(x,x')}$ to $(s_{f(x)},s_{f'(x')})$. Since the symmetry maps generate $\Inn(P\times P')$, these homomorphisms coincide, proving naturality.
\end{proof}

\begin{prop}
\label[prop]{proposition: nilpotency under finite products}
    Let $f\colon P\twoheadrightarrow Q$ and $f'\colon P'\twoheadrightarrow Q'$ be $c$- and $c'$-nilpotent, respectively. Then $f\times f'\colon P\times P'\twoheadrightarrow Q\times Q'$ is nilpotent of class at most $\max\{c,c'\}$.
\end{prop}

\begin{proof}
By \cref{proposition: inner automorphism groups of products}, we have an inclusion between surjections of groups:
\[\begin{tikzcd}
    \Inn(P)\times \Inn(P') \arrow[two heads]{r}{f_*\times f'_*} & \Inn(Q) \times \Inn(Q') \\
    \Inn(P\times P') \arrow[two heads]{r}{(f\times f')_*} \arrow[hook]{u} & \Inn(Q\times Q')\rlap{.} \arrow[hook]{u}
\end{tikzcd}\]
Since $f_*$ and $f'_*$ are nilpotent of class at most $c$ and $c'$ by the assumption, respectively, the product $f_*\times f'_*$ is nilpotent of class at most $d\coloneqq\max\{c,c'\}$ by \cref{proposition: relative nilpotency under finite products}. Therefore, by \cref{lemma: relative lower central series under subgroups}, so is $(f\times f')_*$, and thus the product $f\times f'$ is $d$-nilpotent.
\end{proof}

The relative-inner filtration records the number of covering steps before the final rigid factor.
The characterization below is a direct relativization of \cite[Theorem 2.13]{darne_2026_nilpotent_quandles}.

\begin{lem}
\label[lem]{lemma: quotient covering via commutator}
    Let $P$ be a quandle, and $N\subseteq M\subseteq \Inn(P)$ be (normal) subgroups of $\Inn(P)$. If $[\Inn(P),M] \subseteq N$, then the induced quandle homomorphism $h \colon P/N \to P/M$ is a covering homomorphism.
\end{lem}
\begin{proof}
%Since $[\Inn(P),N]\subseteq[\Inn(P),M]\subseteq N$ and $[\Inn(P),M]\subseteq N\subseteq M$, both $N$ and $M$ are normal in $\Inn(P)$. Thus, the quotients $P/N$ and $P/M$ are defined.
Let $\pi_N$ and $\pi_M$ be the quotient map by $N$ and $M$, respectively. Take $\pi_N(x),\pi_N(y)\in P/N$, where $x,y \in P$, such that they have the same image in $P/M$. Then, $\pi_M(x)=h(\pi_N(x))=h(\pi_N(y))=\pi_M(y)$, so we have $y=\varphi(x)$ for some $\varphi\in M$. The assumption implies that
\[ s_xs_y^{-1} =s_xs_{\varphi(x)}^{-1} = s_x\varphi s_x^{-1} \varphi^{-1} =[s_x,\varphi] \in [\Inn(P),M] \subseteq N. \]
Thus, it follows that $(\pi_N)_*(s_xs_y^{-1})=\id$, that is, $s_{\pi_N(x)}=s_{\pi_N(y)}$ on $P/N$. Therefore, $P/N\twoheadrightarrow P/M$ is a covering homomorphism.
\end{proof}

\begin{thm}
\label[thm]{theorem: nilpotent homomorphism and covering rigid tower}
Let $f\colon P\twoheadrightarrow Q$ be a surjective quandle homomorphism and let $c\geq0$. Then the following conditions are equivalent.
\begin{enumerate}
    \item\label{item:c-nilpotent} $f$ is $c$-nilpotent.
    \item\label{item:factorization_into_c_covering_and_1_rigid} $f$ admits a factorization
    \[
    \begin{tikzcd}
        P=P_{c}
        \arrow[two heads]{r}{p_c}
        &
        P_{c-1}
        \arrow[two heads]{r}{p_{c-1}}
        &
        \cdots
        \arrow[two heads]{r}
        &
        P_{1}
        \arrow[two heads]{r}{p_1}
        &
        P_{0}
        \arrow[two heads]{r}{h}
        &
        Q
    \end{tikzcd}
    \]
    in which the first $c$ homomorphisms $p_c,\dots,p_1$ are coverings and the last homomorphism $h$ is rigid.
\end{enumerate}
\end{thm}

\begin{proof}
\eqref{item:c-nilpotent} $\Rightarrow$ \eqref{item:factorization_into_c_covering_and_1_rigid}:
We abbreviate $\relLCS_i(\Inn(P),\relInn(f))$ to $\relLCS_i(f)$. By \cref{theorem: connected-rigid factorization}, $f$ factors as a connected homomorphism $p\colon P\to P/\relInn(f)$ followed by a rigid homomorphism $h\colon P/\relInn(f) \to Q$. If $f$ is $c$-nilpotent, then $\relLCS_c(f)=1$, and hence the $\Inn(P)$-relative lower central series
\[ 1= \relLCS_c(f) \subseteq \relLCS_{c-1}(f) \subseteq \cdots \subseteq \relLCS_0(f) = \relInn(f) \]
stops. Taking successive quotients by these normal subgroups, we moreover obtain the following factorization of $p$:
\[
\begin{tikzcd}
    P=P/\relLCS_c(f)
    \arrow[two heads]{r}{p_c}
    &
    P/\relLCS_{c-1}(f)
    \arrow[two heads]{r}{p_{c-1}}
    &
    \cdots
    \arrow[two heads]{r}{p_1}
    &
    P/\relLCS_0(f) = P/\relInn(f).
\end{tikzcd}
\]
By \cref{lemma: quotient covering via commutator}, all $p_i$ are in fact a covering homomorphism. Thus $f$ has the required factorization.

\eqref{item:factorization_into_c_covering_and_1_rigid} $\Rightarrow$ \eqref{item:c-nilpotent}: By \cref{proposition: covering hom is central extension}, applying $\Inn(-)$ to the factorization of $f$ induces a factorization of $f_*$ into $c$ central extensions. Therefore, by \cref{theorem: characterization of relative nilpotency}, the group homomorphism $f_*$ is nilpotent of class at most $c$, and hence $f$ is $c$-nilpotent.
\end{proof}

By \cref{theorem: connected-rigid factorization}, the rigid part of the connected--rigid factorization of a connected homomorphism is an isomorphism. Thus, we have the following corollary.

\begin{cor}
\label[cor]{corollary: connected nilpotent homomorphism}
If $f\colon P\twoheadrightarrow Q$ is connected, then $f$ is $c$-nilpotent if and only if it is a composite of $c$ coverings.
\end{cor}

\begin{prop}
    Let $f\colon P\to Q$ be a surjective quandle homomorphism and $k\colon Q'\to Q$ be an arbitrary homomorphism. Consider the pullback $f'\colon P'\coloneqq P\times_Q Q'\to Q'$ of $f$ along $k$. If $f$ is $c$-nilpotent, then the pullback $f'$ is $(c+1)$-nilpotent.
\end{prop}

\begin{proof}
    If $f$ is $c$-nilpotent, then $f$ is a composition of $c+1$ covering homomorphisms by \cref{theorem: nilpotent homomorphism and covering rigid tower}, since rigid homomorphisms are covering. Pulling back this chain yields a factorization of $f'$ into $c+1$ covering homomorphisms, by the stability of covering homomorphisms under pullback (\cref{proposition: pull-back of covering is again covering}). Applying $\Inn(-)$ to this factorization, we obtain a decomposition of $f'_*$ into $c+1$ central extensions. Hence, $f'$ is $(c+1)$-nilpotent by \cref{theorem: characterization of relative nilpotency}.
\end{proof}

\begin{cor}
\label[cor]{corollary: fibers of nilpotent homomorphisms}
If $f\colon P\twoheadrightarrow Q$ is $c$-nilpotent, then every fiber
$f^{-1}(q)$, regarded as a subquandle of $P$, is $c$-nilpotent.
\end{cor}

\begin{proof}
   Consider the pullback of the factorization in \cref{theorem: nilpotent homomorphism and covering rigid tower} along $\{q\}\hookrightarrow Q$. The first $c$ homomorphisms remain coverings, and the last one is rigid since its domain is a fiber of a covering homomorphism and hence a trivial quandle. Thus the assertion follows from \cref{theorem: nilpotent homomorphism and covering rigid tower} again.
\end{proof}

Darn\'e's universal nilpotent quotient construction in \cite[Proposition~2.15]{darne_2026_nilpotent_quandles} admits the following relative version.

\begin{prop}
\label[prop]{proposition: universal property of nilpotent truncation}
    Let $f\colon P\to Q$ be a surjective quandle homomorphism and let $c\geq0$. Then, the induced homomorphism $\nilpTrunc{c}{f}\colon P/\relInnLCS{c}{f} \twoheadrightarrow Q$ is $c$-nilpotent. 
    
    Moreover, the factorization $f= \nilpTrunc{c}{f}\circ \pi$, where $\pi\colon P\to P/\relInnLCS{c}{f}$ is the quotient map, is universal among factorizations of $f$ whose second homomorphism is $c$-nilpotent.
    That is, if $f=h\circ u$ with $u\colon P\twoheadrightarrow R$ surjective and $h\colon R\twoheadrightarrow Q$ $c$-nilpotent, then $u$ factors uniquely through $P/\relInnLCS{c}{f}$.
\end{prop}

\begin{proof}
    Note that since $\relInnLCS{c}{f}\subseteq \relInn(f)$, $f$ factors through the quotient map $\pi\colon P\to P/\relInnLCS{c}{f}$ as $f=\nilpTrunc{c}{f}\circ \pi$ (\cref{prop:basic_property_of_quotient_by_normal_subgroup}). Applying $\Inn(-)$, we have $f_*= (\nilpTrunc{c}{f})_*\circ \pi_*$. Since we see from \cref{lemma: relative symmetry groups under a surjective factorization}
    \[\pi_*(\relInn(f))=\pi_*(\Ker(f_*))= \Ker((\nilpTrunc{c}{f})_*)= \relInn(\nilpTrunc{c}{f}), \]
    \cref{proposition: functoriality of relative lower central series} implies that
    \[ \relInnLCS{c}{\nilpTrunc{c}{f}} = \relLCS_c(\Inn(P/\relInnLCS{c}{f}),\relInn(\nilpTrunc{c}{f})) = \pi_*( \relLCS_c(\Inn(P),\relInn(f)))=\pi_*(\relInnLCS{c}{f}) =1, \]
    which shows that $\nilpTrunc{c}{f}$ is $c$-nilpotent.
    
    For the universal property, if $f=h \circ u$ and $h$ is $c$-nilpotent, then $u_*(\relInnLCS{c}{f})=\relInnLCS{c}{h}=1$, so $u$ factors uniquely through $P/\relInnLCS{c}{f}$ by \cref{prop:basic_property_of_quotient_by_normal_subgroup}.
\end{proof}

\begin{dfn}
\label[dfn]{definition: nilpotent truncations}
    Let $c\geq0$. For a surjective quandle homomorphism $f\colon P \twoheadrightarrow Q$, the induced $c$-nilpotent homomorphism $\nilpTrunc{c}{f}\colon P/\relInnLCS{c}{f}\twoheadrightarrow Q$ is called the \emph{$c$-nilpotent truncation} of $f$.
\end{dfn}

\subsection{Nilpotency via relative transvection groups}
\label{subsection: Nilpotency via relative transvections and adjoint groups}

We have defined the nilpotency of $f$ in terms of the relative inner group $\relInn(f)$. We next show that the same nilpotency condition can also be characterized in terms of the relative transvection group $\relTrans(f)$. The following inclusion provides the key link between these two groups.

\begin{lem}[{\cite[Lemma 3.11]{preprint_yuki_tomoki_2026_relativization_of_symmetries_on_quandles}}]
\label{lem:commutator_inn_rel_trans_rel}
    For a surjective quandle homomorphism $f \colon P \twoheadrightarrow Q$, we have $[\Inn(P), \relInn(f)] \subseteq \relTrans(f)$.
    Moreover, if $f$ is connected, the equality $[\Inn(P), \relInn(f)] = \relTrans(f)$ holds.
\end{lem}

We have defined $\relInnLCS{i}{f}\coloneqq \relLCS_i(\Inn(P),\relInn(f))$ for a surjective homomorphism $f\colon P \twoheadrightarrow Q$. We also define $\relTransLCS{i}{f}\coloneqq\relLCS_{i}(\Inn(P),\relTrans(f))$ for the $i$-th term of the relative lower central series of $\relTrans(f)$.

\begin{prop}
\label[prop]{proposition: nilpotency detected by relative transvection group}
Let $f\colon P\twoheadrightarrow Q$ be a surjective homomorphism, and let $c\geq 0$.
\begin{enumerate}
    \item\label{item:nilpotent_of_relInn_implies_that_of_relTrans} If $\relInn(f)$ is $\Inn(P)$-nilpotent of class at most $c$, then $\relTrans(f)$ is $\Inn(P)$-nilpotent of class at most $c$.
    \item\label{item:nilpotent_of_relTrans_implies_that_of_relInn} If $\relTrans(f)$ is $\Inn(P)$-nilpotent of class at most $c$, then $\relInn(f)$ is $\Inn(P)$-nilpotent of class at most $c+1$.
\end{enumerate}
\end{prop}

\begin{proof}
\eqref{item:nilpotent_of_relInn_implies_that_of_relTrans}: For all $i\geq 0$, the inclusions $\relTransLCS{i}{f}\subseteq \relInnLCS{i}{f}$ follow inductively from $\relTrans(f)\subseteq \relInn(f)$. Therefore, if $\relInnLCS{c}{f}=1$, then we have $\relTransLCS{c}{f}=1$.

\eqref{item:nilpotent_of_relTrans_implies_that_of_relInn}: \cref{lem:commutator_inn_rel_trans_rel} says $\relInnLCS{1}{f}=[\Inn(P),\relInn(f)]\subseteq \relTrans(f)=\relTransLCS{0}{f}$. Thus, by induction again, we see $\relInnLCS{i+1}{f}=\relLCS_i(\Inn(P),\relInnLCS{1}{f}) \subseteq \relLCS_i(\Inn(P),\relTransLCS{0}{f})=\relTransLCS{i}{f}$.
From this, if $\relTransLCS{c}{f}=1$, then we have $\relInnLCS{c+1}{f}=1$.
\end{proof}

From \cite[Corollary 3.7]{darne_2026_nilpotent_quandles}, any connected and nilpotent quandle is isomorphic to the terminal quandle. This can be relativized as follows.

\begin{prop}
\label[prop]{proposition: strongly connected and nilpotent is iso}
    Any strongly connected and nilpotent homomorphism is an isomorphism.
\end{prop}

\begin{proof}
    If $f\colon P \twoheadrightarrow Q$ is strongly connected, then $\relTrans(f)$ is $\Inn(P)$-perfect by \cref{proposition:relTrans_of_strongly_conncted_is_Inn-perfect}, and $\relTransLCS{i}{f}=\relTrans(f)$ for all $i\geq 0$. Thus, if $f$ is moreover $c$-nilpotent for some $c\geq 0$, then $\relInnLCS{c}{f}$, hence $\relTransLCS{c}{f}=\relTrans(f)$, becomes trivial.
    Since $f$ is isomorphic to $\pi_{\relTrans(f)}$ as quotients of $P$, $f$ is an isomorphism.
\end{proof}

The nilpotency class of $\relTrans(f)$, plus one, turns out to record the length of a decomposition into covering homomorphisms.

\begin{thm}
\label[thm]{theorem: transvection class and iterated coverings}
Let $f\colon P\twoheadrightarrow Q$ be a surjective quandle homomorphism and let $c\geq0$. Then, $\relTransLCS{c}{f}=1$ if and only if $f$ is a composition of $c+1$ covering homomorphisms.
\end{thm}

\begin{proof}
($\Rightarrow$): By \cref{theorem: maximal covering factorization}, $f$ factors as a homomorphism $p\colon P\to P/\relTrans(f)$ followed by a covering homomorphism $h\colon P/\relTrans(f) \to Q$.
If $\relTransLCS{c}{f}=1$, then the $\Inn(P)$-relative lower central series
\[ 1= \relTransLCS{c}{f} \subseteq \relTransLCS{c-1}{f} \subseteq \cdots \subseteq \relTransLCS{0}{f} = \relTrans(f) \]
stops. Taking successive quotients by these normal subgroups, we moreover obtain the following factorization of $p$:
\[
\begin{tikzcd}
    P=P/\relTransLCS{c}{f}
    \arrow[two heads]{r}{p_c}
    &
    P/\relTransLCS{c-1}{f}
    \arrow[two heads]{r}{p_{c-1}}
    &
    \cdots
    \arrow[two heads]{r}{p_1}
    &
    P/\relTransLCS{0}{f}= P/\relTrans(f).
\end{tikzcd}
\]
By \cref{lemma: quotient covering via commutator}, all $p_i$ are in fact covering homomorphisms. Thus $f$ has the desired factorization of $(c+1)$ covering homomorphisms.

($\Leftarrow$): We show by induction on $c$. The case $c=0$ is the characterization of covering homomorphisms (see \cref{lemma: elementary properties of relative symmetry groups}). Next suppose that $f=h\circ p$, where $p\colon P\twoheadrightarrow R$ is a covering and $h\colon R\twoheadrightarrow Q$ is a composition of $c$ coverings. By \cref{lemma: relative symmetry groups under a surjective factorization}, the map $p_*$ sends $\relTrans(f)$ onto $\relTrans(h)$. By the induction hypothesis, $\relTransLCS{c-1}{h}=\relLCS_{c-1}(\Inn(R),\relTrans(h))=1$, so \cref{proposition: functoriality of relative lower central series} gives $p_*(\relTransLCS{c-1}{f})=\relTransLCS{c-1}{h}=1$, and hence $\relTransLCS{c-1}{f}\subseteq\Ker(p_*)$. Since $\Ker(p_*)\subseteq \Inn(P)$ is central by \cref{proposition: covering hom is central extension}, we thus obtain $\relTransLCS{c}{f}=[\Inn(P),\relTransLCS{c-1}{f}]=1$.
\end{proof}

\begin{dfn}
\label[dfn]{definition: covering length and transvection nilpotency class}
    Let $f\colon P\to Q$ be a surjective quandle homomorphism.
    We define the \emph{covering length} $\covlen(f)$ as the smallest integer $n\geq 1$ for which $f$ can be expressed as a composition of $n$ non-invertible covering homomorphisms.
    By convention, isomorphisms are understood to have covering length $0$.
    If $f$ does not admit a covering decomposition, set $\covlen(f)=\infty$.
\end{dfn}

In addition, we put $\tncl(f)\coloneqq\ncl_{\Inn(P)}(\relTrans(f))$.
The above \cref{theorem: transvection class and iterated coverings} says $\tncl(f)+1 = \covlen(f)$ for a non-isomorphism $f$.

\begin{prop}
\label[prop]{theorem: universal property of covering length truncation}
    Let $f\colon P\to Q$ be a surjective quandle homomorphism and let $c\geq 0$. Then, the induced homomorphism $\covTrunc{c}{f}\colon P/\relTransLCS{c}{f} \twoheadrightarrow Q$ admits a decomposition into covering homomorphisms of length at most $c+1$. 
    
    Moreover, the factorization $f= \covTrunc{c}{f}\circ \pi$, where $\pi\colon P\to P/\relTransLCS{c}{f}$ is the quotient map, is universal among factorizations of $f$ whose second homomorphism has covering length at most $c+1$.
    That is, if $f=h\circ u$ with $u\colon P\twoheadrightarrow R$ surjective and $h\colon R\twoheadrightarrow Q$ of covering length at most $c+1$, then $u$ factors uniquely through $P/\relTransLCS{c}{f}$.
\end{prop}

\begin{proof}
    Note that since $\relTransLCS{c}{f}\subseteq \relInn(f)$, $f$ factors through the quotient map $\pi\colon P\to P/\relTransLCS{c}{f}$ as $f=\covTrunc{c}{f}\circ \pi$ (\cref{prop:basic_property_of_quotient_by_normal_subgroup}). 
    From \cref{lemma: relative symmetry groups under a surjective factorization} and \cref{proposition: functoriality of relative lower central series}, it follows that $\pi_*(\relTransLCS{c}{f})=\relTransLCS{c}{\covTrunc{c}{f}}$. Since $\pi_*(\relTransLCS{c}{f})=1$, we have $\relTransLCS{c}{\covTrunc{c}{f}}=1$, and hence $\covTrunc{c}{f}$ is a composition of $c+1$ covering homomorphisms by \cref{theorem: transvection class and iterated coverings}.
    
    For the universal property, if $f=h \circ u$ and $h$ is a composition of at most $c+1$ coverings, then $\relTransLCS{c}{h}=1$ by \cref{theorem: transvection class and iterated coverings}. Hence, $u_*(\relTransLCS{c}{f})=\relTransLCS{c}{h}=1$, and $u$ factors uniquely through $P/\relTransLCS{c}{f}$ by \cref{prop:basic_property_of_quotient_by_normal_subgroup}.
\end{proof}

\begin{dfn}
\label[dfn]{definition: covering length truncations}
    Let $c\geq 1$. For a surjective quandle homomorphism $f\colon P \twoheadrightarrow Q$, the induced homomorphism $\covTrunc{c-1}{f}\colon P/\relTransLCS{c-1}{f}\twoheadrightarrow Q$ of covering length at most $c$ is called the \emph{$c$-th covering-length truncation} of $f$.
\end{dfn}

\subsection{Nilpotency via relative adjoint groups}
\label{subsection: Nilpotency via relative adjoint groups}

The relative adjoint group, defined below, gives a third group-theoretic realization of the nilpotency condition. While the two groups $\relInn(f)$ and $\relTrans(f)$ are subgroups of $\Inn(P)$, the relative adjoint group lives in the central extension $\As(P)\twoheadrightarrow\Inn(P)$ and therefore changes the class by at most one. This is the relative counterpart of \cite[Corollary~2.5]{darne_2026_nilpotent_quandles}.

\begin{dfn}
\label[dfn]{definition: relative adjoint group}
For a surjective quandle homomorphism $f\colon P\twoheadrightarrow Q$, we define
\[
\relAs(f)\coloneqq\Ker(\As(f)\colon \As(P)\twoheadrightarrow\As(Q))
\]
and call it the \emph{relative adjoint group} of $f$.
\end{dfn}

\begin{lem}
\label[lem]{lemma: generators of relative adjoint group}
For a surjective homomorphism $f\colon P\twoheadrightarrow Q$, we have
\[
    \relAs(f) = \left\langle g_{x}g_{y}^{-1} \mid x,y\in P \text{ such that } f(x)=f(y) \right\rangle.
\]
\end{lem}

\begin{proof}
Let $E$ denote the subgroup on the right. Since $\As(f)(g_{x}g_{y}^{-1})=1$ whenever $f(x)=f(y)$, we have $E\subseteq\relAs(f)$.

First note that $E$ is normal in $\As(P)$. Indeed, $g_{z}(g_{x}g_{y}^{-1})g_{z}^{-1}=g_{s_z(x)}g_{s_z(y)}^{-1}$, and $f(s_z(x))=f(s_z(y))$ whenever $f(x)=f(y)$. Since $s_z$ is bijective, conjugation by $g_{z}$ permutes the generators of $E$.

Let $\pi\colon\As(P)\twoheadrightarrow\As(P)/E$ be the quotient map. For $q\in Q$, define $\overline{g}_{q}$ to be the common value $\pi(g_{x})$ for $x\in f^{-1}(q)$. If $x,x'\in f^{-1}(q)$, then $g_xg_{x'}^{-1}\in E$, and hence $\pi(g_x)=\pi(g_{x'})$. Thus, $\overline{g}_q $ is well-defined. If $f(x)=q$ and $f(y)=r$, then $\overline{g}_{q}\overline{g}_{r} =\pi(g_{x}g_{y}) =\pi(g_{s_x(y)}g_{x}) =\overline{g}_{s_q(r)}\overline{g}_{q}$.
Hence the assignment $g_{q}\mapsto\overline{g}_{q}$ induces a homomorphism $\theta\colon\As(Q)\to\As(P)/E$. For every $x\in P$, $(\theta\circ\As(f))(g_{x})=\pi(g_{x})$, so $\theta\circ\As(f)=\pi$. If $a\in\relAs(f)$, then $\pi(a)=\theta(\As(f)(a))=1$, whence $a\in E$. Thus $\relAs(f)\subseteq E$.
\end{proof}

\begin{prop}
\label[prop]{proposition: relative adjoint group central extension}
    The restriction of $\rho_{P}\colon\As(P)\twoheadrightarrow\Inn(P)$ induces a central extension $\rho_{f}\colon\relAs(f)\twoheadrightarrow\relTrans(f)$.
\end{prop}

\begin{proof}
We see from \cref{lemma: generators of relative adjoint group} that $\rho_{P}(\relAs(f))=\relTrans(f)$, so $\rho_{P}$ restricts to a surjection $\rho_{f}\colon \relAs(f)\twoheadrightarrow \relTrans(f)$. Its kernel $\Ker(\rho_f)= \Ker(\rho_{P}) \cap \relAs(f)$ is contained in $Z(\As(P))\cap \relAs(f) \subseteq Z(\relAs(f))$, since $\Ker(\rho_P) \subseteq Z(\Adj(P))$ by \cref{proposition: elementary properties of adjoint groups}.
\end{proof}

\begin{prop}
\label[prop]{proposition: nilpotency detected by relative adjoint group}
Let $f\colon P\twoheadrightarrow Q$ be a surjective homomorphism, and let $c\geq 0$.
\begin{enumerate}
    \item\label{item:nilpotent_of_relAs_implies_that_of_relTrans} If $\relAs(f)$ is $\As(P)$-nilpotent of class at most $c$, then $\relTrans(f)$ is $\Inn(P)$-nilpotent of class at most $c$.
    \item\label{item:nilpotent_of_relTrans_implies_that_of_relAs} If $\relTrans(f)$ is $\Inn(P)$-nilpotent of class at most $c$, then $\relAs(f)$ is $\As(P)$-nilpotent of class at most $c+1$.
\end{enumerate}
\end{prop}

\begin{proof}
    By \cref{lemma: generators of relative adjoint group}, the surjection $\rho_{P}$ maps $\relAs(f)$ onto $\relTrans(f)$. Hence it follows from \cref{proposition: functoriality of relative lower central series} that $\rho_P(\relLCS_i(\As(P),\relAs(f))) = \relLCS_i(\Inn(P),\relTrans(f))=\relTransLCS{i}{f}$ for every $i\geq 0$.

    \eqref{item:nilpotent_of_relAs_implies_that_of_relTrans}: Thus, if $\relLCS_c(\As(P),\relAs(f))$ is trivial, then so is $\relTransLCS{c}{f}$.

    \eqref{item:nilpotent_of_relTrans_implies_that_of_relAs}: If $\relTransLCS{c}{f}=1$, then $\relLCS_{c}(\As(P), \relAs(f))\subseteq\Ker(\rho_{P})$. The latter subgroup is central in $\As(P)$, so we have $\relLCS_{c+1}(\As(P), \relAs(f))=1$.
\end{proof}

Summarizing the above results, we arrive at the following theorem.

\begin{thm}
\label[thm]{theorem: main structure theorem for nilpotent homomorphisms}
Let $f\colon P\twoheadrightarrow Q$ be a surjective quandle
homomorphism. Then the following conditions are equivalent.
\begin{enumerate}[label={(\roman*)}]
    \item\label{item: f is nilp} $f$ is nilpotent, that is, $\relInn(f)$ is $\Inn(P)$-nilpotent.
    \item\label{item: reltrans is nilp} $\relTrans(f)$ is $\Inn(P)$-nilpotent.
    \item\label{item: relas is nilp} $\relAs(f)$ is $\As(P)$-nilpotent.
    \item\label{item: composition of coverings} $f$ is a finite composition of covering homomorphisms.
\end{enumerate}
\end{thm}

\begin{proof}
The equivalence of \ref{item: f is nilp} and \ref{item: reltrans is nilp} follows from \cref{proposition: nilpotency detected by relative transvection group},
the equivalence of \ref{item: reltrans is nilp} and \ref{item: relas is nilp} follows from \cref{proposition: nilpotency detected by relative adjoint group}, and
the equivalence of \ref{item: reltrans is nilp} and \ref{item: composition of coverings} follows from \cref{theorem: transvection class and iterated coverings}.
\end{proof}

\begin{cor}
\label{corollary: characterization_of_nilpitent_quandles}
    For a non-empty quandle $P$, the following are equivalent.
    \begin{enumerate}[label={(\roman*)}]
    \item\label{item: P is nilp} $P$ is nilpotent, that is, $\Inn(P)$ is nilpotent.
    \item\label{item: Trans is nilp} $\Trans(P)$ is $\Inn(P)$-nilpotent.
    \item\label{item: Adj_zero is nilp} $\As_0(P)=\relAs(P\to \{\ast\})$ is $\As(P)$-nilpotent.
    \item\label{item: Adj is nilp} $\As(P)$ is nilpotent.
    \item\label{item: constructed by iterated coverings} The map $P\to \{\ast\}$ is a finite composition of covering homomorphisms.
\end{enumerate}
\end{cor}

\begin{proof}
\ref{item: Adj_zero is nilp} $\Leftrightarrow$ \ref{item: Adj is nilp}: 
Since $\As(\{*\})\cong\ZZ$ and $\As(P)\to\As(\{*\})$ is the degree homomorphism, we have $\relAs(P\to\{*\})=\As_0(P)$ and $\As(P)/\As_0(P)\cong\ZZ$. 
Thus, \ref{item: Adj_zero is nilp} and \ref{item: Adj is nilp} are equivalent by \cref{lemma: lower central series with cyclic quotient}.

The other equivalences follow from \cref{theorem: main structure theorem for nilpotent homomorphisms}.
\end{proof}

\begin{rem}
    The equivalences among \ref{item: P is nilp}, \ref{item: Adj is nilp}, and \ref{item: constructed by iterated coverings} in \cref{corollary: characterization_of_nilpitent_quandles} were previously established in \cite[Corollary~2.5 and Theorem~2.13]{darne_2026_nilpotent_quandles}, while the equivalences with \ref{item: Trans is nilp} and \ref{item: Adj_zero is nilp} are new.
\end{rem}

\begin{rem}
\label{remark: minor caveat concerning the empty quandle}
    There is one minor point to note concerning the empty quandle. In Darn\'{e}'s sense, a quandle $Q$ is nilpotent if its inner automorphism group $\Inn(Q)$ is nilpotent. Thus, the empty quandle is nilpotent in this sense, since $\Inn(\emptyset)=\{\id\}$. On the other hand, the unique homomorphism $\emptyset\to\{\ast\}$ is not surjective and hence is not a finite composition of covering homomorphisms. Thus, the two notions of nilpotency differ only in the case of the empty quandle.
\end{rem}

\begin{rem}
    Nilpotency for quandles is also studied by Bonatto and Stanovsk\'{y}~\cite{bonatto_stanovsky_2021_commutator_theory_for_racks_and_quandles} from the universal-algebraic viewpoint. According to \cite[Theorem 1.2]{bonatto_stanovsky_2021_commutator_theory_for_racks_and_quandles}, the nilpotency of a quandle $P$ in their sense is equivalent to that of $\Trans(P)$. Hence, our notion of nilpotency is stronger than theirs, as noted by Darn\'{e} in the introduction of \cite{darne_2026_nilpotent_quandles}.
\end{rem}

\subsection{Comparison of nilpotency classes and covering length}
\label{subsection: Comparison of nilpotency classes and covering length}

We summarize and compare the three numerical invariants: the $\Inn(P)$-nilpotency class $\ncl(f)$ of $\relInn(f)$, the $\Inn(P)$-nilpotency class $\tncl(f)$ of $\relTrans(f)$, and the covering length $\covlen(f)$.

\begin{thm}
\label{theorem: comparison of ncl tncl covlen}
Let $f\colon P\twoheadrightarrow Q$ be a surjective quandle homomorphism. Then,
\begin{align*}
    \tncl(f)&\leq \ncl(f)\leq \tncl(f)+1,\text{ and} \\
    \tncl(f)&\leq\ncl_{\As(P)}(\relAs(f))\leq\tncl(f)+1.
\end{align*}
If $f$ is not an isomorphism, then $\covlen(f)=\tncl(f)+1$ and hence $\ncl(f) \leq\covlen(f) \leq\ncl(f)+1$.
If $f$ is connected, then $\covlen(f)=\ncl(f)$.
\end{thm}

\begin{proof}
The first inequalities follow from \cref{proposition: nilpotency detected by relative transvection group}, and the second from \cref{proposition: nilpotency detected by relative adjoint group}.
The formula for $\covlen(f)$ follows from \cref{theorem: transvection class and iterated coverings}, and the connected case follows from \cref{corollary: connected nilpotent homomorphism}.
\end{proof}

For a non-isomorphism, the covering length is either $\ncl(f)$ or $\ncl(f)+1$, and the nilpotency class of the relative transvection group is either $\ncl(f)-1$ or $\ncl(f)$. The following examples show that both possibilities occur.

\begin{eg}
\label[eg]{example: basic nilpotent homomorphisms}
Let $f\colon X\twoheadrightarrow Y$ be a non-bijective surjection between trivial quandles. Then $\Inn(X)=\Inn(Y)=1$, so the map is rigid and hence $0$-nilpotent. It is also a covering, but not an isomorphism. Thus $\ncl(f)=\tncl(f)=0$ and $\covlen(f)=1$.
\end{eg}

\begin{eg}
\label[eg]{example: two possible class gaps}
Let
$f\colon\Lambda_{4, 3}\twoheadrightarrow\Lambda_{2,1}$ be reduction
modulo $2$ between linear Alexander quandles. If $x\equiv y\pmod2$, then
$(1-3)(x-y)\equiv0\pmod4$, so $s_x=s_y$. Hence $f$ is a covering and
$\relTrans(f)=1$. The target is trivial, while
$\Inn(\Lambda_{4,3})$ is a nontrivial abelian group. Therefore
$\ncl(f)=\covlen(f)=1$ and $\tncl(f)=0$.

%Together with the non-bijective surjection between trivial quandles in \cref{example: basic nilpotent homomorphisms}, this realizes both possible gaps in \cref{theorem: main structure theorem for nilpotent homomorphisms}.

Note that, although the only connected nilpotent quandle is the one-point trivial quandle $\{\ast\}$ (\cite[Corollary 3.7]{darne_2026_nilpotent_quandles}), 
this homomorphism is connected and $1$-nilpotent but not an isomorphism; compare \cref{proposition: strongly connected and nilpotent is iso}.
\end{eg}

We next consider the nilpotency under base change. 
Recall that covering homomorphisms are stable under arbitrary pullbacks by \cref{proposition: pull-back of covering is again covering}, and rigid homomorphisms are stable under pullbacks along surjections by \cref{theorem: connected-rigid factorization}.

\begin{prop}
\label[prop]{proposition: nilpotency and covering length under pullback}
Let $f\colon P\twoheadrightarrow Q$ be a nilpotent homomorphism, $v\colon R\to Q$ a quandle homomorphism, and $f'\colon P\times_QR\twoheadrightarrow R$ the pullback of $f$ along $v$.
Then the following hold.
\begin{enumerate}
    \item\label{item: pullback general inequalities} $\covlen(f')\leq\covlen(f)$, and $\ncl(f')\leq \ncl(f)+1$.
    \item\label{item: pullback surjective case} If $v$ is surjective, then $\covlen(f')=\covlen(f)$, $\tncl(f')=\tncl(f)$, and $\ncl(f')\leq \ncl(f)\leq \ncl(f')+1$.
    \item\label{item: pullback connected case} If $f$ is connected, $\ncl(f')\leq\ncl(f)$.
\end{enumerate}
\end{prop}

\begin{proof}
\eqref{item: pullback general inequalities}: Pulling back a covering factorization gives $\covlen(f')\leq\covlen(f)$ by \cref{proposition: pull-back of covering is again covering}. Hence, $\ncl(f')\leq\covlen(f')\leq\covlen(f)\leq\ncl(f)+1$ by \cref{theorem: comparison of ncl tncl covlen}.

\eqref{item: pullback surjective case}: Let $u\colon P\times_Q R\twoheadrightarrow P$ be the projection. By \cite[Proposition~3.30]{preprint_yuki_tomoki_2026_relativization_of_symmetries_on_quandles}, $u_*$ restricts to an isomorphism $\relTrans(f')\xrightarrow{\sim}\relTrans(f)$. \cref{proposition: functoriality of relative lower central series} and this give $u_*\colon\relTransLCS{i}{f'}\xrightarrow{\sim}\relTransLCS{i}{f}$ for every $i\geq0$. Hence, $\tncl(f')=\tncl(f)$. Since $v$ is surjective, $f'$ is an isomorphism if and only if $f$ is. If they are isomorphisms, both covering lengths are $0$. Otherwise, $\covlen(f')=\tncl(f')+1=\tncl(f)+1=\covlen(f)$ by \cref{theorem: transvection class and iterated coverings}.

The inequality $\ncl(f')\leq\ncl(f)$ follows by pulling back the factorization in \cref{theorem: nilpotent homomorphism and covering rigid tower}. By \cref{theorem: comparison of ncl tncl covlen}, we have $\ncl(f)\leq \tncl(f)+1=\tncl(f')+1\leq\ncl(f')+1$.

\eqref{item: pullback connected case}: If $f$ is connected, then $\covlen(f)=\ncl(f)$ by \cref{theorem: comparison of ncl tncl covlen}. Therefore, $\ncl(f')\leq\covlen(f')\leq\covlen(f)=\ncl(f)$.
\end{proof}

\begin{eg}
\label[eg]{example: sharpness of pullback class estimate}
Let $G=Q_8\rtimes\langle\sigma\rangle$, where $Q_8$ is the quaternion group and $\sigma$  cyclically permutes $i,j,k\in Q_8$ and has order $3$. Then, $Z(G)=\{\pm1\}$ and $H\coloneqq G/Z(G)\cong A_4$. 
The quotient $\pi\colon G\twoheadrightarrow H$ induces $f\coloneqq \Conj(\pi)\colon\Conj(G)\twoheadrightarrow\Conj(H)$. 
Since $Z(H)=1$, we have $\relInn(f)=1$, and thus, $f$ is rigid and $\ncl(f)=0$.

Let $V=Q_8/Z(G)\cong V_4\subseteq H$. Pulling back $f$ along $\Conj(V)\hookrightarrow\Conj(H)$ gives $f'\colon \Conj(Q_8)\twoheadrightarrow\Conj(V_4)$. 
As $\Conj(V_4)$ is a trivial quandle, 
\[
\relInn(f') = \Inn(\Conj(Q_8)) \cong Q_8/Z(Q_8)\cong V_4.
\] 
Hence the pullback is $1$-nilpotent but not rigid. 
It shows that the inequality in \eqref{item: pullback general inequalities} of \cref{proposition: nilpotency and covering length under pullback} is sharp.
\end{eg}

We end this section with an explicit computation for Alexander homomorphisms and an application showing that the upper bound for the nilpotency class under composition is also sharp.

\begin{prop}
\label[prop]{example: relative Alexander lower central series}
    In the setting in \cref{example: Alexander quandle}, put $f\coloneqq\Alex(\Phi)$ and $D\coloneqq(1-\sigma)A\cap K$. Then, we have $\relInnLCS{i}{f} = \{L_d\mid d\in(1-\sigma)^iD\}$ for $i\geq1$ and $\relTransLCS{i}{f} = \{L_e\mid e\in(1-\sigma)^{i+1}K\}$ for $i\geq0$.
\end{prop}
\begin{proof}
Put $t\coloneqq1-\sigma$, $G\coloneqq\Inn(\Alex(A,\sigma))$,
and $L(U)\coloneqq\{L_u\mid u\in U\}$.
Since $G$ is generated by $\sigma$ and $L(tA)$, the translations
commute, and $[\sigma,L_u]=L_{-tu}$, we have $[G,L(U)]=L(tU)$
for every $\sigma$-invariant subgroup $U\subseteq tA$.

For $x\in A$ and $\varphi\in\relInn(f)$, we have
$\varphi(x)-x\in D$, since $\varphi$ preserves the cosets of $tA$
and the fibers of $f$. Thus,
$[\varphi,s_x]=s_{\varphi(x)}s_x^{-1}
=L_{t(\varphi(x)-x)}\in L(tD)$.
Since $L(tD)$ is normal in $G$, this gives
$[G,\relInn(f)]\subseteq L(tD)$.
The reverse inclusion follows from $L(D)\subseteq\relInn(f)$
and $[G,L(D)]=L(tD)$.
The two formulas now follow by induction, using
$\relTrans(f)=L(tK)$ from \cref{example: Alexander quandle}.
\end{proof}

From \cref{proposition: nilpotency under composition}, $\ncl(g\circ f)\leq \ncl(f)+\ncl(g)$ in general. The following shows that the equality can occur.

\begin{eg}
\label[eg]{example: sharpness of composition class estimate}
Let $m, n$ be integers such that $m>n\ge2$ and $R_n\coloneqq \Lambda_{n, -1}$ be the dihedral quandle of order $n$.
Consider the quandle homomorphism $h_{m, n}\colon R_{2^m} \twoheadrightarrow R_{2^n}$ induced by the reduction group homomorphism $\Phi\colon \ZZ/2^m\ZZ\twoheadrightarrow\ZZ/2^n\ZZ$.
Since $\Ker(\Phi) = 2^n\ZZ/2^m\ZZ$, $D= 2(\ZZ/2^m\ZZ)\cap \Ker(\Phi) = 2^n\ZZ/2^m\ZZ$. 
Therefore, we have $\relInnLCS{i}{h_{m, n}} = \{L_x \mid x\in 2^{n+i}\ZZ/2^m\ZZ\}$ by \cref{example: relative Alexander lower central series}. 
Thus, we have $\ncl(h_{m, n}) = m-n$.
Applying it to the composition $R_{2^{m+n+2}}\xrightarrow{f} R_{2^{m+2}}\xrightarrow{g} R_{4}$, we have
$\ncl(f)=n$, $\ncl(g)=m$, and $\ncl(g\circ f)=m+n$. Thus, $\ncl(g\circ f)=\ncl(f)+\ncl(g)$ holds.
\end{eg}

\section{Hypocentral homomorphisms and factorization system}
\label{section: hypocentral homomorphisms}

In this section, we extend relative lower central series to transfinite ordinals and introduce hypocentral homomorphisms. The main result is an orthogonal factorization system (for surjections) whose left class consists of strongly connected homomorphisms.

\subsection{Hypocentral homomorphisms of quandles}
\label{subsection: hypocentral homomorphisms of quandles}
We introduce the notion of a \emph{hypocentral homomorphism} and first establish some basic properties parallel to those of nilpotent homomorphisms.

\begin{dfn}
\label[dfn]{definition: transfinite relative lower central series}
Let $N\trianglelefteq G$ be a normal subgroup. We define the \emph{transfinite $G$-relative lower central series} of $N$ in the following way: $\relLCS_0(G,N)\coloneqq N$ and $\relLCS_{\alpha+1}(G,N)\coloneqq[G,\relLCS_\alpha(G,N)]$ for an ordinal $\alpha$. For every non-zero limit ordinal $\lambda$, put
\[
    \relLCS_\lambda(G,N)
    \coloneqq
    \bigcap_{\alpha<\lambda}\relLCS_\alpha(G,N).
\]
We say $N$ is \emph{$G$-hypocentral} if $\relLCS_\alpha(G,N)=1$ for some ordinal $\alpha$.
\end{dfn}

\begin{lem}
\label[lem]{lemma: relative hypocenter}
    Let $N\trianglelefteq G$ be a normal subgroup. The transfinite $G$-relative lower central series of $N$ stabilizes, and its stable term is the largest $G$-perfect normal subgroup of $G$ contained in $N$.
\end{lem}

\begin{proof}
The series stabilizes since $G$ has only set many normal subgroups. Its stable term is $G$-perfect. If $K\trianglelefteq G$ satisfies $K\subseteq N$ and $[G,K]=K$, then transfinite induction gives $K\subseteq\relLCS_\alpha(G,N)$ for every ordinal $\alpha$. Thus, $K$ is contained in the stable term.
\end{proof}

We write $\relHypCenter{G}{N}$ for the stable term and call it the \emph{$G$-relative hypocenter} of $N$. Thus, $N$ is $G$-hypocentral if and only if $\relHypCenter{G}{N}=1$.

\begin{cor}
\label{corollary:relative_hypocentral_iff_no_non-trivial_relative_perfect}
    A normal subgroup $N\trianglelefteq G$ is $G$-hypocentral if and only if there is no non-trivial $G$-perfect normal subgroup of $G$ contained in $N$.
\end{cor}

In turn, we extend the notion of nilpotency to the transfinite setting.

\begin{dfn}
\label[dfn]{definition: hypocentral homomorphism}
    Let $f\colon P\twoheadrightarrow Q$ be a surjective quandle homomorphism. We call $f$ \emph{hypocentral} if $\relInn(f)$ is $\Inn(P)$-hypocentral. 
\end{dfn}

Every nilpotent homomorphism is hypocentral. For the unique homomorphism $P\twoheadrightarrow\{*\}$, the preceding definition is the usual hypocentrality condition for the group $\Inn(P)$.

For every ordinal $\alpha$, we use the notation $\relInnLCS{\alpha}{f}$ and $\relTransLCS{\alpha}{f}$ for the $\alpha$-th term of the transfinite relative lower central series of $\relInn(f)$ and $\relTrans(f)$ in $\Inn(P)$, respectively.

\begin{lem}
\label[lem]{lemma: transfinite interlacing of the two relative lower central series}
Let $f\colon P\twoheadrightarrow Q$ be a surjective quandle homomorphism. Then, for every ordinal $\alpha$, we have
\[ \relInnLCS{\alpha+1}{f} \subseteq \relTransLCS{\alpha}{f} \subseteq \relInnLCS{\alpha}{f}. \]
Moreover, $\relInnLCS{\alpha}{f}=\relTransLCS{\alpha}{f}$ for $\alpha\geq\omega$.
\end{lem}

\begin{proof}
Put $G\coloneqq\Inn(P)$, $N\coloneqq\relInn(f)$, and $T\coloneqq\relTrans(f)$. The second inclusion follows by transfinite induction from $T\subseteq N$. Since $[G,N]\subseteq T$ by \cref{lem:commutator_inn_rel_trans_rel}, the first inclusion holds for $\alpha=0$, and the successor step follows by taking commutators with $G$. If $\lambda$ is a limit ordinal, then $\relLCS_\lambda(G,N)\subseteq\relLCS_{\beta+1}(G,N) \subseteq\relLCS_\beta(G,T)$ for every $\beta<\lambda$. Hence, $[G,\relLCS_\lambda(G,N)]\subseteq\relLCS_\lambda(G,T)$.

Taking intersections over the finite stages gives $\relInnLCS{\omega}{f}=\relTransLCS{\omega}{f}$. The equality for $\alpha\geq\omega$ follows by transfinite induction.
\end{proof}

For a surjective quandle homomorphism $f$, \cref{lemma: transfinite interlacing of the two relative lower central series} shows that the stable terms of the two transfinite relative lower central series associated with $\relInn(f)$ and $\relTrans(f)$ coincide. We will denote this common term by $\textstyle \hypCenter{f}\coloneqq \bigcap_\alpha \relInnLCS{\alpha}{f} = \bigcap_{\alpha} \relTransLCS{\alpha}{f}$.

\begin{cor}
\label[cor]{corollary: relative hypocenter and hypocentral lengths}
A surjective homomorphism $f\colon P\twoheadrightarrow Q$ is hypocentral if and only if the normal subgroup $\relTrans(f)$ is $\Inn(P)$-hypocentral.
\end{cor}

\begin{proof}
This follows from \cref{lemma: transfinite interlacing of the two relative lower central series}.
\end{proof}

\begin{rem}
We call $f$ \emph{residually nilpotent} if $\textstyle \relInnLCS{\omega}{f}=\bigcap_{i<\omega}\relInnLCS{i}{f}=1$. Every residually nilpotent homomorphism is hypocentral.
\end{rem}

\subsection{Characterizations of hypocentral homomorphisms}
\label{subsection: characterizations of hypocentral homomorphisms}

In this subsection, we give characterizations of hypocentral homomorphisms in terms of $\relTrans(f)$, $\relAs(f)$, and $\Inn(P)$-perfect subgroups.
To establish these characterizations, we first record a group-theoretic observation used for the relative adjoint group.

\begin{lem}
\label[lem]{lemma: hypocentrality under central extensions}
    Let $\pi\colon A\twoheadrightarrow G$ be a central extension of groups, let $R\trianglelefteq A$ be a normal subgroup, and put $T\coloneqq\pi(R)$. Then, $R$ is $A$-hypocentral if and only if $T$ is $G$-hypocentral.
\end{lem}

\begin{proof}
($\Leftarrow$): Let $L\trianglelefteq A$ be an $A$-perfect subgroup contained in $R$. Then, since $[G,\pi(L)]=\pi([A,L])=\pi(L)$, the image $\pi(L)$ is $G$-perfect and contained in $T=\pi(R)$. By the assumption that $T$ is $G$-hypocentral, \cref{corollary:relative_hypocentral_iff_no_non-trivial_relative_perfect} yields $\pi(L)=1$. Therefore, we have $L\subseteq \Ker(\pi) \subseteq Z(A)$, which implies $L=[A,L]=1$. Thus $R$ is $A$-hypocentral by \cref{corollary:relative_hypocentral_iff_no_non-trivial_relative_perfect}.

($\Rightarrow$): Let $K\trianglelefteq G$ be a $G$-perfect subgroup contained in $T$. Put $\widetilde{K}\coloneqq(\pi\rvert_R)^{-1}(K) =R\cap\pi^{-1}(K)$ and $L\coloneqq[A,\widetilde{K}]$. Then, $\widetilde{K}\trianglelefteq A$, and hence $L\trianglelefteq A$ and $L\subseteq\widetilde{K}\subseteq R$. Moreover, $\pi(\widetilde{K})=K$ and $\pi(L)=[G,K]=K$.

We claim that $\widetilde{K} = L\bigl(\widetilde{K}\cap\Ker(\pi)\bigr)$. Indeed, the inclusion from right to left is clear. Conversely, let $x\in\widetilde{K}$. Since $\pi(L)=K=\pi(\widetilde{K})$, there is $l\in L$ such that $\pi(l)=\pi(x)$. Then $l^{-1}x\in\widetilde{K}\cap\Ker(\pi)$, and hence $x\in L(\widetilde{K}\cap\Ker(\pi))$.

Since $\widetilde{K}\cap\Ker(\pi)\subseteq Z(A)$, this equality gives
\[ L = [A,\widetilde K] = [A,L(\widetilde K\cap\Ker(\pi))] = [A,L]. \]
Thus, $L$ is an $A$-perfect subgroup contained in $R$, and hence must be trivial by \cref{corollary:relative_hypocentral_iff_no_non-trivial_relative_perfect}. It follows that $K=\pi(L)=1$, which shows that $T$ is $G$-hypocentral.
\end{proof}

\begin{thm}
\label[thm]{theorem: characterizations of hypocentral homomorphisms}
Let $f\colon P\twoheadrightarrow Q$ be a surjective quandle homomorphism. Then, the following conditions are equivalent.
\begin{enumerate}[label={(\roman*)}]
    \item\label{item: f is hypocentral}
    The homomorphism $f$ is hypocentral, that is, $\relInn(f)$ is $\Inn(P)$-hypocentral.
    \item\label{item: relative transvection is hypocentral}
    The group $\relTrans(f)$ is $\Inn(P)$-hypocentral.
    \item\label{item: relative adjoint group is hypocentral}
    The group $\relAs(f)$ is $\As(P)$-hypocentral.
    \item\label{item: no perfect subgroup in relative inner group}
    The group $\relInn(f)$ contains no non-trivial $\Inn(P)$-perfect normal subgroup of $\Inn(P)$.
    %\item\label{item: no strongly connected quotient factor} For every quotient factorization $f=h\circ e$ with $e$ strongly connected, the homomorphism $e$ is an isomorphism.
\end{enumerate}
\end{thm}

\begin{proof}
The equivalence of \ref{item: f is hypocentral} and \ref{item: relative transvection is hypocentral} follows from \cref{corollary: relative hypocenter and hypocentral lengths}.
The equivalence of \ref{item: relative transvection is hypocentral} and \ref{item: relative adjoint group is hypocentral} follows from \cref{lemma: hypocentrality under central extensions,proposition: relative adjoint group central extension}.
By \cref{corollary:relative_hypocentral_iff_no_non-trivial_relative_perfect}, \ref{item: f is hypocentral} is equivalent to \ref{item: no perfect subgroup in relative inner group}. 
%It remains to show \ref{item: no perfect subgroup in relative inner group} $\Leftrightarrow$ \ref{item: no strongly connected quotient factor}.
%
%Assume \ref{item: no perfect subgroup in relative inner group} and write $f=h\circ e$, where $e$ is strongly connected. Then $\relTrans(e)$ is $\Inn(P)$-perfect by \cref{theorem: characterization of strongly connected via perfect subgroup} and is contained in $\relInn(f)$. Hence, the assumption yields $\relTrans(e)=1$. Hence, as $e$ is strongly connected, it follows from \cref{proposition: strongly connected and nilpotent is iso} that $e$ is an isomorphism.
%Conversely, let $K\trianglelefteq\Inn(P)$ be a non-trivial $\Inn(P)$-perfect subgroup contained in $\relInn(f)$. Then $K=[\Inn(P),K]\subseteq[\Inn(P),\relInn(f)]\subseteq\relTrans(f)$. Thus, $f$ factors through $P\twoheadrightarrow P/K$, and this quotient is strongly connected by \cref{theorem: characterization of strongly connected via perfect subgroup}. Hence, \ref{item: no strongly connected quotient factor} fails.
\end{proof}

The following generalizes \cref{proposition: strongly connected and nilpotent is iso}.

\begin{prop}
\label[prop]{theorem: strongly connected and hypocentral is iso}
    Any strongly connected and hypocentral homomorphism is an isomorphism.
\end{prop}

\begin{proof}
Suppose that $f$ is both strongly connected and hypocentral. Then, the group $\relTrans(f)$ is $\Inn(P)$-perfect by \cref{proposition:relTrans_of_strongly_conncted_is_Inn-perfect} and is contained in $\relInn(f)$. Hence, the hypocentrality of $f$ yields $\relTrans(f)=1$ by \cref{theorem: characterizations of hypocentral homomorphisms}. Thus, as $f$ is strongly connected and hence is isomorphic to $\pi_{\relTrans(f)}$, it follows that $f$ is an isomorphism.
\end{proof}

\begin{prop}
\label[prop]{proposition: hypocentrality under composition}
    Let $f\colon P\twoheadrightarrow Q$ and $g\colon Q\twoheadrightarrow R$ be surjective quandle homomorphisms. Then, the following hold.
    \begin{enumerate}
        \item\label{item: hypocentrality is closed under composition}
        If $f$ and $g$ are hypocentral, then so is $g\circ f$.
        \item\label{item: hypocentrality descends to the first factor}
        If $g\circ f$ is hypocentral, then $f$ is hypocentral.
    \end{enumerate}
\end{prop}

\begin{proof}
\eqref{item: hypocentrality is closed under composition}:
Put $h=g\circ f$ and let $K\trianglelefteq\Inn(P)$ be an $\Inn(P)$-perfect subgroup contained in $\relInn(h)$. Then $f_*(K)$ is an $\Inn(Q)$-perfect subgroup contained in $\relInn(g)$, as $[\Inn(Q),f_*(K)]=f_*([\Inn(P),K])=f_*(K)$ and $f_*(\relInn(h)) \subseteq \relInn(g)$. The hypocentrality of $g$ forces $f_*(K)=1$. Hence, $K\subseteq\relInn(f)$, and therefore the hypocentrality of $f$ shows $K=1$.
Thus, it follows from \ref{item: no perfect subgroup in relative inner group} of \cref{theorem: characterizations of hypocentral homomorphisms} that $h=g\circ f$ is hypocentral.

\eqref{item: hypocentrality descends to the first factor}: Since $\relInn(f)\subseteq\relInn(g\circ f)$, this also follows from \ref{item: no perfect subgroup in relative inner group} of \cref{theorem: characterizations of hypocentral homomorphisms}.
\end{proof}

\subsection{Strongly connected--hypocentral factorization}
\label{subsection: strongly connected-hypocentral factorization}

In this subsection, we show that every surjection admits a canonical strongly connected--hypocentral factorization and that these two classes form an orthogonal factorization system for surjections in $\Quandle$.

\begin{thm}[Strongly connected--hypocentral factorization]
\label[thm]{theorem: universal strongly connected-hypocentral factorization}
    Let $f\colon P\twoheadrightarrow Q$ be a surjective quandle homomorphism and put $H\coloneqq\hypCenter{f}$. Then, $f$ factors as 
    \[\begin{tikzcd}[column sep=small]
        P \arrow[two heads]{rr}{f} \arrow[two heads]{rd}[swap]{\hypMap{f}} & & Q \\
        & P/H \arrow[two heads]{ru}[swap]{\hypFac{f}}
    \end{tikzcd}\]
    where $\hypMap{f}$ is strongly connected and $\hypFac{f}$ is hypocentral.
\end{thm}

\begin{proof}
The subgroup $H$ is $\Inn(P)$-perfect and is contained in $\relInn(f)$. Hence, $f$ factors through $P/H$ as $f=\hypFac{f}\circ\hypMap{f}$, and $\hypMap{f}$ is strongly connected by \cref{theorem: characterization of strongly connected via perfect subgroup}.

It remains to show that $\hypFac{f}$ is hypocentral. Put $G=\Inn(P)$, $N=\relInn(f)$, and $S=\relInn(\hypMap{f})$. Then, we have $\Inn(P/H) \cong \Inn(P)/\relInn(\hypMap{f}) = G/S$ and, from \cref{lemma: relative symmetry groups under a surjective factorization}, we see $\relInn(\hypFac{f}) \cong \relInn(f)/\relInn(\hypMap{f}) = N/S$. We need to show that $N/S$ is $G/S$-hypocentral.

Since $\hypMap{f}$ is the quotient map by $H$, we have $H \subseteq \relInn(\hypMap{f})=S \subseteq \relInn(f)=N$. We claim that $N/H$ is $G/H$-hypocentral; indeed, if $K/H$ is a $G/H$-perfect subgroup contained in $N/H$, then the condition $[G/H,K/H]=[G,K]/H=K/H$ implies $[G,K] = K$, which shows that $K$ is $G$-perfect. As $H \subseteq K \subseteq N$, the maximality of $H$ forces $H=K$; in other words, $K/H$ is trivial. Thus $N/H$ is $G/H$-hypocentral by \cref{corollary:relative_hypocentral_iff_no_non-trivial_relative_perfect}.

For every $\varphi\in \relInn(\hypMap{f})=S$ and $x\in P$, we have $\hypMap{f}(\varphi(x))=\hypMap{f}(x)$, and hence $\varphi(x)=\xi_x(x)$ for some $\xi_x \in H$. Then,
\[ [\varphi,s_x]=s_{\varphi(x)}s_x^{-1} = s_{\xi_x(x)}s_x^{-1} = [\xi_x,s_x] \in [H,\Inn(P)]=H, \]
and therefore we obtain $[S,G] \subseteq H$. This shows that the induced map $G/H \twoheadrightarrow G/S$ is a central extension.
As $N/H$ is $G/H$-hypocentral, \cref{lemma: hypocentrality under central extensions} yields that $N/S$ is $G/S$-hypocentral, which completes the proof.
\end{proof}

As in \cite[Definition 3.20]{preprint_yuki_tomoki_2026_relativization_of_symmetries_on_quandles}, we refer to a \emph{pre-covering} homomorphism $f\colon P\to Q$ as a quandle homomorphism satisfying $f(x) = f(y)$ implies $s_x=s_y$ for all $x,y\in P$. Surjective pre-covering homomorphisms are precisely covering homomorphisms.
The following orthogonality result is established in the same reference.

\begin{prop}[{\cite[Proposition 3.21]{preprint_yuki_tomoki_2026_relativization_of_symmetries_on_quandles}}]
\label[prop]{lemma: strongly connected homomorphisms and covering condition}
    The class of strongly connected homomorphisms is left orthogonal to the class of pre-covering homomorphisms in the category of quandles.
\end{prop}

Since (pre-)covering maps are not closed under composition, they cannot form the right class of an orthogonal factorization system. We overcome this obstruction by replacing coverings with hypocentral homomorphisms.

We use transfinite cocompositions in the usual sense: at every nonzero limit ordinal, the corresponding object is the limit of the preceding tower.

\begin{comment}
\begin{lem} %%%% well-known fact %%%%%
\label[lem]{lemma: orthogonality under transfinite cocomposition}
    Let $e\colon A\to B$ be a quandle homomorphism, and let $X_\lambda\to X_0$ be the transfinite cocomposite of an inverse ordinal tower $(X_\alpha)_{\alpha\leq\lambda}$. If $e$ is left orthogonal to $X_{\alpha+1}\to X_\alpha$ for every $\alpha<\lambda$, then $e$ is left orthogonal to $X_\lambda\to X_0$.
\end{lem}

\begin{proof}
Consider a commutative square from $e$ to
$X_\lambda\to X_0$.
We construct compatible homomorphisms
$w_\alpha\colon B\to X_\alpha$ by transfinite induction.

In degree $0$, let $w_0$ be the bottom homomorphism of the square.
Suppose that $w_\alpha$ has been constructed. The lifting property
against $X_{\alpha+1}\to X_\alpha$ gives a unique compatible
homomorphism $w_{\alpha+1}$.
If $\beta$ is a nonzero limit ordinal, the compatible family
$(w_\alpha)_{\alpha<\beta}$ induces a unique homomorphism
$w_\beta\colon B\to X_\beta$ by the universal property of the inverse
limit. Thus, $w_\lambda$ is the required diagonal filler.
The same induction proves uniqueness.
\end{proof}
\end{comment}

\begin{thm}
\label[thm]{theorem: transfinite covering characterization of hypocentral homomorphisms}
    Let $f\colon P\twoheadrightarrow Q$ be a surjective quandle homomorphism. Then, the following conditions are equivalent.
\begin{enumerate}[label={(\roman*)}]
    \item\label{item: homomorphism is hypocentral}
    $f$ is hypocentral.
    \item\label{item: transfinite cocomposite of covering condition}
    $f$ is a transfinite cocomposite of pre-covering homomorphisms.
    \item\label{item: hypocentral right orthogonal}
    $f$ is right orthogonal to every strongly connected homomorphism.
\end{enumerate}
\end{thm}

\begin{proof}
\ref{item: homomorphism is hypocentral} $\Rightarrow$ \ref{item: transfinite cocomposite of covering condition}:
Choose an ordinal $\lambda$ such that $\relInnLCS{\lambda}{f}=1$, and put $C_\alpha\coloneqq P/\relInnLCS{\alpha}{f}$ for $\alpha\leq\lambda$.
Then, $C_\lambda=P$; the induced homomorphism $C_0=P/\relInn(f)\twoheadrightarrow Q$ is rigid and hence a covering homomorphism; and every successor homomorphism $C_{\alpha+1}\twoheadrightarrow C_\alpha$ is a covering by \cref{lemma: quotient covering via commutator}.
So, the induced decomposition tower of $f$
\[\begin{tikzcd}[column sep=small]
P=C_\lambda \arrow[two heads]{r} &
\cdots \arrow[two heads]{r} &
C_{\alpha+1} \arrow[two heads]{r} &
C_{\alpha} \arrow[two heads]{r} &
\cdots \arrow[two heads]{r} &
C_0 \arrow[two heads]{r} & Q
\end{tikzcd}\]
consists of covering homomorphisms at every successor stage. However, at a limit ordinal $\beta$, $C_\beta$ is not necessarily the limit of the preceding tower; so, this tower cannot be regarded as a transfinite cocomposition of covering homomorphisms.

Instead, we refine this tower at its limit stages. Suppose that $\beta\leq\lambda$ is a nonzero limit ordinal, and let $\textstyle L_\beta = \lim_{\alpha<\beta} C_{\alpha}$ be the inverse limit of the part of the refined tower already constructed below $\beta$. The compatible quotient homomorphisms induce a canonical homomorphism $g_\beta\colon C_\beta\to L_\beta$.

We claim that $g_\beta$ satisfies the pre-covering condition. Let $x,y\in P$ and suppose that the classes of $x$ and $y$ in $C_\beta=P/\relInnLCS{\beta}{f}$ have the same image under $g_\beta$. Then, for every ordinal $\alpha<\beta$, their classes in $C_\alpha$ necessarily coincide, and there is $\varphi_\alpha\in\relInnLCS{\alpha}{f}$ such that $y=\varphi_\alpha(x)$. Therefore, $s_ys_x^{-1} = [\varphi_\alpha,s_x] \in \relInnLCS{\alpha+1}{f}$ for every $\alpha<\beta$. Since $\beta$ is a limit ordinal, it follows that $\textstyle s_ys_x^{-1}\in\relInnLCS{\beta}{f}=\bigcap_{\alpha<\beta}\relInnLCS{\alpha}{f}$. Thus, the classes of $x$ and $y$ induce the same structure map on $C_\beta=P/\relInnLCS{\beta}{f}$, proving that $g_\beta$ is pre-covering.

At every nonzero limit ordinal $\beta$, insert $L_\beta$ immediately below $C_\beta$ with successor homomorphism $g_\beta$. Reindexing the resulting tower gives a transfinite cocomposition of homomorphisms satisfying the pre-covering condition, whose cocomposite is $f$.

\ref{item: transfinite cocomposite of covering condition} $\Rightarrow$ \ref{item: hypocentral right orthogonal}:
Each homomorphism occurring in the tower is right orthogonal to every strongly connected homomorphism by \cref{lemma: strongly connected homomorphisms and covering condition}. Since right orthogonality is preserved under transfinite cocompositions, the transfinite cocomposite is also right orthogonal to strongly connected homomorphisms.

\ref{item: hypocentral right orthogonal} $\Rightarrow$ \ref{item: homomorphism is hypocentral}:
Suppose that $f$ is right orthogonal to every strongly connected homomorphism. Considering the commutative diagram
\[\begin{tikzcd}
P \arrow[r, "\id"] \arrow[d, two heads, "\hypMap{f}"]
    & P \arrow[d, two heads, "f"]\\
P/\Hyp(f) \arrow[r, two heads, "\hypFac{f}"']
    & Q\rlap{,}
\end{tikzcd}\]
we have a filler $w\colon P/\Hyp(f) \to P$, since $\hypMap{f}$ is strongly connected (\cref{theorem: universal strongly connected-hypocentral factorization}). Then, $w$ is a retraction of $\hypMap{f}$. Therefore $\hypMap{f}$ is injective, and hence an isomorphism. Thus $f\cong \hypFac{f}$ is hypocentral.
\end{proof}

In particular, every transfinite cocomposite of covering homomorphisms is hypocentral.

\begin{rem}
    A hypocentral homomorphism is a transfinite cocomposite of homomorphisms satisfying the pre-covering condition. However, the homomorphisms inserted at limit stages need not be surjective (see \cref{example: hypocentral but not nilpotent Alexander quandle}). Thus, a hypocentral homomorphism is not necessarily a transfinite cocomposite of covering homomorphisms.
\end{rem}

Even and Gran~\cite[Proposition 3.2]{even_gran_2014_on_factorization_systems_for_surjective_quandle_homomorphisms} previously showed that the classes of connected and rigid homomorphisms form an orthogonal factorization system relative to surjective homomorphisms (see also \cite[Proposition 2.24]{preprint_yuki_tomoki_2026_relativization_of_symmetries_on_quandles}). With the preceding preparations, we conclude that the analogous statement holds for strongly connected homomorphisms.
 
\begin{thm}[Strongly connected--hypocentral factorization system]
\label[thm]{theorem: strongly connected-hypocentral factorization system}
    The pair $(\{\mathrm{strongly\>connected}\}, \{\mathrm{hypocentral}\})$ of classes of surjections is an orthogonal factorization system for the surjections in $\Quandle$.
\end{thm}

\begin{proof}
Every surjective homomorphism admits the factorization in \cref{theorem: universal strongly connected-hypocentral factorization}, and the two classes are orthogonal by \cref{theorem: transfinite covering characterization of hypocentral homomorphisms}.
\end{proof}

\begin{rem}
    From \cite[Proposition 3.32]{preprint_yuki_tomoki_2026_relativization_of_symmetries_on_quandles}, strongly connected homomorphisms are closed under pullbacks along surjective homomorphisms.
    Therefore, this factorization system for the surjections is in fact \textit{stable} (see \cite[\href{https://ncatlab.org/nlab/show/stable+factorization+system}{stable factorization system}]{preprint_nlab_2008_nlab}).
\end{rem}

\begin{cor}
\label[cor]{corollary: hypocentral homomorphisms under pullback}
    Hypocentral homomorphisms are preserved by pullback along arbitrary quandle homomorphisms.
\end{cor}

\begin{proof}
By \cref{theorem: transfinite covering characterization of hypocentral homomorphisms}, hypocentral homomorphisms are right orthogonal to strongly connected homomorphisms in the category of quandles. Right orthogonality is preserved by pullback. Since a pullback of a surjective homomorphism is surjective, the assertion follows from the same theorem.
\end{proof}

\subsection{Finite-stage factorizations and examples}
\label{subsection: finite-stage strongly connected-nilpotent factorizations}

The strongly connected--hypocentral factorization has a nilpotent second homomorphism precisely when the relative transvection series stabilizes at a finite stage.

\begin{prop}
\label[prop]{proposition: criterion for strongly connected-nilpotent factorization}
    Let $f\colon P\twoheadrightarrow Q$ be a surjective quandle homomorphism. Then, the following conditions are equivalent.
\begin{enumerate}[label={(\roman*)}]
    \item\label{item: strongly connected-nilpotent factorization}
        The homomorphism $f$ admits a factorization $f=m\circ e$ where $e$ is strongly connected and $m$ is nilpotent.
    \item\label{item: canonical hypocentral factor is nilpotent}
        The canonical hypocentral factor $\hypFac{f}$ is nilpotent.
    \item\label{item: finite stabilization of relative transvection series}
        We have $\relTransLCS{c}{f}=\relTransLCS{c+1}{f}$ for some finite $c\geq0$.
\end{enumerate}
If these conditions hold, then $\hypCenter{f}=\relTransLCS{c}{f}$ for every sufficiently large integer $c$, and the canonical factorization is $P\twoheadrightarrow P/\relTransLCS{c}{f}\twoheadrightarrow Q$.
\end{prop}

\begin{proof}
\ref{item: finite stabilization of relative transvection series} $\Rightarrow$ \ref{item: canonical hypocentral factor is nilpotent}:
If $\relTransLCS{c}{f}=\relTransLCS{c+1}{f}$ for some integer $c\geq0$, then $\hypCenter{f}=\relTransLCS{c}{f}$. 
Therefore, $\hypFac{f}$ is in fact $\covTrunc{c}{f}\colon P/\relTransLCS{c}{f} \twoheadrightarrow Q$ in \cref{theorem: universal property of covering length truncation}, which is nilpotent by \cref{theorem: main structure theorem for nilpotent homomorphisms}.

\ref{item: canonical hypocentral factor is nilpotent} $\Rightarrow$ \ref{item: strongly connected-nilpotent factorization}: It follows from \cref{theorem: universal strongly connected-hypocentral factorization}.

\ref{item: strongly connected-nilpotent factorization} $\Rightarrow$ \ref{item: finite stabilization of relative transvection series}:
Let $f=m\circ e$ be the factorization. 
Since $e$ is strongly connected, $\relTrans(e)$ is $\Inn(P)$-perfect by \cref{proposition:relTrans_of_strongly_conncted_is_Inn-perfect}.
It is immediate to see $\relTrans(e) \subseteq \relTrans(f)\subseteq \relInn(f)$. Hence, $\relTrans(e) \subseteq \hypCenter{f}$ by \cref{lemma: relative hypocenter}; in particular, $\relTrans(e) \subseteq \relTransLCS{i}{f}$ for every $i\geq0$.

Since $m$ is nilpotent, there is an integer $c\geq0$ such that $\relTransLCS{c}{m}=1$. By \cref{lemma: relative symmetry groups under a surjective factorization} and \cref{proposition: functoriality of relative lower central series}, we have $e_*(\relTransLCS{c}{f})=\relTransLCS{c}{m}=1$, and hence $\relTransLCS{c}{f}\subseteq \relInn(e)$. Therefore, using \cref{lem:commutator_inn_rel_trans_rel}, we have
\[ \relTransLCS{c+1}{f} = [\Inn(P),\relTransLCS{c}{f}] \subseteq [\Inn(P),\relInn(e)] = \relTrans(e). \]
Together with $\relTrans(e)\subseteq \relTransLCS{c+1}{f}$, we obtain $\relTransLCS{c+1}{f}=\relTrans(e)$, which is $\Inn(P)$-perfect. Consequently, it holds that $\relTransLCS{c+1}{f}=[\Inn(P),\relTransLCS{c+1}{f}]=\relTransLCS{c+2}{f}$.
\end{proof}

\begin{cor}
\label[cor]{corollary: descending chain condition and finite-stage factorization}
    Assume that the $\Inn(P)$-normal subgroups of $\relTrans(f)$ satisfy the descending chain condition. Then, $f$ admits a factorization into a strongly connected homomorphism followed by a nilpotent homomorphism. Moreover, $f$ is hypocentral if and only if it is nilpotent. In particular, these assertions hold when $\relTrans(f)$ is finite.
\end{cor}

\begin{proof}
The assumption shows that the relative transvection series stabilizes at a finite stage, so the first assertion follows from \cref{proposition: criterion for strongly connected-nilpotent factorization}. If $f$ is hypocentral, its stable term is trivial, and hence the series vanishes at a finite stage. Thus, $f$ is nilpotent. The converse is trivial.
\end{proof}

\begin{cor}
\label[cor]{corollary: strongly connected-nilpotent factorization system for finite inner groups}
    On the full subcategory of $\Quandle$ consisting of quandles with finite inner automorphism groups, the classes of strongly connected and nilpotent homomorphisms form a (stable) orthogonal factorization system for surjections. In particular, the same holds for finite quandles.
\end{cor}

\begin{proof}
For every homomorphism in this category, the relative transvection group is finite, so hypocentrality and nilpotency coincide by \cref{corollary: descending chain condition and finite-stage factorization}. The canonical factorization remains in the subcategory, since $\Inn(P/\hypCenter{f})$ is a quotient of $\Inn(P)$. Thus, the orthogonal factorization system of \cref{theorem: strongly connected-hypocentral factorization system} restricts to this subcategory, and hypocentrality can be replaced by nilpotency.
(Moreover, the subcategory is closed under pullbacks: the inner automorphism group of a pullback is a quotient of a subgroup of the product of the inner automorphism groups of its two factors. Hence the resulting orthogonal factorization system is stable as well.)
\end{proof}

Finally, we observe that in general, the classes of nilpotent and hypocentral homomorphisms do not coincide. In particular, strongly connected and nilpotent homomorphisms do not form a factorization system on all quandles.

\begin{eg}
\label[eg]{example: hypocentral but not nilpotent Alexander quandle}
    Let $P\coloneqq\Alex(\ZZ,-\id)$, and let $f\colon P\twoheadrightarrow\{*\}$ be the unique homomorphism into the terminal quandle. 
    By \cref{example: relative Alexander lower central series}, we have
    \[
        \relTransLCS{i}{f}=\{L_{2^{i+1}n}\mid n\in\ZZ\} \qquad (i<\omega),
    \]
    and hence $\relTransLCS{\omega}{f}=1$. Thus, $f$ is hypocentral but not nilpotent. 
    Moreover, $f$ is not a transfinite cocomposite of covering homomorphisms.

    Since $s_x(0)=2x$ for every $x\in P$, the quandle $P$ is faithful, that is, $s_x= s_y$ implies $x=y$. Thus, every covering homomorphism with domain $P$ is an isomorphism. Moreover, every non-injective surjective homomorphism $q\colon P\twoheadrightarrow X$ has finite target and infinite fibers. Indeed, if $q(a)=q(b)$ with $a\neq b$, then, for every $z\in\ZZ$, we have
    \[
        q(z+2(a-b))=q(s_as_b^{-1}(z))
        =s_{q(a)}s_{q(b)}^{-1}(q(z))=q(z).
    \]
    Hence, $q$ has a nonzero period, so $|X|\leq2|a-b|$ and every fiber of $q$ is infinite.

    Suppose that $f$ is the cocomposite of a continuous tower $(X_\alpha)_{\alpha\leq\lambda}$ with $X_0=\{*\}$, $X_\lambda=P$, and every successor homomorphism $X_{\alpha+1}\twoheadrightarrow X_\alpha$ covering. Write $q_\alpha\colon P\to X_\alpha$ for the canonical homomorphisms. Successive lifting, using continuity at limit ordinals, shows that every $q_\alpha$ is surjective.

    Let $\beta$ be the least ordinal for which $q_\beta$ is an isomorphism. Since $P\neq\{*\}$, we have $\beta>0$. If $\beta=\alpha+1$, then $q_\alpha$ is a covering with domain $P$, and hence an isomorphism, contradicting the minimality of $\beta$. Thus, $\beta$ is a limit ordinal, and
    \[
        P\cong X_\beta\cong\varprojlim_{\alpha<\beta}X_\alpha.
    \]
    For every $\alpha<\beta$, the quandle $X_\alpha$ is finite and every fiber of $q_\alpha$ is infinite.

    Give each $X_\alpha$ the discrete topology and $P$ the resulting inverse limit topology. Then, $P$ is a nonempty countable compact Hausdorff space. Since the index set is linearly ordered, the fibers of the $q_\alpha$ form a basis of open sets. These fibers are all infinite, so $P$ has no isolated points. This contradicts the Baire category theorem, since every nonempty countable compact Hausdorff space has an isolated point.
\end{eg}

\section{Reduced homomorphisms and nilpotency ascent}
\label{section: relative reduced homomorphisms}

In this section, we introduce reduced and operator-reduced surjective homomorphisms and establish the nilpotency ascent theorem.

\subsection{Reduced and operator-reduced homomorphisms}
\label{subsection: definition of relative reduced homomorphisms}

Following \cite[Definition 3.1]{inoue_2013_quasitriviality_of_quandles_for_linkhomotopy} and \cite[Definition 7.1]{darne_2026_nilpotent_quandles} (cf.~\cite[Definition 4.1]{hughes_2011_link_homotopy_invariant_quandles}), a quandle $P$ is called \emph{reduced} if $s_{\varphi(x)}(x)=x$ for every $x\in P$ and every $\varphi\in\Inn(P)$.

We introduce a relative version of reduced quandles, as well as a slightly more general notion called \emph{operator-reduced} (cf.~\cite[Section 5]{hughes_2011_link_homotopy_invariant_quandles}).

\begin{dfn}
\label[dfn]{definition: reduced hom and operator reduced hom}
A surjective quandle homomorphism $f\colon P\twoheadrightarrow Q$ is called
\begin{enumerate}
    \item \emph{reduced} if $s_{\varphi(x)}(x)=x$, and
    \item \emph{operator-reduced} if $s_xs_{\varphi(x)}=s_{\varphi(x)}s_x$
\end{enumerate}
for every $x\in P$ and every $\varphi\in\relInn(f)$.
\end{dfn}

For the terminal homomorphism $P\twoheadrightarrow\{*\}$, the first condition is precisely the notion of a reduced quandle. Since $\relInn(f)\subseteq\Inn(P)$, every surjective homomorphism with reduced domain is reduced.

The two conditions are related to the first relative commutators.

\begin{prop}
\label[prop]{proposition: commutator interpretation of reducedness}
Let $f\colon P\twoheadrightarrow Q$ be a surjective quandle homomorphism.
For $x\in P$ and $\varphi\in\relInn(f)$, it holds that $s_{\varphi(x)}=[\varphi,s_x]s_x$ and $[s_x,s_{\varphi(x)}]=[[\varphi,s_x],s_x]^{-1}$.
Consequently, the following hold.
\begin{enumerate}
    \item\label{item: 1 nil} $f$ is $1$-nilpotent if and only if $[\varphi,s_x]=1$ for every $x$ and $\varphi$.
    \item\label{item: reduced} $f$ is reduced if and only if $[\varphi,s_x](x)=x$ for every $x$ and $\varphi$.
    \item\label{item: op reduced} $f$ is operator-reduced if and only if $[[\varphi,s_x],s_x]=1$ for every $x$ and $\varphi$.
\end{enumerate}
In particular, we have the implications $1$-nilpotent $\Rightarrow$ reduced $\Rightarrow$ operator-reduced and $2$-nilpotent $\Rightarrow$ operator-reduced.
\end{prop}

\begin{proof}
Since $\varphi\in\Inn(P)$, we have $s_{\varphi(x)}=\varphi s_x\varphi^{-1}$. Hence, we have $s_{\varphi(x)}=[\varphi,s_x]s_x$ and $[s_x,s_{\varphi(x)}]=[[\varphi,s_x],s_x]^{-1}$.
For \eqref{item: 1 nil}, since the symmetry maps generate $\Inn(P)$, the condition is equivalent to $[\Inn(P),\relInn(f)]=1$. For \eqref{item: reduced}, the first identity and $s_x(x)=x$ give $s_{\varphi(x)}(x)=[\varphi,s_x](x)$.
The equivalence in \eqref{item: op reduced} follows from the second identity.

For the last implication, since $[[\varphi,s_x],s_x]\in\relInnLCS{2}{f}$, \eqref{item: op reduced} shows that $2$-nilpotent homomorphisms are operator-reduced.
\end{proof}

\begin{prop}
\label[prop]{proposition: reduced Alexander homomorphisms}
In the Alexander setting of
\cref{example: relative Alexander lower central series}, the following
hold.
\begin{enumerate}
    \item\label{item: reduced Alexander homomorphisms}
    The homomorphism $f$ is reduced if and only if
    $(1-\sigma)D=0$, or equivalently, if and only if $f$ is
    $1$-nilpotent.
    \item\label{item: operator-reduced Alexander homomorphisms}
    The homomorphism $f$ is operator-reduced if and only if
    $(1-\sigma)^2D=0$, or equivalently, if and only if $f$ is
    $2$-nilpotent.
\end{enumerate}
\end{prop}

\begin{proof}
As observed in the proof of
\cref{example: relative Alexander lower central series},
we have $\varphi(x)-x\in D$ for every $x\in A$ and
$\varphi\in\relInn(f)$. On the other hand,
$L(D)\subseteq\relInn(f)$ by \cref{example: Alexander quandle}.
Hence,
\(
    \relInn(f)\cdot x=x+D
\)
for every $x\in A$.
Therefore, $f$ is reduced if and only if
$s_{x+d}(x)=x$ for every $x\in A$ and $d\in D$.
Since
\(
    s_{x+d}(x)=x+(1-\sigma)d,
\)
this is equivalent to $(1-\sigma)D=0$.
By \cref{example: relative Alexander lower central series},
this is further equivalent to $\Gamma_1(f)=1$, that is, to
$f$ being $1$-nilpotent.

Similarly, $f$ is operator-reduced if and only if
$s_x$ and $s_{x+d}$ commute for every $x\in A$ and $d\in D$.
Since
\(
    s_{x+d}=L_{(1-\sigma)d}s_x
\)
and conjugation by $s_x$ sends $L_u$ to $L_{\sigma(u)}$,
these two symmetry maps commute if and only if
\(
    \sigma((1-\sigma)d)=(1-\sigma)d,
\)
or equivalently, $(1-\sigma)^2d=0$.
Thus, $f$ is operator-reduced if and only if
$(1-\sigma)^2D=0$. Again by
\cref{example: relative Alexander lower central series},
this is equivalent to $\Gamma_2(f)=1$, that is, to
$f$ being $2$-nilpotent.
\end{proof}

\begin{prop}
\label[prop]{proposition: characterizations of reduced homomorphisms}
For a surjective quandle homomorphism $f\colon P\twoheadrightarrow Q$, the following conditions are equivalent.
\begin{enumerate}[label={\textup{(\roman*)}}]
    \item\label{item: f is reduced} $f$ is reduced.
    \item\label{item: orbit is trivial} Every $\relInn(f)$-orbit in $P$ is a trivial subquandle.
\end{enumerate}
In particular, if every fiber of $f$ is trivial, seen as a subquandle of $P$, then $f$ is reduced. The converse holds when $f$ is connected.
\end{prop}

\begin{proof}
\ref{item: orbit is trivial} $\Rightarrow$ \ref{item: f is reduced}: Consider an orbit $P'\coloneqq \relInn(f)\cdot x \subseteq P$. Since it is a trivial subquandle, $s_{\varphi(x)}\rvert_{P'}=\id_{P'}$ for all $\varphi\in \relInn(f)$. In particular, $s_{\varphi(x)}(x)=\id(x)=x$, and hence $f$ is reduced.
\ref{item: f is reduced}$\Rightarrow$ \ref{item: orbit is trivial}: if we take $y=\varphi(x)$ and $z=\psi(x)$ with $\varphi,\psi\in\relInn(f)$ from the orbit $\relInn(f)\cdot x$, then $s_y(z)=\psi(s_{\psi^{-1}\varphi(x)}(x))=\psi(x)=z$, which shows that $\relInn(f)\cdot x$ is a trivial subquandle.

The last assertions follow because every $\relInn(f)$-orbit is contained in a fiber, and the two coincide when $f$ is connected.
\end{proof}

In particular, covering homomorphisms are reduced, since we can easily verify that every fiber of a covering homomorphism is a trivial subquandle.

\begin{prop}
\label[prop]{proposition: first properties of reduced homomorphisms}
Let $f\colon P\to Q$, $f'\colon P'\to Q'$, and $g\colon Q\to R$ be surjective quandle homomorphisms. Put $h\coloneqq g\circ f$.
\begin{enumerate}
    \item\label{item: (operator-)reduced descent} If $h$ is (operator-)reduced, then both $f$ and $g$ are (operator-)reduced.
    \item\label{item: product of (operator-)reduced homs} If $f$ and $f'$ are (operator-)reduced, then $f\times f'$ is (operator-)reduced.
    \item\label{item: (operator-)reduced pullback} If $f$ is (operator-)reduced, then every pullback of $f$ along a surjective quandle homomorphism is (operator-)reduced.
    \item\label{item: connected (operator-)reduced pullback} If $f$ is connected and (operator-)reduced, then every pullback of $f$ along an arbitrary quandle homomorphism is (operator-)reduced.
\end{enumerate}
\end{prop}

\begin{proof}
\eqref{item: (operator-)reduced descent}:
Since $\relInn(f)\subseteq\relInn(h)$, if $h$ is reduced or operator-reduced, then so is $f$. Let $y\in Q$ and $\psi\in\relInn(g)$. Choose $x\in P$ with $f(x)=y$. By \cref{lemma: relative symmetry groups under a surjective factorization}, there is $\varphi\in\relInn(h)$ such that $f_*(\varphi)=\psi$.
If $h$ is reduced, then $s_{\psi(y)}(y) = f\bigl(s_{\varphi(x)}(x)\bigr) = f(x) = y$, so $g$ is reduced.
If $h$ is operator-reduced, then $[s_y,s_{\psi(y)}] = f_*([s_x,s_{\varphi(x)}])=1$. Hence, $g$ is operator-reduced.

\eqref{item: product of (operator-)reduced homs}:
Let $z=(x,x')\in P\times P'$ and $\varphi\in\relInn(f\times f')$. Through the injection $\Inn(P\times P') \hookrightarrow \Inn(P) \times \Inn(P')$ from \cref{proposition: inner automorphism groups of products}, we identify $\varphi=(\varphi_1,\varphi_2)$, where $\varphi_1\in\relInn(f)$ and $\varphi_2\in\relInn(f')$ are its projections. Hence we have $s_{\varphi(z)}=(s_{\varphi_1(x)},s_{\varphi_2(x')})$.
If $f$ and $f'$ are reduced, then we have $s_{\varphi(z)}(z)= \bigl(s_{\varphi_1(x)}(x), s_{\varphi_2(x')}(x')\bigr)=(x,x')=z$, so $f\times f'$ is reduced.
If $f$ and $f'$ are operator-reduced, then we have $s_zs_{\varphi(z)}=(s_x,s_{x'})(s_{\varphi_1(x)},s_{\varphi_2(x')}) = (s_{\varphi_1(x)},s_{\varphi_2(x')})(s_x,s_{x'}) = s_{\varphi(z)}s_z$, so $f\times f'$ is operator-reduced.

\eqref{item: (operator-)reduced pullback}:
Consider a pullback square of $f$ along a surjective homomorphism:
\[\begin{tikzcd}
P' \arrow[r, two heads, "u"] \arrow[d, two heads, "f'"']
    & P \arrow[d, two heads, "f"]\\
Q' \arrow[r, two heads, "v"']
    & Q\rlap{.} \arrow[phantom]{ul}[very near end]{\lrcorner}
\end{tikzcd}\]
Let $z\in P'$ and $\psi\in\relInn(f')$. Then $u_*(\psi)\in\relInn(f)$. Moreover $f'_*(\psi)=1$ gives $f's_{\psi(z)}=s_{f'\psi(z)}f'=s_{f'(z)} f'$.
If $f$ is reduced, then we have
\[ u\bigl(s_{\psi(z)}(z)\bigr) = s_{u\psi(z)}(u(z)) = s_{u_*(\psi)(u(z))}(u(z)) = u(z), \]
as well as $f'(s_{\psi(z)}(z)) = s_{f'(z)}(f'(z)) = f'(z)$. Since $(u,f')$ is injective by the pullback property, $s_{\psi(z)}(z)=z$, so $f'$ is reduced.
If $f$ is operator-reduced, then we have
\[ us_zs_{\psi(z)} = s_{u(z)}s_{u\psi(z)} u = s_{u(z)}s_{u_*(\psi)(u(z))}u = s_{u_*(\psi)(u(z))}s_{u(z)}u = u s_{\psi(z)} s_z, \]
as well as $f's_zs_{\psi(z)} = s_{f'(z)}s_{f'(z)} f' = f's_{\psi(z)} s_z$. Hence the universality of pullback yields $s_zs_{\psi(z)}= s_{\psi(z)} s_z$, so $f'$ is operator-reduced.

\eqref{item: connected (operator-)reduced pullback}:
Suppose first that $f$ is reduced. By
\cref{proposition: characterizations of reduced homomorphisms},
every fiber of $f$ is trivial. Every fiber of a pullback of $f$
is isomorphic to a fiber of $f$, and hence is trivial. The assertion
follows from the same proposition.

Suppose that $f$ is operator-reduced.
If $x,y\in P$ satisfy $f(x)=f(y)$, then the connectedness of $f$
gives some $\varphi\in\relInn(f)$ such that $y=\varphi(x)$.
Hence $[s_x,s_y]=1$.
Consider an arbitrary pullback $f'\colon P'\to Q'$ of $f$.
For $z=(x,q')\in P'$ and $\psi\in\relInn(f')$, write
$\psi(z)=(y,q')$. Then $f(x)=f(y)$, and hence $[s_x,s_y]=1$.
Since $s_z=(s_x,s_{q'})|_{P'}$ and $s_{\psi(z)}=(s_y,s_{q'})|_{P'}$,
we obtain $[s_z,s_{\psi(z)}]=1$. Thus, $f'$ is operator-reduced.
\end{proof}

\subsection{Universal (operator-)reduced factor}
\label{subsection: universal relative reduced quotient}

Let $f\colon P\twoheadrightarrow Q$ be a surjective quandle homomorphism. 
Let $\redCong{f}$ be the quandle congruence generated by $s_{\varphi(x)}(x)\mathrel{\redCong{f}}x$ for $x\in P$ and $\varphi\in\relInn(f)$. 
Let $\opredDef{f}$ be the normal subgroup of $\Inn(P)$ generated by $[s_x,s_{\varphi(x)}]$ for $x\in P$ and $\varphi\in\relInn(f)$.
The first relation lies in the kernel congruence $\Eq(f)$, and \cref{proposition: commutator interpretation of reducedness} gives $\opredDef{f}\subseteq\relInnLCS{2}{f}\subseteq \relInn(f)$.

\begin{prop}
\label[prop]{proposition: universal relative reduced and operator reduced quotient}
    Let $f\colon P\to Q$ be a surjective quandle homomorphism.
    \begin{enumerate}
        \item Put $\redQuot{f}\coloneqq P/{\redCong{f}}$ and let $\redMap{f}\colon P\twoheadrightarrow\redQuot{f}$ be the quotient map.
        Then, the induced homomorphism $\redFac{f}\colon\redQuot{f}\twoheadrightarrow Q$ is reduced.
        
        Moreover, the factorization $f= \redFac{f}\circ \redMap{f}$ is universal among factorizations of $f$ whose second homomorphism is reduced.
        That is, if $f=h\circ u$ with $u\colon P\twoheadrightarrow R$ surjective and $h\colon R\twoheadrightarrow Q$ reduced, then $u$ factors uniquely through $\redQuot{f}$.

        \item Put $\opredQuot{f}\coloneqq P/\opredDef{f}$ and let $\opredMap{f}\colon P\twoheadrightarrow\opredQuot{f}$ be the quotient map.
        Then, the induced homomorphism $\opredFac{f}\colon\opredQuot{f}\twoheadrightarrow Q$ is operator-reduced.

        Moreover, the factorization $f= \opredFac{f}\circ \opredMap{f}$ is universal among factorizations of $f$ whose second homomorphism is operator-reduced.
        That is, if $f=h\circ u$ with $u\colon P\twoheadrightarrow R$ surjective and $h\colon R\twoheadrightarrow Q$ operator-reduced, then $u$ factors uniquely through $\opredQuot{f}$.
    \end{enumerate}
    \[\begin{tikzcd}[column sep=small]
        P \arrow[two heads]{rr}{f} \arrow[two heads]{rd}[swap]{\redMap{f}} & & Q\rlap{,} \\
        & \redQuot{f} \arrow[two heads]{ru}[swap]{\redFac{f}}
    \end{tikzcd}\qquad
    \begin{tikzcd}[column sep=small]
        P \arrow[two heads]{rr}{f} \arrow[two heads]{rd}[swap]{\opredMap{f}} & & Q\rlap{.} \\
        & \opredQuot{f} \arrow[two heads]{ru}[swap]{\opredFac{f}}
    \end{tikzcd}\]
\end{prop}

\begin{proof}
    Let $\pi=\redMap{f}$ and put $\bar f=\redFac{f}$. By \cref{lemma: relative symmetry groups under a surjective factorization}, every $\psi\in\relInn(\bar f)$ has the form $\pi_*(\varphi)$ for some $\varphi\in\relInn(f)$. Hence, $s_{\psi([x])}([x])=[s_{\varphi(x)}(x)]=[x]$ by the definition of $\redCong{f}$. Thus, $\redFac{f}$ is reduced.

    The same argument for $\pi=\opredMap{f}$ gives $[s_{[x]},s_{\psi([x])}] =\pi_*([s_x,s_{\varphi(x)}])=1$, since $[s_x,s_{\varphi(x)}]\in\opredDef{f}$. Hence, $\opredFac{f}$ is operator-reduced.

    By \cref{lemma: relative symmetry groups under a surjective factorization}, the defining relations of $\redCong{f}$ and the generators of $\opredDef{f}$ vanish in the corresponding quotient factors. The same lemma shows that they vanish in every reduced, respectively operator-reduced, quotient factorization of $f$. This proves the two universal properties. 
\end{proof}

\begin{dfn}
\label[dfn]{definition: reduced and operator-reduced factor}
    For a surjective quandle homomorphism $f\colon P \twoheadrightarrow Q$, the induced homomorphisms $\redFac{f}\colon\redQuot{f}\twoheadrightarrow Q$ and $\opredFac{f}\colon\opredQuot{f}\twoheadrightarrow Q$ are called the \emph{reduced factor} and the \emph{operator-reduced factor} of $f$, respectively.
\end{dfn}

\begin{lem}
\label[lem]{proposition: reduction preserves connected components}
    Let $f\colon P \twoheadrightarrow Q$ be a surjective quandle homomorphism.
    The quotient map $\redMap{f}\colon P\to\redQuot{f}$ induces a bijection $\pi_0(P)\to\pi_0(\redQuot{f})$.
\end{lem}

\begin{proof}
Since $\redMap{f}$ is surjective, $\pi_0(\redMap{f})$ is surjective.
Each defining pair of $\redCong{f}$ has the same image in $\pi_0(P)$. Hence, the canonical homomorphism $P\twoheadrightarrow\pi_0(P)$ factors through $\redMap{f}$ as
\[\begin{tikzcd}
    P \arrow[two heads]{r} \arrow[two heads]{d}[swap]{\redMap{f}} & \pi_0(P)\rlap{.} \\
    \redQuot{f} \arrow[dashed]{ru}[swap]{h}
\end{tikzcd}\]
Moreover, as $\pi_0(P)$ is a trivial quandle, $h$ factors through the canonical map $\redQuot{f} \to \pi_0(\redQuot{f})$. Writing $\overline{h}\colon \pi_0(\redQuot{f}) \to \pi_0(P)$ for the induced homomorphism, we have $\overline{h}\circ \pi_0(\redMap{f}) = \id$ by the universality, which shows that $\pi_0(\redMap{f})$ is also injective.
Thus, it is bijective.
\end{proof}

We conclude this subsection with further properties of the reduced factor. First, the reduced congruence has the following orbit description. Put $\redOpGrp_f(x)\coloneqq \langle s_{\varphi(x)}\mid\varphi\in\relInn(f)\rangle$.

\begin{prop}
\label[prop]{proposition: orbit description of relative reduction}
Let $f\colon P\twoheadrightarrow Q$ be an operator-reduced quandle homomorphism.
Then, $\redOpGrp_f(x)$ is abelian for every $x\in P$, and the fibers
of $\redMap{f}\colon P\twoheadrightarrow\redQuot{f}$ are the orbits
$\redOpGrp_f(x)\cdot x$.
\end{prop}

\begin{proof}
For $\varphi,\psi\in\relInn(f)$, operator-reducedness gives
\[ [s_{\varphi(x)},s_{\psi(x)}] = \varphi[s_x,s_{\varphi^{-1}\psi(x)}]\varphi^{-1} = 1. \]
Thus, $\redOpGrp_f(x)$ is abelian.
We note that $\theta \redOpGrp_f(x) \theta^{-1}= \redOpGrp_f(\theta(x))$ for $\theta\in \Inn(P)$, because $\theta s_{\varphi(x)} \theta^{-1} = s_{\theta\varphi(x)} = s_{\theta\varphi\theta^{-1}(\theta x)}$.

We define $x \sim_{\redOpGrp_f} y$ precisely when $y \in \redOpGrp_f(x)\cdot x$. We next want to show that the relations $\sim_{\redOpGrp_f}$ and $\redCong{f}$ coincide.

We observe that $\sim_{\redOpGrp_f}$ is a quandle congruence. Since $s_x\in \redOpGrp_f(x)$, we have $x\in \redOpGrp_f(x)\cdot x$, and hence $x \sim_{\redOpGrp_f} x$.
If $x\sim_{\redOpGrp_f} y$, then $y=b(x)$ for some $b \in \redOpGrp_f(x)$. As $\redOpGrp_f(x)$ is abelian, we have $\redOpGrp_f(y) = b\redOpGrp_f(x) b^{-1} = \redOpGrp_f(x) bb^{-1}=\redOpGrp_f(x)$. Hence, $x=b^{-1}(y)$ and $b^{-1}\in \redOpGrp_f(y)$ imply $x \in \redOpGrp_f(y)\cdot y$, which means $y \sim_{\redOpGrp_f} x$.
If $x \sim_{\redOpGrp_f} y$ and $y \sim_{\redOpGrp_f} z$, then a similar argument yields that $\redOpGrp_f(x)= \redOpGrp_f(y) = \redOpGrp_f(z)$ and $x\sim_{\redOpGrp_f}z$. Thus, $\sim_{\redOpGrp_f}$ is an equivalence relation.
Moreover, if $x \sim_{\redOpGrp_f} x'$ and $y \sim_{\redOpGrp_f} y'$, then we see $s_x=s_{x'}$. Therefore, as $\sim_{\redOpGrp_f}$ is preserved by $\Inn(P)$, we obtain $s_x^\varepsilon(y)\sim s_x^\varepsilon(y') =s_{x'}^\varepsilon(y')$ for $\varepsilon\in\{\pm 1\}$. Hence, $\sim_{\redOpGrp_f}$ is indeed a quandle congruence.

For each defining pair $x \redCong{f} s_{\varphi(x)}(x)$, we have $x \sim_{\redOpGrp_f} s_{\varphi(x)}(x)$ since $s_{\varphi(x)}(x) \in \redOpGrp_f(x)\cdot x$. Thus, $x\redCong{f}y$ implies $x\sim_{\redOpGrp_f} y$.
Conversely, if $y=s_{\varphi(x)}(x)$ for some $\varphi \in\relInn(f)$, then $x\redCong{f} s_{\varphi(x)}(x)=y$. Since the operators $s_{\varphi(x)}$ generate $\redOpGrp_f(x)$, we see that if $x\sim_{\redOpGrp_f} y$, then $x\redCong{f}y$.
Thus, the two congruences coincide.
\end{proof}

\subsection{Nilpotency ascent and commutative action criterion}
\label{subsection: quadratic relative actions and nilpotency ascent}

In this subsection, we study sufficient conditions for operator-reduced homomorphisms to be nilpotent.

The operator-reduced condition controls the action on abelian sections of the relative inner automorphism group. Using this, we first prove that an operator-reduced homomorphism with a finite relative inner automorphism group is nilpotent whenever its target is nilpotent.

For this, we need to make an observation on operator-reduced homomorphisms. 
Recall that the \emph{Fitting subgroup} $\Fit(G)$ of a group $G$ is the subgroup generated by all normal nilpotent subgroups of $G$. When $G$ is finite, $\Fit(G)$ is a finite nilpotent normal subgroup of $G$.

\begin{prop}
\label[prop]{proposition: fitting control for operator reduced homomorphisms}
Let $f\colon P\twoheadrightarrow Q$ be an operator-reduced quandle homomorphism, and let $M\subseteq \Inn(P)$ be a finite normal subgroup with $M\subseteq \relInn(f)$. Then, $[\Inn(P), M]\subseteq\Fit(M)$.
%Equivalently, $M/\Fit(M)\subseteq Z(G/\Fit(M))$. In particular, $M$ is solvable.
\end{prop}

\begin{proof}
Let $x\in P$, and put $B_x\coloneqq\redOpGrp_f(x)\cap M$, where $\redOpGrp_f(x)$ is the abelian group defined in \cref{proposition: orbit description of relative reduction}. Since $M\subseteq \relInn(f)$, we have $m\redOpGrp_f(x)m^{-1} =\redOpGrp_f(m(x))=\redOpGrp_f(x)$ for every $m\in M$. Thus, $B_x$ is an abelian normal subgroup of $M$, and hence $B_x\subseteq\Fit(M)$.

For $m\in M$, we see $[m,s_x]=s_{m(x)}s_x^{-1}\in B_x\subseteq \Fit(M)$. 
Since the symmetry maps $s_x$ generate $\Inn(P)$ and the Fitting group $\Fit(M)$ is normal in $\Inn(P)$, it follows that $[M, \Inn(P)]=[\Inn(P), M]\subseteq\Fit(M)$.
%Finally, $\Fit(M)$ is nilpotent and $M/\Fit(M)$ is abelian, so $M$ is solvable.
\end{proof}

The following elementary group-theoretic observation will be useful.

\begin{lem}
\label[lem]{lemma: quadratic action on abelian sections}
    Let $M\subseteq M'$ be normal subgroups of a group $G$ such that $A\coloneqq M'/M$ is abelian. Let $s\in G$ be an element satisfying $[[x,s],s]=1$ for all $x\in M'$. Then, letting $\alpha(s)\in\Aut(A)$ be induced by conjugation by $s$, we have $(\alpha(s)-\id_A)^2=0$.
\end{lem}

\begin{proof}
Let $x\in M'$, and write $\overline{x}=xM\in A$. Writing $A$ additively, the image of $[x,s]=xsx^{-1}s^{-1}$ in $A$ is $\overline{x}-\alpha(s)(\overline{x})=(\id_A-\alpha(s))(\overline{x})$.
Therefore $[[x,s],s]=1$ yields $(\id_A-\alpha(s))^2(\overline{x})=0$ for every $x\in M'$. Hence we have $(\alpha(s)-\id_A)^2=0$ in $\End(A)$.
\end{proof}

\begin{lem}[{\cite[Lemma 3.11]{book_isaacs_2008_finite_group_theory}}]
\label{fact:finite_solvable_minimal_normal_subgroup_is_elementary_abelian_p-group}
    Let $M$ be a solvable minimal normal subgroup of an arbitrary group $G$.
    Then $M$ is abelian, and if $M$ is finite, it is an elementary abelian $p$-group for some prime $p$.
\end{lem}

\begin{lem}% standard fact
\label{lemma:finite_nilpotent_generated_by_p-elements_is_p-group}
    A finite nilpotent group generated by $p$-elements is a $p$-group.
\end{lem}

\begin{proof}
    In a finite nilpotent group, every Sylow $p$-group is normal, and hence unique. Since $p$-elements are contained in the unique Sylow $p$-group, the generating condition implies that the finite nilpotent group itself is a $p$-group.
\end{proof}

\begin{lem}
\label{lemma:invariant_space_of_finite_vector_space_with_p-group_action_is_nontrivial}
    Let $H$ be a finite $p$-group, and let $V$ be a nonzero finite-dimensional vector space over $\mathbb{F}_{p}$.
    Suppose that $H$ acts linearly on $V$.
    Then the fixed-point subspace $V^H=\{v \in V \mid h\cdot v=v \text{ for all } h\in H \}$ is nonzero.
\end{lem}

\begin{proof}
    For each $v \in V$, the orbit $H\cdot v$ has cardinality dividing $\lvert H \rvert$. Since $H$ is a $p$-group, $\lvert H\cdot v\rvert$ is a power of $p$. In particular, every nontrivial orbit has cardinality divisible by $p$. The elements of $V^H$ are precisely those lying in orbits of cardinality $1$. Hence, by the orbit-counting formula,
    \[ \lvert V \rvert = \lvert V^H \rvert + (\text{a multiple of }p). \]
    Since $\lvert V\rvert$ is also divisible by $p$, we obtain $\lvert V^H \rvert \equiv 0 \mod p$. As $0\in V^H$, we have $\lvert V^H \rvert \geq 1$, and hence in fact $\lvert V^H \rvert \geq p$. Thus $V^H\neq 0$. 
\end{proof}

Recall that a \emph{$G$-chief series of $N$} for a normal subgroup $N\trianglelefteq G$ is a finite chain $1= N_{0}< N_{1}<\cdots< N_{m}=N$ of normal subgroups of $G$ such that $N_i/N_{i-1}$ is a nontrivial minimal normal subgroup of $G/N_{i-1}$ for every $1\leq i \leq m$.

\begin{thm}
\label{theorem:nilpotency ascent theorem for groups}
    Let $G$ be a group, $N\leq G$ be a finite normal subgroup, and $F\leq G$ be a finite nilpotent normal subgroup with $[G,N]\subseteq F\subseteq N$.
    Suppose that $G$ is generated by a set $S\subseteq G$ and that $[[x,s],s]=1$ for every $s\in S$ and $x\in N$.
    If $G/N$ is nilpotent, then $G$ is nilpotent.
    More precisely, if $G/N$ is nilpotent of class at most $r$ and $N$ has a $G$-chief series of length $m$, then $N$ is $G$-nilpotent of class at most $m$ and $G$ is nilpotent of class at most $m+r$.
\end{thm}

\begin{proof}
    \textbf{Step 1:}
    % N: finite solvable and G/F: nilpotent
    By assumption, $[N,N] \subseteq [G,N] \subseteq F \subseteq N$, so $N/F$ is abelian. Since $F$ is nilpotent, it is solvable, and hence $N$ is solvable. Moreover, $[G,N] \subseteq F$ implies $N/F \subseteq Z(G/F)$. Since $(G/F)/(N/F)\cong G/N$ is nilpotent, $G/F$ is nilpotent.

    \textbf{Step 2:}
    % Each V_i=N_i/N_{i-1} is a vec sp with linear action of $\alpha_i(G)$
    Consider the upper central series $1=Z_0(F)\leq Z_1(F)\leq\cdots\leq Z_c(F)=F$ of $F$. Since $F\trianglelefteq G$, each $Z_j(F)$ is normal in $G$. Thus
    \begin{equation}
        1=Z_0(F) \leq Z_1(F) \leq \cdots \leq Z_c(F)=F \leq N \label{equation:G-normal_series_of_N}
    \end{equation}
    is a $G$-normal series for $N$. Since $N$ is finite, this series can be refined to a $G$-chief series
    \[ 1=N_0 < N_1 < \cdots < N_m=N. \]
    For each \(1\le i\le m\), put $V_i=N_i/N_{i-1}$. Since $N$ is finite and solvable, each $V_i$ is finite and solvable. Moreover, $V_i$ is a minimal normal subgroup of $G/N_{i-1}$, so by \cref{fact:finite_solvable_minimal_normal_subgroup_is_elementary_abelian_p-group}, $V_i$ is an elementary abelian $p_i$-group for some prime $p_i$. Thus $V_i$ is a finite-dimensional vector space over $\mathbb{F}_{p_i}$, and $\Aut(V_i)=\GL_{\mathbb{F}_{p_i}}(V_i)$. 
    The conjugation action of $G$ on $V_i$ gives a homomorphism
    \[ \alpha_i\colon G\longrightarrow \Aut(V_i), \quad \alpha_i(g)(v)= gvg^{-1}. \]
    Hence $V_i$ naturally admits a linear action of $\alpha_i(G)$.

    \textbf{Step 3:}
    % \alpha_i(F)=1
    Fix $1\leq i \leq m$. Since the chosen $G$-chief series refines the series \eqref{equation:G-normal_series_of_N} we have either $N_i \leq F$ or $F\leq N_{i-1}$.
    If $N_i \leq F$, then there is some $j$ such that $Z_{j-1}(F) \leq N_{i-1} \leq N_i \leq Z_j(F)$. Hence $[N_i, F] \leq [Z_j(F), F] \leq Z_{j-1}(F) \leq N_{i-1}$.
    If $F\leq N_{i-1}$, then $[N_i, F] \leq F \leq N_{i-1}$ since $F \trianglelefteq G$.
    Thus, in either case, we obtain $[N_i,F] \leq N_{i-1}$, which means $[F,V_i]=1$, or, equivalently, $\alpha_i(F)=1$.

    \textbf{Step 4:}
    % \alpha_i(G): finite nilpotent
    Since $\alpha_i(F)=1$, the homomorphism $\alpha_i\colon G \to \Aut(V_i)$ factors through $G/F$. In particular, there is a surjective homomorphism $G/F \twoheadrightarrow \alpha_i(G)$.
    By Step 1, $G/F$ is nilpotent, so $\alpha_i(G)$ is also nilpotent. Moreover, since $V_i$ is finite, so $\Aut(V_i)$, hence $\alpha_i(G)$, is finite. Thus $\alpha_i(G)$ is a finite nilpotent group.

    \textbf{Step 5:}
    % \alpha_i(G): finite p-group
    By assumption, $[[x,s],s]=1$ for $x\in N$ and $s\in S$. By passing to the factor $V_i$ and writing additively, we obtain $(\alpha_i(s)-1)^2=0$ in 
    $\End(V_i)$
    %$\Aut(V_i)=\GL_{\mathbb{F}_{p_i}}(V_i)$ 
    by \cref{lemma: quadratic action on abelian sections}.
    Putting $u=\alpha_i(s)-1$, we have $u^2=0$. Since $V_i$ is a vector space over $\mathbb{F}_{p_i}$, its endomorphism ring has characteristic $p_i$. Hence we have
    \[ (\alpha_i(s))^{p_i} = (u+1)^{p_i} = u^{p_i} + 1 =1, \]
    because $u^2=0$ implies $u^{p_i}=0$. Thus every $\alpha_i(s)$ is a $p_i$-element. Since $\alpha_i(G)$ is a finite nilpotent group generated by the $p_i$-elements $\alpha_i(s)$, it follows that $\alpha_i(G)$ is a $p_i$-group (\cref{lemma:finite_nilpotent_generated_by_p-elements_is_p-group}).

    \textbf{Step 6:}
    % \alpha_i(G)=1 
    Thus the finite $p_i$-group $\alpha_i(G)$ acts linearly on the nonzero finite-dimensional $\mathbb{F}_{p_i}$-vector space $V_i$. By \cref{lemma:invariant_space_of_finite_vector_space_with_p-group_action_is_nontrivial}, $V_i^{\alpha_i(G)}\neq 0$. Since the action is induced by conjugation, we have
    \[ V_i^{\alpha_i(G)}= \{v \in V_i \mid \alpha_i(g)(v)=gvg^{-1}=v \text{ for all } g\in G \} = Z(G/N_{i-1},V_i)\neq 0, \]
    which is a normal subgroup of $G/N_{i-1}$. By the minimality of $V_i$, we obtain $V_i^{\alpha_i(G)}=V_i$. This means $\alpha_i(G)=1$. 

    \textbf{Step 7:}
    % N: G-nilpotent and G: nilpotent
    Since $\alpha_i(G)=1$, we have $[G,N_i] \subseteq N_{i-1}$ for every $1\leq i \leq m$. The chosen $G$-chief series of $N$
    \[ 1=N_0 \leq N_1 \leq \cdots \leq N_m=N \]
    is in fact a $G$-central series, and hence $N$ is $G$-nilpotent. Finally, since $G/N$ is nilpotent, it follows that $G$ is nilpotent.
\end{proof}

\begin{thm}[Nilpotency ascent theorem]
\label[thm]{theorem: finite operator reduced ascent}
    Let $f\colon P\twoheadrightarrow Q$ be an operator-reduced quandle homomorphism such that $\relInn(f)$ is a finite group. If $Q$ is nilpotent, then both $f$ and $P$ are nilpotent. 
    More precisely, if $Q$ is $r$-nilpotent and $\relInn(f)$ has an $\Inn(P)$-chief series of length $m$, then $f$ is $m$-nilpotent and $P$ is $(m+r)$-nilpotent.
\end{thm}

\begin{proof}
This follows immediately from \cref{proposition: commutator interpretation of reducedness,proposition: fitting control for operator reduced homomorphisms} and \cref{theorem:nilpotency ascent theorem for groups}, by taking $G=\Inn(P)$, $N=\relInn(f)$, and $F=\Fit(N)$.
\end{proof}

\begin{cor}
\label[cor]{corollary: consequences of finite ascent}
    Suppose that $f$ is operator-reduced and that $\relInn(f)$ is finite. Then $P$ is nilpotent if and only if $Q$ is nilpotent. Under these equivalent conditions, $f$ is nilpotent. 
\end{cor}

\begin{proof}
If $Q$ is nilpotent, the assertion follows from \cref{theorem: finite operator reduced ascent}. Conversely, if $P$ is nilpotent, then $Q$ is nilpotent because $\Inn(Q)$ is a quotient of $\Inn(P)$. The last assertion follows from \cref{theorem: finite operator reduced ascent}.
\end{proof}

\begin{rem}
Note that if $f$ is reduced, or if $f$ is connected and every fiber of $f$ is trivial, then $f$ is operator-reduced. Thus, the same equivalence holds in these cases.
\end{rem}

The absolute theorem of Darn\'e is recovered without any additional work: if $P$ is finitely generated and reduced, then $P$ is nilpotent by \cite[Proposition~7.13]{darne_2026_nilpotent_quandles}, and hence every surjective homomorphism out of $P$ is nilpotent by \cref{corollary: relative nilpotency implies ordinary nilpotency}. The following finite example shows why no direct relative analogue can hold without a condition on the base.

\begin{eg}%[Necessity of the nilpotent-base hypothesis]
\label[eg]{example: finite connected reduced homomorphism not nilpotent}
Let $T_{n}$ be the subquandle of $\Conj(S_n)$ consisting of
the transpositions. The quotient
$S_4\twoheadrightarrow S_4/V_4\cong S_3$ induces a surjective
homomorphism
$f\colon T_{4}\twoheadrightarrow T_{3}$.
Since the transpositions generate $S_n$ and $Z(S_n)=1$ for $n=3,4$,
we have
$\Inn(T_{4})\cong S_4$ and
$\Inn(T_{3})\cong S_3$, and
$\relInn(f)=V_4$. The fibers are
$\{(12),(34)\}$, $\{(13),(24)\}$, and $\{(14),(23)\}$.
The group $V_4$ acts transitively on each fiber, so $f$ is connected.
Elements in the same fiber are equal or disjoint transpositions, and
therefore commute. Hence every fiber is trivial and $f$ is reduced.

On the other hand, $[S_4,V_4]=V_4$, so
$\relLCS_{i}(S_4,V_4)=V_4$ for every $i\geq0$. Thus $f$ is not
nilpotent.
\end{eg}

This example is finite, connected, has trivial fibers, and has an abelian
relative inner automorphism group. In particular, finite generation of
the domain, finite generation of the kernel congruence, or finite
normal generation of the relative inner automorphism group cannot
replace the nilpotency hypothesis on the target in
\cref{theorem: finite operator reduced ascent}. 
%The obstruction is visible in \cref{proposition: finite relative nilpotency via chief factors}: $V_4$ is a minimal nontrivial normal subgroup of $S_4$, hence a $S_4$-chief factor, but it is not central. Equivalently, the induced action of $S_4$ on $V_4$ is nontrivial.

The second result is the following commutative action criterion. Unlike the preceding ascent theorem, the following quantitative criterion does not require the relative inner automorphism group to be finite.

\begin{lem}
\label[lem]{lemma: abelian images of inner automorphism groups}
Let $P$ be a quandle satisfying either of the following conditions:
\begin{enumerate}[label={(\alph*)}]
    \item\label{item:Inn(P)_is_finitely_generated} $\Inn(P)$ is generated by finitely many symmetry maps. 
    \item\label{item:P_has_finite_connected_component} $P$ has finitely many connected components.
\end{enumerate}
Let $\beta\colon\Inn(P)\to H$ be a group homomorphism with abelian image.
Then, the image $\beta(\Inn(P))$ is generated by finitely many elements of the form $\beta(s_{x_1}),\ldots,\beta(s_{x_r})$.
\end{lem}

\begin{proof}
For the first case \ref{item:Inn(P)_is_finitely_generated}, if $\Inn(P)$ is generated by $s_{x_1},\dots,s_{x_r}$ for $x_i\in P$, then $\beta(\Inn(P))$ is generated by $\beta(s_{x_1}),\ldots,\beta(s_{x_r})$.

For the second case \ref{item:P_has_finite_connected_component}, suppose that $P$ has $r<\infty$ connected components and choose $x_1,\ldots,x_r\in P$ from each connected component.
For every $x\in P$, there are some $i\in \{1,\dots,r\}$ and $\varphi\in\Inn(P)$ such that $x=\varphi(x_i)$. Since $s_x=\varphi s_{x_i}\varphi^{-1}$ and $\beta(\Inn(P))$ is abelian, we have
\[ \beta(s_x) = \beta(\varphi)\beta(s_{x_i})\beta(\varphi)^{-1} = \beta(s_{x_i}). \]
Since $\Inn(P)$ is generated by all symmetry maps, its image under $\beta$ is generated by $\beta(s_{x_1}),\ldots,\beta(s_{x_r})$.
\end{proof}

\begin{thm}[Commutative action criterion]
\label[thm]{theorem: commutative action criterion}
    Let $f\colon P\twoheadrightarrow Q$ be an operator-reduced quandle homomorphism. 
    Assume that there is a finite sequence of $\Inn(P)$-normal subgroups $1=N_0\subseteq N_1\subseteq\cdots\subseteq N_d=\relInn(f)$ such that each $A_j=N_j/N_{j-1}$ is abelian and the image of the conjugation action $\alpha_j\colon \Inn(P)\to \Aut(A_j)$ is abelian and generated by finitely many elements of the form $\alpha_j(s_{x_1}),\dots,\alpha_j(s_{x_r})$ for $x_i\in P$. Then, $f$ is nilpotent of class at most $d(r+1)$.
\end{thm}

\begin{proof}
Let $G=\Inn(P)$ and $N=\relInn(f)$.
For a fixed $j$, put $\alpha_i=\alpha_j(s_{x_i})$ for $1\leq i\leq r$. From the assumption, $\alpha_j(G)$ is an abelian group generated by $\alpha_1,\ldots,\alpha_r$.

Let $R_j$ be the subring of $\End(A_j)$ generated by $\alpha_1^{\pm1},\ldots,\alpha_r^{\pm1}$, which is in fact a commutative ring by the assumption. Put $u_i=\alpha_i-\id$ and $I_j=(u_1,\ldots,u_r)\subseteq R_j$. By \cref{proposition: commutator interpretation of reducedness,lemma: quadratic action on abelian sections}, we have $u_i^2=0$. Since the $u_i$ commute, every product of $r+1$ of them contains a repeated factor. Hence we obtain $I_j^{r+1}=0$.

Since $\alpha_i^{-1}=\id-u_i$, $\alpha_i^{\pm1} \equiv\id \pmod{I_j}$. Since $\alpha_j(G)$ is generated by $\alpha_1,\ldots,\alpha_r$, every image $\alpha_j(g)$ is congruent to $\id$ modulo $I_j$, and hence the action difference $\id - \alpha_j(g)$ belongs to $I_j$.
Because taking the commutator of an element in $A_j$ with an element in $G$ is equivalent to multiplying it by an element in $I_j$ of the form $\id - \alpha_j(g)$, the $(r+1)$-fold commutator of $A_j$ with $G$ is trivial. In the language of subgroups, this implies
\[  [N_j, \underbrace{G, \dots, G}_{r+1}] \subseteq N_{j-1}.\]
Applying this relation successively along the series $1 = N_0 \subseteq N_1 \subseteq \cdots \subseteq N_d = N$ of length $d$, we deduce that $N$ is annihilated by $d(r+1)$ commutators with $G$. This immediately yields $N\subseteq Z_{d(r+1)}(G)$, and hence $\relInn(f)$ is nilpotent of class at most $d(r+1)$ by \cref{theorem: characterizations of nilpotent groups}. 
\end{proof}

\begin{cor}
\label[cor]{corollary: easy consequences of comm action criterion}
    Let $f\colon P\twoheadrightarrow Q$ be an operator-reduced quandle homomorphism. 
    Suppose that $P$ satisfies either of the two conditions in \cref{lemma: abelian images of inner automorphism groups}, and let $r$ denote the number of symmetry maps needed there to generate the corresponding abelian image.
    If $\relInn(f)$ is nilpotent of class at most $d$ and $\Inn(Q)$ is abelian, then $f$ is $d(r+1)$-nilpotent.
    %In particular, if both $\relInn(f)$ and $\Inn(Q)$ are abelian, then $f$ is $(r+1)$-nilpotent.
\end{cor}

\begin{proof}
    Put $G=\Inn(P)$ and $N=\relInn(f)$. Consider the upper central series of \(N\):
    \[ 1=Z_0(N)\leq Z_1(N)\leq\cdots\leq Z_d(N)=N. \]
    Since each \(Z_j(N)\) is characteristic in \(N\) and \(N\trianglelefteq G\), this is a \(G\)-normal series. Set $ A_j=Z_j(N)/Z_{j-1}(N)$ for $1\leq j\leq d$. Each \(A_j\) is abelian. Moreover, by the definition of the upper central series, $[N,Z_j(N)]\subseteq Z_{j-1}(N)$, which means that \(N\) acts trivially on \(A_j\) by conjugation. It follows that the action of \(G\) on \(A_j\) factors through $G/N\cong \Inn(Q)$.
    Since \(\Inn(Q)\) is abelian, the image of \(G\) in \(\Aut(A_j)\) is abelian. Thus all the hypotheses of the commutative action criterion are satisfied. Consequently, \cref{theorem: commutative action criterion} shows that $f$ is $d(r+1)$-nilpotent.
\end{proof}

If $P$ is generated by $r$ elements as a quandle, then \cref{lemma: quandle generators generate inner group} shows that $\Inn(P)$ is generated by the corresponding $r$ symmetry maps.

Finally, we study the cyclic case.

\begin{prop}
\label[prop]{proposition: operator reduced homomorphisms with cyclic relative inner group}
Let $f\colon P\twoheadrightarrow Q$ be a surjective quandle homomorphism such that $\relInn(f)$ is cyclic. Then, $f$ is operator-reduced if and
only if it is $2$-nilpotent.
Under these equivalent conditions, $f$ is $1$-nilpotent if $\relInn(f)$ is infinite or has square-free order.
\end{prop}

\begin{proof}
The implication from $2$-nilpotency to operator-reducedness follows from \cref{proposition: commutator interpretation of reducedness}. Suppose that $f$ is operator-reduced.
Let $G=\Inn(P)$.
Put $N=\relInn(f)$, write it additively, and for $x\in P$, let $\alpha_x\in\Aut(N)$ be induced by conjugation by $s_x$. Put $u_x\coloneqq\alpha_x-\id_N$. Since $[[a,s_x],s_x]=1$ for every $a\in N$, \cref{lemma: quadratic action on abelian sections} gives $u_x^2=0$.

If $N$ is infinite cyclic, then $\End(N)\cong\ZZ$, so $u_x=0$ for every $x\in P$. Since the symmetry maps generate $G$, we have $[G,N]=1$, and hence $f$ is $1$-nilpotent.

Suppose that $N\cong\ZZ/n\ZZ$ for some $n\geq1$, and write $\textstyle n=\prod_{p\mid n}p^{e_p}$. Under $\End(N)\cong\ZZ/n\ZZ$, each $u_x$ is multiplication by an integer $a_x$. The equality $u_x^2=0$ means that $n\mid a_x^2$, and hence $p^{\lceil e_p/2\rceil}\mid a_x$ for every prime $p\mid n$. Put $\textstyle k\coloneqq\prod_{p\mid n}p^{\lceil e_p/2\rceil}$. Then, $k\mid a_x$ and $n\mid k^2$, so $u_x(N)\subseteq kN$ and $u_x(kN)=0$ for every $x\in P$. Thus, every symmetry map acts trivially on both $N/kN$ and $kN$. Since the symmetry maps generate $G$, it follows that $[G,N]\subseteq kN$ and $[G,kN]=1$. Therefore,
\[
    \relInnLCS{2}{f} = [G,[G,N]] \subseteq [G,kN] = 1.
\]
Hence, $f$ is $2$-nilpotent. If $n$ is square-free, then $k=n$ and $kN=0$, so $[G,N]=1$. Thus, $f$ is $1$-nilpotent in this case as well.
\end{proof}

\appendix
\crefalias{section}{appendix}
\crefalias{subsection}{appendix}

\renewcommand{\thesection}{\Alph{section}}
\renewcommand{\thesubsection}{\thesection.\arabic{subsection}}
\renewcommand{\theHsection}{\Alph{section}}
\renewcommand{\theHsubsection}{\theHsection.\arabic{subsection}}

\section{Relative nilpotency for groups}
\label{appendix: relative nilpotency for groups}

In this appendix, we collect the relative theory of nilpotency for a normal subgroup $N$ of a group $G$, or equivalently for the quotient epimorphism $G\twoheadrightarrow G/N$.
Although the underlying notions and basic characterizations are classical, we include the proofs in order to fix our conventions and to record statements needed in this paper.
For general background on nilpotent groups, central series, and chief series, see \cite{book_robinson_1982_a_course_in_the_theory_of_groups}.

\begin{HisNote}
The material in this section consists mostly of standard facts that appear in various forms throughout the literature. 
To the best of the authors' knowledge, the earliest explicit formulation we have found of nilpotency for a pair of groups in terms of a relative central series is due to Hilton \cite{hilton_1976_nilpotency_in_group_theory_and_topology,hilton_1982_nilpotency_in_group_theory_and_topology}.

Earlier relative central-series constructions occur in Baer's study of upper central series and hypercentral normal subgroups (\cite{baer_1953_the_hypercenter_of_a_group}). In a different, action-theoretic setting, Kaloujnine~\cite{kaloujnine_1953_uber_gewisse_beziehungen_zwischen_einer_gruppe_und_ihren_automorphismen} introduced $A$-central series and $A$-nilpotency for groups equipped with a group $A$ of automorphisms; in this line of work, Mohamed~\cite{mohamed_1963_on_series_of_subgroups_related_to_groups_of_automorphisms} treated both lower and upper $A$-central series, and Bechtell~\cite{bechtell_1966_frattini_subgroups_and_phi_central_groups} applied Kaloujnine's construction to normal subgroups.
If $N\trianglelefteq G$ and $A$ is the image of the conjugation homomorphism $G\to\Aut(N)$, these operator-central series and the notion of nilpotency specialize to the relative ones considered below.

A further generalization replaces the commutator word by a system of group words; the resulting generalized central series are called \emph{marginal series}. Hulse--Lennox~\cite{hulse_lennox_1976_marginal_series_in_groups} and Fung~\cite{fung_1977_some_theorems_of_hall_type} developed this theory in the 1970s.
For pairs of groups, corresponding relative marginal-series and $\mathscr{V}_G$-nilpotency theories were developed in \cite{moghaddam_salemkar_2002_some_properties_of_ultra_hall_and_schur_pairs,rismanchian_araskhan_2013_some_properties_on_the_baerinvariant_of_a_pair_of_groups_and_mathscrvgmarginal_series}, where $\mathscr{V}$ is a variety of groups. The ordinary relative nilpotency considered here is recovered when $\mathscr{V}=\Ab$. See also \cite{rismanchian_2021_group_extensions_and_marginal_series_of_pair_of_groups} for a later account of marginal-series theory for pairs of groups.

A later group-theoretic treatment specifically devoted to nilpotency of pairs of groups is given in \cite{hassanzadeh_pourmirzaei_kayvanfar_2013_on_the_nilpotency_of_a_pair_of_groups}.
\end{HisNote}

\subsection{Relative central series and nilpotent homomorphisms}
\label{subsection in appendix: relative nilpotency for groups}

\begin{dfn}
\label[dfn]{definition: relative central series}
Let $G$ be a group and $N\trianglelefteq G$ be a normal subgroup. A \emph{$G$-relative central series} for $N$ is a finite sequence of normal subgroups of $G$,
\[ 1=N_{0}\subseteq N_{1}\subseteq\cdots\subseteq N_{c}=N, \]
such that $[G, N_{i+1}]\subseteq N_{i}$ for every $0\leq i<c$, or equivalently, $ N_{i+1}/ N_{i}\subseteq Z(G/ N_{i})$.

The normal subgroup $N$ is called \emph{nilpotent relative to $G$}, or \emph{$G$-nilpotent}, if it admits such a series. We define its \emph{$G$-nilpotency class}, denoted by $\ncl_G(N)$, as the least possible length. For notational convenience, we set $\ncl_G(N)= \infty$ when \(N\) is not \(G\)-nilpotent. When $N=G$, we simply say that $G$ is \emph{nilpotent} and write $\ncl(G)\coloneqq\ncl_G(G)$.
\end{dfn}

\begin{dfn}
\label[dfn]{definition: relative lower and upper central series}
Let $G$ be a group and $N\trianglelefteq G$ be a normal subgroup.
\begin{enumerate}
    \item The \emph{$G$-relative lower central series} of $N$ is defined inductively by $\relLCS_{0}(G,N)\coloneqq N$ and $\relLCS_{i+1}(G,N)\coloneqq[G,\relLCS_{i}(G,N)]$:
    \[ \cdots \subseteq \relLCS_{i}(G,N) \subseteq \cdots\subseteq \relLCS_{1}(G,N)=[G,N] \subseteq \relLCS_{0}(G,N) =N. \]
    \item The \emph{$G$-relative upper central series} of $N$ is defined inductively by $Z_0(G, N)\coloneqq1$ and $Z_{i+1}(G, N)\coloneqq \{x\in N\mid[G,x]\subseteq Z_i(G, N)\}$:
    \[ Z_{0}(G, N)=1 \subseteq Z_{1}(G, N) =N\cap Z(G) \subseteq \cdots\subseteq Z_{i}(G, N) \subseteq \cdots. \]
\end{enumerate}
Note that we use an indexing convention that starts in degree $0$ in the definition of the lower central series.
When $N=G$, we write $\grpLCS_{i}(G)\coloneqq\relLCS_{i}(G,G)$ and $Z_i(G)\coloneqq Z_{i}(G,G)$; these are the usual lower and upper central series of $G$.
\end{dfn}

In fact, if the $G$-relative lower and upper central series have finite length, then they are $G$-relative central series.

\begin{lem}
\label[lem]{proposition: relative upper central series as intersection}
For every $i\geq0$, we have $Z_i(G, N)=N\cap Z_i(G)$.
\end{lem}

\begin{proof}
For $x\in N$, we have $[G,x]\subseteq N$ by normality. Therefore, if the equality holds in degree $i$, $x\in Z_{i+1}(G, N)$ if and only if $[G, x]\subseteq Z_i(G)$. This completes the proof.
\end{proof}

\begin{thm}
\label[thm]{theorem: characterizations of nilpotent groups}
Let $G$ be a group and $N\trianglelefteq G$ be a normal subgroup, and let $c\geq0$. The following conditions are equivalent.
\begin{enumerate}[label={(\roman*)}]
    \item\label{item:G-nilpotent_of_class_c} The subgroup $N$ is $G$-nilpotent of class at most $c$.
    \item\label{item:relLCS_stops_at_c} $\relLCS_{c}(G,N)=1$.
    \item\label{item:relUCS_stops_at_c} $Z_{c}(G, N)=N$.
    \item\label{item:N_in_contained_in_absUCS_at_c} $N\subseteq Z_c(G)$.
\end{enumerate}
In particular, $\ncl_G(N)=\min\{i\geq 0\mid \Gamma_i(G,N)=1\}=\min\{i\geq 0\mid Z_i(G,N)=N\}$.
\end{thm}

\begin{proof}
\ref{item:relLCS_stops_at_c} or \ref{item:relUCS_stops_at_c} $\Rightarrow$ \ref{item:G-nilpotent_of_class_c}: immediate.
\ref{item:G-nilpotent_of_class_c} $\Rightarrow$ \ref{item:relLCS_stops_at_c}, \ref{item:relUCS_stops_at_c}: let $1=N_0\subseteq\cdots\subseteq N_c=N$ be any $G$-relative central series for $N$. Then, observe that induction gives $\relLCS_i(G,N)\subseteq N_{c-i}$ and $N_i\subseteq Z_i(G,N)$ for $0\leq i\leq c$. Thus we have $\relLCS_c(G,N)=1$ and $Z_c(G,N)=N$.
\ref{item:relUCS_stops_at_c} $\Leftrightarrow$ \ref{item:N_in_contained_in_absUCS_at_c}: follows from \cref{proposition: relative upper central series as intersection}.
\end{proof}

Taking $N=G$ in the preceding theorem gives the usual characterizations of nilpotent groups.

\begin{rem}
\label[rem]{remark: lower and upper operator central lengths}
For a general group with operators, the lower and upper operator-central series need not determine the same length, even when the relevant nilpotency conditions hold; see \cite[Example 4]{mohamed_1963_on_series_of_subgroups_related_to_groups_of_automorphisms}. The coincidence in \cref{theorem: characterizations of nilpotent groups} is a special feature of the conjugation action of an ambient group on a normal subgroup, as reflected in the identity $Z_i(G,N)=N\cap Z_i(G)$.
\end{rem}

\begin{eg}
    A normal subgroup $N$ of a group $G$ is nilpotent with $\ncl_G(N)=0$ if and only if $N$ is trivial; it is nilpotent with $\ncl_G(N)\leq1$ if and only if $N$ is centralized by $G$, which means $[G,N]=1$.
\end{eg}

\begin{dfn}
\label[dfn]{definition: nilpotent group homomorphism}
A surjective group homomorphism $\phi\colon G\twoheadrightarrow H$ is called \emph{nilpotent} if the kernel $\Ker(\phi)$ is $G$-nilpotent.
Its nilpotency class is denoted by $\ncl(\phi)\coloneqq \ncl_G(\Ker(\phi))$.
\end{dfn}

The characterization by iterated central extensions is now a direct relative analogue of the usual central-series description. Recall that a surjective group homomorphism $\pi\colon G\twoheadrightarrow H$ is a \emph{central extension} if $\Ker(\pi)\subseteq Z(G)$.

\begin{lem}
\label[lem]{lemma:central_series_correspond_to_central_extensions}
    Let $N\subseteq M$ be normal subgroups of $G$, and let $\pi\colon G/N \twoheadrightarrow G/M$ be the induced surjection. Then, the inclusion $N\subseteq M$ is $G$-central (i.e.\ $[G,M]\subseteq N$) if and only if $\pi$ is a central extension.
\end{lem}

\begin{proof}
    The condition $[G,M]\subseteq N$ is equivalent to $M/N\subseteq Z(G/N)$. Since $M/N=\Ker(\pi)$, this is equivalent to saying that $\pi$ is a central extension.
\end{proof}

\begin{thm}
\label[thm]{theorem: characterization of relative nilpotency}
Let $\phi\colon G\twoheadrightarrow H$ be a surjective group homomorphism and let $c\geq 0$. Then, $\phi$ is nilpotent of class at most $c$ if and only if $\phi$ admits a factorization as a composite of $c$ central extensions.
\end{thm}

\begin{proof}
($\Rightarrow$): Put $N\coloneqq\Ker(\phi)$. Let $1= N_{0}\subseteq\cdots\subseteq N_{c}=N$ be a $G$-relative central
series for $N$ (we may assume the length is $c$, by repeating terms if necessary). Then this sequence induces a factorization of $\phi$
\[\begin{tikzcd}
G=G/N_{0} \arrow[two heads]{r} &
G/N_{1} \arrow[two heads]{r} &
\cdots \arrow[two heads]{r} &
G/N_{c}\cong H
\end{tikzcd}\]
into $c$ central extensions by \cref{lemma:central_series_correspond_to_central_extensions}.
($\Leftarrow$): Conversely, from a factorization of $\phi$ into $c$ central extensions, take the kernels in $G$ of the successive composites. These kernels form a $G$-relative central series of length $c$, by \cref{lemma:central_series_correspond_to_central_extensions} again.
\end{proof}

From this theorem, we can also define a nilpotent homomorphism as a finite sequence of central extensions.
We will freely use this perspective in what follows, without further mention.

In particular, the lower and upper relative central series give two canonical factorizations of a nilpotent homomorphism. Let $\phi\colon G\twoheadrightarrow H$ be nilpotent of class at most $c$ and put $N\coloneqq\Ker(\phi)$. Then both sequences
\[\begin{tikzcd}
    G\cong G/\relLCS_{c}(G,N) \arrow[two heads]{r} &
    G/\relLCS_{c-1}(G,N) \arrow[two heads]{r} &
    \cdots \arrow[two heads]{r} &
    G/\relLCS_{0}(G,N)= H
\end{tikzcd}\]
and
\[\begin{tikzcd}
    G=G/Z_0(G, N) \arrow[two heads]{r} &
    G/ Z_{1}(G, N) \arrow[two heads]{r} &
    \cdots \arrow[two heads]{r} &
    G/Z_{c}(G, N)\cong H
\end{tikzcd}\]
are factorizations into central extensions.

\begin{eg}
    A surjective group homomorphism $\phi\colon G \twoheadrightarrow H$ is nilpotent with $\ncl(\phi)=0$ if and only if $\Ker(\phi)$ is trivial, and hence $\phi$ is an isomorphism; it is nilpotent with $\ncl(\phi)\leq1$ if and only if $\Ker(\phi)$ is contained in $Z(G)$, i.e., $\phi$ is a central extension.
\end{eg}

%For the terminal homomorphism $G\twoheadrightarrow1$, the preceding theorem says that $G$ is nilpotent of class at most $c$ if and only if $G\twoheadrightarrow1$ is a composite of $c$ central extensions. Thus all of the ordinary statements from the absolute theory are contained in the relative formulation.

The relative lower central series lies between the ordinary lower central series of the normal subgroup and that of the ambient group.

\begin{prop}
\label[prop]{proposition: comparison of lower central series}
    For a normal subgroup $N\subseteq G$, the inclusions $\Gamma_i(N) \subseteq \Gamma_i(G, N)\subseteq \Gamma_i(G)$ hold for all $i \geq 0$.
\end{prop}

\begin{proof}
When $i=0$, we have $\Gamma_0(N)=\Gamma_0(G,N)=N \subseteq G=\Gamma_0(G)$. If the inclusions hold for $i\geq 0$, then it follows that $\Gamma_{i+1}(N) = [N,\Gamma_i(N)] \subseteq [G,\Gamma_i(N)] \subseteq [G,\Gamma_i(G,N)]=\Gamma_{i+1}(G,N)$ and $\Gamma_{i+1}(G,N)=[G,\Gamma_i(G,N)] \subseteq [G,\Gamma_i(G)]=\Gamma_{i+1}(G)$.
\end{proof}

\begin{cor}
\label[cor]{corollary: relative nilpotency implies ordinary nilpotency}
    \begin{enumerate}
        \item If a normal subgroup $N\subseteq G$ is $G$-nilpotent of class at most $c$, then $N$ is nilpotent of class at most $c$.
        \item If $G$ is nilpotent of class at most $c$, then every normal subgroup $N$ of $G$ is $G$-nilpotent of class at most $c$.
    \end{enumerate}
\end{cor}

\begin{rem}
\label[rem]{remark: ordinary nilpotency does not imply relative nilpotency}
    Absolute nilpotency does not imply relative nilpotency.
    For example, consider the symmetric group $S_3$ and the alternating group $A_3\trianglelefteq S_3$ of degree $3$. Then $A_3$ is abelian, hence nilpotent. But $[S_3,A_3]=A_3$, so $\relLCS_{i}(S_3,A_3)=A_3$ for every $i\geq0$.
\end{rem}

\begin{lem}
\label[lem]{lemma: lower central series with cyclic quotient}
Let $N\trianglelefteq G$ be a normal subgroup such that $G/N$ is cyclic. Then, we have $\grpLCS_i(G)=\relLCS_i(G,N)$ for every $i\geq 1$. In particular, $N$ is $G$-nilpotent if and only if $G$ is nilpotent.
\end{lem}

\begin{proof}
First, we prove the case with $i=1$. The inclusion $[G,N] \subseteq [G,G]$ is trivial. The subgroup $N/[G,N]$ is central in $G/[G,N]$, and the corresponding quotient is isomorphic to $G/N$, which is cyclic. Then we claim that $G/[G,N]$ is abelian; indeed, putting $K=N/[G,N]$ and $H=G/[G,N]$, there is some $h \in H$ such that $H/K=\langle \overline{h}\rangle$. For $x,y \in H$, we can write $x=h^m k_1$ and $y=h^nk_2$, where $k_1,k_2\in K$ and $m,n\in \ZZ$. Since $k_1,k_2$ are in the center of $H$, we have $xy=h^mk_1h^n k_2 = h^{m+n}k_1k_2 = h^nk_2 h^mk_1 = yx$, and hence $H=G/[G,N]$ is abelian.
From this, it follows that $[G,G]\subseteq[G,N]$, and therefore $\grpLCS_1(G)=\relLCS_1(G,N)$. The equality in every degree
$i\geq1$ follows by induction.
%If $H$ is $G$-nilpotent of positive class, the equality of the two series shows that $G$ is nilpotent. If $H=1$, then $G$ is cyclic. The converse follows from the same equality.
\end{proof}

We record the following proposition, which will be used repeatedly in the main text.

\begin{prop}
\label[prop]{proposition: functoriality of relative lower central series}
Let $\phi\colon G\twoheadrightarrow H$ be a surjective homomorphism of groups. Let $N\trianglelefteq G$ be a normal subgroup, and put $M\coloneqq\phi(N)$, which is also normal by the surjectivity of $\phi$. Then, $\phi(\relLCS_i(G, N)) = \relLCS_i(H, M)$ for every $i\ge0$.
%In particular, $\ncl_H(M)\le\ncl_G(N)$.
\end{prop}

\begin{proof}
    We use the fact that $\phi([A,B])=[\phi(A),\phi(B)]$ in general. For $i=0$, we have $\phi(\relLCS_0(G,N))=\phi(N)=M=\relLCS_0(H,M)$. Suppose $\phi(\relLCS_i(G, N)) = \relLCS_i(H, M)$ for $i\geq 0$. Then, $\phi(\relLCS_{i+1}(G, N)) = \phi([G,\relLCS_i(G,N)]) =[H,\phi(\relLCS_i(G,N))]= [H,\relLCS_i(H,M)]=\relLCS_{i+1}(H, M)$. Thus, induction completes the proof.
\end{proof}

The lower series also gives the universal approximation by homomorphisms of bounded nilpotency class.

\begin{thm}
\label[thm]{theorem: universal relative nilpotent truncation}
Let $\phi\colon G\twoheadrightarrow H$ be a surjective group homomorphism and let $c\ge0$. With the abbreviation $\relLCS_c= \relLCS_c(G,\Ker(\phi))$, the induced homomorphism $\phi_c\colon G/\relLCS_c \twoheadrightarrow H$ is nilpotent of class at most $c$.

Moreover, the factorization $\phi=\phi_c\circ \pi$, where $\pi\colon G \to G/\relLCS_c$ is the quotient map, is universal among factorizations of $\phi$ whose second homomorphism has class at most $c$. That is, if $\phi=\theta\circ \xi$ with $\xi\colon G\twoheadrightarrow K$ surjective and $\theta\colon K\twoheadrightarrow H$ nilpotent of class at most $c$, then $\xi$ factors uniquely through $G/\relLCS_{c}$.
\end{thm}

\begin{proof}
One can easily check that $\pi(\Ker(\phi))=\Ker(\phi_c)$. By \cref{proposition: functoriality of relative lower central series}, we have $\relLCS_c(G/\relLCS_c,\Ker(\phi_c)) = \pi(\relLCS_c(G,\Ker(\phi)))=1$, which shows that $\phi_c$ is nilpotent of class at most $c$.
For the universal property, $\xi(\relLCS_{c}(G,\Ker(\phi))) =\relLCS_{c}(K,\Ker(\theta))=1$, so $\xi$ factors uniquely through the quotient.
\end{proof}

\begin{dfn}
\label[dfn]{definition: relative nilpotent truncation}
Let $c\geq 0$. For a surjective group homomorphism $\phi\colon G\twoheadrightarrow H$, we refer to the induced nilpotent homomorphism $\phi_c\colon G/\relLCS_{c}\twoheadrightarrow H$ as the \emph{$c$-nilpotent truncation} of $\phi$.
\end{dfn}

\subsection{Stable properties}
\label{subsection:stable_property_of_nilpotency}

\begin{prop}
\label[prop]{proposition:nilpotency_for_group_descend_along_surjection}
    Consider a commutative diagram of surjective homomorphisms
    \[\begin{tikzcd}
        G' \arrow[two heads]{r}{\theta} \arrow[two heads]{d}[swap]{\phi'} & G \arrow[two heads]{d}{\phi} \\
        H' \arrow[two heads]{r}[swap]{\kappa} & H\rlap{.}
    \end{tikzcd}\]
    with $\theta(\Ker(\phi'))=\Ker(\phi)$. If $\phi'$ is nilpotent of class at most $c$, then so is $\phi$.
\end{prop}

\begin{proof}
    It follows from \cref{proposition: functoriality of relative lower central series} that $\relLCS_c(G,\Ker(\phi)) = \theta(\relLCS_c(G',\Ker(\phi')))=1$, provided that $\relLCS_c(G',\Ker(\phi'))=1$.
\end{proof}

\begin{prop}
\label[prop]{proposition: relative lower central series under pullback}
Consider a pullback square of a surjective homomorphism $\phi$:
\[\begin{tikzcd}
    G' \arrow{r}{\theta} \arrow[two heads]{d}[swap]{\phi'} & G \arrow[two heads]{d}{\phi} \\
    H' \arrow{r}[swap]{\kappa} & H\arrow[phantom]{lu}[very near end]{\lrcorner} \rlap{.}
\end{tikzcd}\]
If $\phi$ is nilpotent of class at most $c$, then so is $\phi'$.
Moreover, if $\kappa$ is surjective, then the converse also holds; in particular, $\ncl(\phi)=\ncl(\phi')$ holds.
\end{prop}

\begin{proof}
Note that $\theta$ induces an isomorphism $\Ker(\phi')\cong \Ker(\phi)$ between the kernels.
By \cref{proposition: functoriality of relative lower central series}, it also induces $\relLCS_i(G',\Ker(\phi'))\cong \theta(\relLCS_i(G',\Ker(\phi'))) = \relLCS_i(\Image(\theta),\Ker(\phi))$ for every $i\geq 0$.
Since $\relLCS_i(\Image(\theta),\Ker(\phi)) \subseteq \relLCS_i(G,\Ker(\phi))$, if $\phi$ is nilpotent of class $c$, then so is $\phi'$.

The last assertion follows from \cref{proposition:nilpotency_for_group_descend_along_surjection}.
\end{proof}

\begin{lem}
\label[lem]{lemma: relative lower central series under subgroups}
Let $\phi\colon G\twoheadrightarrow H$ and $\phi'\colon G'\twoheadrightarrow H'$ be two surjective group homomorphisms. Suppose that we have inclusions $\iota_G\colon G'\hookrightarrow G$ and $\iota_H\colon H'\hookrightarrow H$ such that $\phi\circ \iota_G = \iota_H\circ \phi'$. If $\phi$ is nilpotent of class at most $c$, then so is $\phi'$.
\end{lem}

\begin{proof}
Put $N=\Ker(\phi)$ and $N'=\Ker(\phi')$. Then we have an inclusion $N' \subseteq N$. Regarding $G'\subseteq G$ as a subgroup via $\iota_G$, we obtain $\relLCS_{i}(G',N') \subseteq \relLCS_{i}(G,N)$ for every $i \geq 0$ inductively. Therefore, the assertion follows from it.
\end{proof}

\begin{prop}
\label[prop]{proposition: relative nilpotency under finite products}
Let $\phi\colon G\twoheadrightarrow H$ and $\phi'\colon G'\twoheadrightarrow H'$ be two surjective group homomorphisms. Then, the product $\phi\times \phi'$ is nilpotent of class at most $c$ if and only if $\phi$ and $\phi'$ are nilpotent of class at most $c$.
In particular, $\ncl(\phi\times \phi') = \max\{\ncl(\phi),\ncl(\phi')\}$ holds.
\end{prop}

\begin{proof}
We use the fact that $[A\times A',B\times B'] = [A,B]\times [A',B']$ as subsets of $G\times G'$ for subgroups $A,B\subseteq G$ and $A',B'\subseteq G'$.
Put $N=\Ker(\phi)$ and $N'=\Ker(\phi')$.
Then, by induction, we obtain $\relLCS_i(G\times G',N\times N') \cong \relLCS_i(G,N)\times \relLCS_i(G',N')$ for every $i\geq 0$.
Since $\Ker(\phi\times\phi')=N\times N'$, the assertion easily follows from it.
\end{proof}

\begin{prop}
\label[prop]{proposition: composition of nilpotent group homomorphisms}
Let $\phi\colon G\twoheadrightarrow H$ and $\psi\colon H \twoheadrightarrow K$ be surjective group homomorphisms.
\begin{enumerate}
    \item\label{item:composition_of_nilpotent_hom} If $\phi$ is nilpotent of class at most $c$ and $\psi$ is nilpotent of class at most $d$, then $\psi\circ \phi$ is nilpotent of class at most $c+d$.
    \item\label{item:cancelation_of_nilpotent_hom} If $\psi\circ\phi$ is nilpotent of class at most $c$, then both $\phi$ and $\psi$ are so.
\end{enumerate}
Consequently, whenever all three homomorphisms are nilpotent,
\[ \max\{\ncl(\phi),\ncl(\psi)\} \leq \ncl(\psi\circ\phi) \leq \ncl(\phi)+\ncl(\psi).\]
\end{prop}

\begin{proof}
\eqref{item:composition_of_nilpotent_hom}: When $\phi$ (resp.\ $\psi$) is a composite of $c$ (resp.\ $d$) central extensions, $\psi\circ\phi$ is a composite of $(c+d)$ central extensions. Hence, it follows from \cref{theorem: characterization of relative nilpotency}.

\eqref{item:cancelation_of_nilpotent_hom}: Put $\chi=\psi\circ\phi$ and assume that $\chi$ is nilpotent of class $c$. Then, $\Ker(\phi) \subseteq \Ker(\chi)$ gives $\relLCS_c(G,\Ker(\phi)) \subseteq \relLCS_c(G,\Ker(\chi))=1$, which shows that $\phi$ is nilpotent of class at most $c$.
Also, because $\phi(\Ker(\chi))=\Ker(\psi)$, \cref{proposition:nilpotency_for_group_descend_along_surjection} implies that $\psi$ is nilpotent of class at most $c$.
\end{proof}

We collect the closure properties for relatively nilpotent normal subgroups.

\begin{prop}
\label[prop]{proposition: products and intersections of relatively nilpotent subgroups}
    Let $N,M\trianglelefteq G$ be normal subgroups, and let $c, d\geq 0$. Then the following hold:
    \begin{enumerate}
        \item\label{item:subgroup_of_relative_nilpotent} If $N$ is $G$-nilpotent of class at most $c$ and $M\subseteq N$, then $M$ is $G$-nilpotent of class at most $c$.
        \item\label{item:intersection_of_relative_nilpotent} If $N$ and $M$ are $G$-nilpotent of class at most $c$ and $d$ respectively, then $N \cap M$ is $G$-nilpotent of class at most $\min\{c,d\}$. Hence, $\ncl_G(N\cap M)\le\min\{\ncl_G(N), \ncl_G(M)\}$.
        \item\label{item:product_subgroup_of_relative_nilpotent} If $N$ and $M$ are $G$-nilpotent of class at most $c$ and $d$ respectively, then $NM$ is $G$-nilpotent of class at most $\max\{c,d\}$. In particular, we have $\ncl_G(NM) = \max\{\ncl_G(N), \ncl_G(M)\}$.
    \end{enumerate}
\end{prop}

\begin{proof}
\eqref{item:subgroup_of_relative_nilpotent}: By \cref{theorem: characterizations of nilpotent groups}, $N$ is nilpotent of class at most $c$ if and only if $N\subseteq Z_c(G)$. Therefore, $M$ is nilpotent of class at most $c$ when $M \subseteq N$.
\eqref{item:intersection_of_relative_nilpotent}: It follows immediately from \eqref{item:subgroup_of_relative_nilpotent}.
\eqref{item:product_subgroup_of_relative_nilpotent}: Put $r=\max\{c,d\}$. Then $N,M \subseteq Z_r(G)$, which implies that $NM \subseteq Z_r(G)$. Hence $NM$ is $G$-nilpotent of class at most $r$. In particular, $\ncl_G(NM) \leq \max\{\ncl_G(N),\ncl_G(M)\}$. The reverse inequality follows from \eqref{item:subgroup_of_relative_nilpotent}, and hence equality holds.
\end{proof}

\bibliographystyle{amsalpha}
\bibliography{converted_bibtex_quandles_tyoshida}

\end{document}